\documentclass[10pt,leqno]{article}
\usepackage[T1]{fontenc}
\usepackage{lmodern} 
\usepackage[utf8]{inputenc}
\usepackage[a4paper]{geometry} %left=2cm,right=2cm,top=2.5cm,bottom=2.5cm,margin=1cm
\usepackage{indentfirst}
\usepackage{amsfonts,amsmath,bm,amssymb,amsthm,latexsym,mathtools,url,textcomp}
\usepackage[all]{xy}
\usepackage[normal]{caption}
\usepackage{color}
\usepackage[colorlinks=true,citecolor=cyan,linkcolor=magenta,urlcolor=blue]{hyperref}
\usepackage[title,titletoc]{appendix}

\newlength{\wideitemsep}		
	\let\olditem\item
\renewcommand{\item}{\setlength{\itemsep}{\wideitemsep}\olditem}

\newcommand{\cM}{{\mathcal M}}
\newcommand{\cC}{{\mathcal C}}

\newcommand{\cR}{{\mathcal R}}

\newcommand{\cO}{{\mathcal O}}

\newcommand{\cH}{{\mathcal H}}
\newcommand{\cI}{{\mathcal I}}
\newcommand{\cL}{{\mathcal L}}

\newcommand{\cD}{{\mathcal D}}
\newcommand{\cJ}{{\mathcal J}}

\def\fH{\mathfrak{H}}

\def\rap{\mathrm{a_p}}
\def\rbq{\mathrm{b_q}}

\def\rd{\mathrm{d}}
\def\ri{\mathrm{i}}
\def\re{\mathrm{e}}
\def\rloc{\mathrm{loc}}

\newcommand{\N}{{\mathbb N}}
\newcommand{\Z}{{\mathbb Z}}
\newcommand{\Q}{{\mathbb Q}}
\newcommand{\R}{{\mathbb R}}
\newcommand{\C}{{\mathbb C}}

\newcommand{\I}{{\mathbb I}}

\theoremstyle{plain}
\newtheorem{thm}{Theorem}[section]
\newtheorem{lem}[thm]{Lemma}
\newtheorem{cor}[thm]{Corollary}
\newtheorem{prop}[thm]{Proposition}

\theoremstyle{definition}
\newtheorem{defi}[thm]{Definition}

\newtheorem{rem}{Remark}
\newtheorem{rems}{Remarks}

\newcommand{\bi}{\begin{itemize}}
\newcommand{\ei}{\end{itemize}}
\newcommand{\bd}{\begin{description}}
\newcommand{\ed}{\end{description}}
\newcommand{\be}{\begin{enumerate}}
\newcommand{\ee}{\end{enumerate}}

\newcommand{\ps}{\prescript}

\newcommand{\dis}{\displaystyle}

\def\bc{\begin{center}}
\def\ec{\end{center}}

\def\Sum{\dis \sum}

\def\bf{\textbf}

\newcommand{\tfor}{\text{for}}
\newcommand{\tif}{\text{if}}
\newcommand{\tand}{\text{and}}

\newcommand{\tor}{\text{or}}

\newcommand{\lsum}{\sum\limits}
\newcommand{\lint}{\int\limits}
\newcommand{\lprod}{\prod\limits}
\newcommand{\llim}{\lim\limits}
\newcommand{\sbs}{\substack}

\def\no{\noindent}
\def\l{\left}
\def\r{\right}
\def\bl{\bigl}
\def\br{\bigr}
\def\bgl{\biggl}
\def\bgr{\biggr}

\def\s{\smallskip}
\numberwithin{equation}{section}

\begin{document}

\title{Fractional calculus and generalized\\
Mittag-Leffler type functions}
\author{Christian Lavault\thanks{LIPN, CNRS UMR 7030. \emph{E-mail:}\
\href{mailto:lavault@lipn.univ-paris13.fr}{lavault@lipn.univ-paris13.fr}}
}

\date{\empty}
\maketitle

\begin{abstract}
In this paper, the generalized fractional integral operators of two generalized Mittag-Leffler type functions are investigated. The special cases of interest involve the generalized Fox--Wright function and the generalized $M$-series and $K$-function. 

In the next Section~\ref{genfracint} we first recall some generalized fractional integral operators among the most widely used in fractional calculus. Section~\ref{M-K} is devoted to the definitions of 
$M$-series and $K$-function and their relations to special functions. In Sections~\ref{MKfc} and~\ref{fcMKF3}, effective fractional calculus of the generalized $M$-series and the $K$-function is carried out. The last section briefly concludes and opens up new perspectives.

The results established herein generalize recent properties of generalized Mittag-Leffler type functions using left- and right-sided generalized fractional differintegral operators. The note results also in important applications in physics and mathematical engineering.

\s \no {\bf Keywords:} Fox--Wright psi function; Generalized hypergeometric function;  $M$-series and $K$-function; Mittag-Leffler type functions; Riemann--Liouville's, Saigo's and Saigo--Maeda's generalized fractional calculus operators.

\s \no 2010 Mathematics Subject Classification: 26A33, 33C05, 33C10, 33C20, 33C60, 44A15.
\end{abstract}

\vskip 1cm
\begin{small}
\setcounter{tocdepth}{2}
\hypersetup{hidelinks}
\tableofcontents
\end{small}

\vskip 1.5cm
\pagenumbering{arabic}

\section{Introduction and motivations} \label{intro}
During the last two decades, the interest in Mittag-Leffler type functions has considerably developed. This is due to their vast potential of applications in applied sciences and engineering, and their steadily increasing importance in physics researches. More precisely, deviations of physical phenomena having an exponential behavior may be governed by physical laws (exponential and power laws) with the help of {\em generalized Mittag-Leffler type functions}. For example, they appear especially important in research domains such as stochastic systems theory, dynamical systems theory, statistical distribution theory, disordered and chaotic systems, etc., with special emphasis placed on applications to fractional differential equations---although this topic is not addressed herein. Furthermore, geometric properties including starlikeness, convexity and close-to-convexity for the Mittag-Leffler type functions were also recently investigated, e.g. by Bansal and Prajapat in~\cite{BanPra16}, Kilbas {\em et al.}~\cite{KiSrTr06} and Kiryakova~\cite{Kirya06,Kirya10a,Kirya15}.
This makes these functions directly and naturally amenable to fractional calculus techniques as studied by Gorenflo {\em et al.}~\cite{GoKiMaRo14}, Kilbas {\em et al.}~\cite{Kilbas05,KiSaSa04,KiSrTr06}, Kiryakova~\cite{Kirya94}, Kumar--Saxena~\cite{KumSax15}, Saigo~\cite{Saigo78,Saigo85,Saigo96}, Samko {\em et al.}~\cite{SamKilMar93}, Saxena--Saigo~\cite{SaxSai05}, Sharma~\cite{Sharma08,Sharma12}, Srivastava  {\em et al.} \cite{Srivast16,SrivastAgarwal13}, etc.

\subsection{The Mittag-Leffler and generalized Mittag-Leffler type functions}
The one-parametric {\em Mittag-Leffler function} (M-L for short) $E_\alpha(z)$ was first introduced by the swedish mathematician G. M. Mittag-Leffler in five notes~\cite[\color{cyan}1903]{Mittag03} \cite[\color{cyan}1905]{Mittag05} and also studied by Wiman~\cite[\color{cyan}1905]{Wiman05}. It is a special function of $z\in \C$ which depends on the complex parameter $\alpha$ and is defined by the power series 
\begin{equation} \label{def_mitt1}
E_\alpha(z) := \lsum_{n\ge 0} \frac{z^n}{\Gamma(\alpha n+1)}\ \qquad (\alpha\in \C).
\end{equation}
One can see that the series~\eqref{def_mitt1} converges in the whole complex plane for all $\Re(\alpha) > 0$. For all $\Re(\alpha) < 0$, it diverges everywhere on $\C\setminus \{0\}$ and, for $\Re(\alpha) = 0$, its radius of convergence is $R = \re^{\pi/2 |\Im(z)|}$.

A first generalization of $E_\alpha(z)$ introduced by Wiman~\cite[\color{cyan}1905]{Wiman05}, and later studied by Agarwal {\em et al.}~\cite[\color{cyan}1953]{Agar53,HumAgr53}, is the two-parametric M-L function of $z\in \C$, defined by the series
\begin{equation} \label{def_mitt2}
E_{\alpha,\beta}(z) := \lsum_{n\ge 0} \frac{z^n}{\Gamma(\alpha n+\beta)}\ \qquad %
(\alpha, \beta\in \C;\ \Re(\alpha) > 0\ \Re(\beta) > 0).
\end{equation}
In the case when $\alpha$ and $\beta$ are real positive, the series converges for all values of $z\in \C$, while when $\alpha,\, \beta\in \C$, the conditions of convergence closely follow the ones for 
$E_{\alpha,1}(z) := E_\alpha(z)$.

$E_\alpha(z)$ and $E_{\alpha,\beta}(z)$ are entire functions of $z\in \C$ of order $\rho = 1/\Re(\alpha)$ and type $\sigma = 1$; in a sense, they are the simplest two entire functions of this order (see, e.g., \cite[\color{cyan}\S3.1]{GoKiMaRo14}).\footnote{%
Following Levin~\cite[\color{cyan}Lect.~1]{Levin96}, every entire function $f(z)$ is represented by a power series $f(z) = \lsum_{n\ge 0} c_n z^n$ which converges everywhere in the complex plane; the radius of convergence is infinite, which implies that $\llim_{n\to \infty} \sqrt[n]{|c_n|} = 0$. Any  power series satisfying this criterion will represent an entire function. The global behaviour of an entire function of finite order is characterized by its {\em order} $\rho$ and its {\em type} $\sigma$ represented by the formulas
\[
\rho = \limsup_{n\to \infty} \frac{n\log n}{\log\l(1/|c_n|\r)}\ \qquad \tand\ \qquad 
\sigma = \frac{1}{\rho\re} \limsup_{n\to \infty} \l(n \sqrt[n]{|c_n|^\rho}\r).\]
Moreover, the asymptotic behaviour ($|z|\to \infty$) of an entire function is usually studied via its restriction to rays in an angle $\theta_1\le |\arg z|\le \theta_2$ (cf. Phargmén--Lindel{\"o}f). The so-called {\em indicator function} of an entire function of order $\rho$ is introduced as
\[
h(\theta) = \limsup_{r\to \infty} \frac{\log \bl|f(r\re^{\ri \theta})\br|}{r^\rho}\ %
\qquad (\theta_1\le \theta\le \theta_2).\]
Thus, the entire function $f(z) = \sum \l(\frac{\sigma \re^\rho}{n}\r)^{n/\rho} z^n$ has order $\rho$ and type $\sigma$. For instance, the M-L function $\lsum_{n=0}^\infty \frac{\l(A^\alpha z\r)^n}{\Gamma(\alpha n+1)}$, where $A > 0$ and $\alpha > 0$, is an entire function of order $\rho = 1/\alpha$ and type $\sigma = A$ (by Stirling formula).
}
The one- and two-parametric M-L function are fractional extensions of the basic functions $E_1(\pm z) := E_{1,1}(\pm z) = \re^{\pm z}$, $E_{1,2}(z) := (\re^z - 1)/z$, $E_2(z) := E_{2,1}(z) = \cosh(\sqrt{z})$, $E_{2,2}(z) = \sinh(\sqrt{z})/\sqrt{z}$, etc. (see, e.g., \cite[\color{cyan}\S7.1]{AnnMan12}, \cite{ErMaObTr53}, \cite[\color{cyan}Sec.1--4]{GoKiMaRo14}, \cite{GoKiRo98}, \cite{Kilbas05,KiSaSa04,Kirya94,Kirya06,Wiman05}, and references therein). They also appear as solutions of fractional integro-differential equations, e.g. in~\cite{Dzrba66,KiSrTr06,Kirya94,Kirya10a,Kirya15}.

Among the numerous generalizations of the M-L function let us point out the {\em standard Wright function}, defined by Wright in a bunch of papers (from 1933 to 1940)~\cite[\color{cyan}1934]{Wright34} by the series
\begin{equation} \label{wrightphi}
\phi(\alpha,\beta; z) := \lsum_{n\ge 0} \frac{1}{\Gamma(\alpha n + \beta)}\, \frac{z^n}{n!} %
= \ps{}{0}\psi_1^{}\l( \sbs{--\\[.1cm] (\beta, \alpha)}\, ; z\r)\ \qquad (z,\, \alpha,\, \beta\in \C),
\end{equation}
where $\ps{}{0}\psi_1^{}(-;(\beta, \alpha) ; z)$ is a special case of the {\em generalized Fox--Wright psi function}, which is given in Definition~\ref{def_foxwright}, Eq.~\eqref{eq_foxwright} of~\S\ref{htf}. The standard Wright function, which is very close to the two-parametric M-L function, may be rewritten also in terms of the {\em Fox $H$-function} introduced by Fox~\cite{Fox28} (see Definition~\ref{def_foxfunction} in~\S\ref{htf} and, e.g., \cite{AskDaal10,Kilbas05}).
%\footnote{% As far as the standard Wright function is concerned, it may be rewritten also in terms of the {\em Fox $H$-function} introduced by Fox~\cite{Fox28} (see Definition~\ref{def_foxfunction} in~\S\ref{htf} and, e.g., \cite{AskDaal10,Kilbas05}). The notation depends on the type and reads according to the values of the parameter~$\alpha$.
%\[
%\phi(\alpha,\beta;z) := %
%\begin{cases}
%H_{0,2}^{1,0}\l[-z\,\bgl|\,\sbs{ \hline\\[.1cm] (0,1), (1-\beta,\alpha)}\r] %
%&\ \ \tif\ \ \alpha > 0,\\[.3cm]
%H_{1,1}^{1,0}\l[-z\,\bgl|\,\sbs{ (\beta,-\alpha)\\[.1cm] (0,1)}\r] &\ \ \tif\ \ -1 < \alpha < 0.
%\end{cases}\]
%}
If $\alpha > -1$, this series is absolutely convergent for all $z\in \C$, while for $\alpha = -1$, it is absolutely convergent for $|z| < 1$, and for $|z| = 1$ and $\Re(\beta) > -1$~\cite[\color{cyan}Sec.~1.11]{KiSrTr06}.

For $\alpha > -1$, $\phi(\alpha,\beta;z)$ is an entire function of $z$, wherefrom one can deduce that for $\alpha > -1$, the function has order $\rho = 1/(\alpha+1)$ and type $\sigma = (\alpha+1) \alpha^{\frac{1}{\alpha+1}} = \alpha^\rho/\rho$.

Wright investigated the function $\phi(\delta,\nu+1;-z) := J_\nu^\delta(z)$, known as the {\em Bessel-Wright function} (or the Wright generalized Bessel function) and derived also the asymptotic behaviour of $\phi(\alpha,\beta;z)$ at infinity by means of its integral representation in terms of a {\em Mellin--Barnes contour integral} (see, e.g., \cite[\color{cyan}App.~E.ii]{Kirya94},\cite[\color{cyan}1.11]{KiSrTr06}). Additionally, for $z\in \C\setminus (-\infty,0]$ and\ $\nu\in \C$, the functions 
$\phi(1,\nu+1;\pm z^2/4)$ can  be expressed also in terms of the {\em Bessel function of the first kind}, 
\begin{flalign*}
J_\nu(z) &:= (z/2)^{\nu} \lsum_{k=0}^\infty (-1)^{k} \frac{(z^{2}/4)^{k}}{k! \Gamma(\nu+k+1)} %
= \frac{z/2}{\Gamma(\nu+1)}\, \ps{}{0}F_1^{}\bl(\nu + 1; z^2/4\br)\\
\shortintertext{%
and of the {\em modified Bessel function}
}
I_\nu(z) &:= (z/2)^{\nu} \lsum_{k=0}^\infty \frac{(z^{2}/4)^{k}}{k! \Gamma(\nu+k+1)} %
= \frac{z/2}{\Gamma(\nu+1)}\, \ps{}{0}F_1^{}\bl(\nu + 1; -z^2/4\br),\\
\shortintertext{%
as follows,
}
\phi(1,\nu+1;-z^2/4) &= 2z^{-\nu} J_\nu(z)\ \qquad \tand\ \qquad \phi(1,\nu+1;z^2/4)= 2z^{-\nu} I_\nu(z).
\end{flalign*}
Both Bessel functions are analytic functions of $z\in \C$, except for a branch point at $z = 0$ when $\nu$ is not an integer. The principal branches of $J_\nu(z)$ and $I_\nu(z)$ correspond to the principal value of $(z/2)^\nu$ and is analytic in the $z$-plane cut along the interval $(-\infty,0]$. When $\nu\in \Z$, they are entire in $z$. For fixed $z\ne 0$ each branch of the functions $J_\nu(z)$ and $I_\nu(z)$ is entire in $\nu$.\footnote{%
The one- and two-parametric M-L functions are generalized to the multiple M-L function $F_{\alpha,\beta}^{(\mu)}(z) = \lsum_{n=0}^\infty \frac{z^n}{\Gamma(\alpha n + \beta)^\mu}$ for real values 
$\alpha$, $\beta$ and $\mu > 0$ (see Appendix~\ref{asFml}). They are also tightly related to the M-L type function $E_{\alpha,1}^{(a)}(s;z)$ due to Barnes: $E_\alpha(z) = \llim_{s\to 0} E_{\alpha,1}^{(a)}(s;z)$\ and $E_{\alpha,\beta}(z) = \llim_{s\to 0} E_{\alpha,\beta}^{(a)}(s;z)$ (with complex parameters). Moreover,
\[
\llim_{\alpha\to 0} E_{\alpha,1}^{(a)}(s;z) = \frac{1}{\Gamma(\beta)} \Phi(z,s,a),\]
where the series $\Phi(z,s,a) := \lsum_{n\ge 0} \frac{z^n}{(n + a)^s}$ defines the {\em Lerch's transcendent zeta} function, which is analytic when $|z| < 1$\ for $a\in \C\setminus \Z_{\le 0}$ and 
$s\in \C$; the series also converges when $|z| = 1$, provided that $\Re(s) > 1$.\par
As special cases, $\Phi(z,s,a)$ contains the {\em Riemann zeta function} $\zeta(s) = \Phi(1,s,1)$, the {\em Hurwitz (or generalized) zeta} function $\zeta(s,a) = \Phi(1,s,a)$ and the {\em Lerch zeta} function $\ell(\xi,s,a) = \re^{2\pi \ri \xi} \Phi\l(\re^{2\pi \ri \xi},s,a\r)$ ($\xi\in \R,\; \Re(s) > 1$), but also the {\em Polylogarithmic function} Li$_s(z) = z \Phi(z,s,1)$ (for $s\in \C$, it is analytic when $|z| < 1$\ and the series converges when $|z| = 1$\ provided that $\Re(s) > 1$) and the {\em Lipschitz--Lerch} zeta function $L(\xi,s,a) := \Phi\l(\re^{2\pi \ri \xi},s,a\r)$ (for $a\in \C\setminus \Z_{\le 0}$, it is analytic when $\Re(s) > 0$\ for $\xi\in \R\setminus \Z$\ and when $\Re(s) > 1$\ for $\xi\in \Z$). (For other values of $z$, these functions are defined by analytic continuation.)
}

Prabhakar~\cite[\color{cyan}1971]{Prabha71} introduced a three-parametric generalization of 
$E_{\alpha,\beta}(z)$ defined in~\eqref{def_mitt2} as a kernel of certain fractional differential equations in terms of the series
\begin{equation} \label{def_Prab}
E_{\alpha,\beta}^{\gamma}(z) := \lsum_{n\ge 0} \frac{(\gamma)_n}{\Gamma(\alpha n+\beta)} %
\frac{z^n}{n!}\ \qquad (\alpha,\, \beta,\, ;\ \Re(\alpha) > 0,\, \Re(\beta) > 0),
\end{equation}
where $(\lambda)_n$ denotes the usual Pochhammmer symbol defined by the identity
\[
(\lambda)_n := \lambda(\lambda + 1)\cdots(\lambda + n - 1)\ \ \tif\ \ n = 1, 2, 3,\ldots,\ \quad (\lambda)_0 = 1\ \qquad (\lambda\neq 0).\]
In~\eqref{def_Prab}, no condition is imposed on $\lambda$, provided it is not zero; for example, $\lambda$ can be a negative integer, in which case the series is terminating into a polynomial. By contrast, whenever $(\lambda)_n$ is to be written in terms of a gamma function as
\[
(\lambda)_n = \frac{\Gamma(\lambda + n)}{\Gamma(\lambda)}\ \qquad (\Re(\lambda) > 0),\]
then the condition $\Re(\lambda) > 0$ must be fulfilled, or at least the constraint $\lambda\neq 0$, $-1$, $-2$,\ldots is required. Note that if $E_{\alpha,\beta}^{\gamma}(z)$ is represented as a Fox $H$-function, $\alpha$ is real and positive: in that case, such a condition becomes a requirement. 

Prabhakar's three-parametric M-L function is an entire function of $z\in \C$ of order $\rho = 1/\Re(\alpha)$ and type $\sigma = 1$. If $\gamma = 1$, then $E_{\alpha,\beta}^{1}(z) = E_{\alpha,\beta}(z)$ and, if $\gamma = \beta = 1$, then $E_{\alpha,1}^{1}(z) := E_\alpha(z)$.

Due to its integral representation (see, e.g., \cite[\color{cyan}Chap.~5, \S5.1.2]{GoKiMaRo14}), 
$E_{\alpha,\beta}^{\gamma}(z)$ is considered as a special case of Fox's $H$-function as well as of {\em Wright's generalized hypergeometric} $\ps{}{p}\psi_q^{}$, so-called {\em Fox--Wright psi function} of $z\in \C$. (see, e.g., \cite{AskDaal10,Kilbas05} and the definitions of these higher transcendental functions, including the {\em Meijer $G$-function}, in~\S\ref{htf}, Definitions~\ref{def_foxwright}, \ref{def_foxfunction} and~\ref{def_meijerG}).

It is straightforward to verify that
\begin{flalign}
E_{\alpha,\beta}(z) &= \ps{}{1}\psi_1^{}\l( \sbs{(1,1)\\[.1cm] (\beta, \alpha)}\, ; z\r) %
= H_{1,2}^{1,1}\l[ -z\, \bgl|\, \sbs{ (0,1)\\[.1cm] (0,1), (1-\beta,\alpha)}\r]\ %
\qquad \tand \label{eq_mittH1}\\
E_{\alpha,\beta}^{\gamma}(z) &= \frac{1}{\Gamma(\gamma)}\, \ps{}{1}\psi_1^{}\l( %
\sbs{(\gamma,1)\\[.1cm] (\beta,\alpha)}\, ; z\r) = \frac{1}{\Gamma(\gamma)}\, %
H_{1,2}^{1,1}\l[ -z\, \bgl|\, \sbs{ (1-\gamma,1)\\[.1cm] (0,1), (1-\beta,\alpha)}\r] \label{eq_mittH2}\\
&= \frac{1}{\Gamma(\gamma)}\, \ps{}{1}F_m^{}\l(\gamma; \frac{\beta}{m}, \frac{\beta+1}{m}, %
\ldots, \frac{\beta+m-1}{m}\, ; \frac{z}{m^m}\r)\ \qquad \tfor \ \ \alpha = m\in \N. \label{eq_mittH3}
\end{flalign}
In particular, when $\alpha = 1$ the $H$-function coincides with the Meijer $G$-function (see \S\ref{htf}, Eq.~\eqref{eq_FoxMeijer})
\begin{flalign} \label{eq_alpha=1}
E_{1,\beta}^{\gamma}(z) &= %
\frac{1}{\Gamma(\gamma)}\, \ps{}{1}\psi_1^{}\l( \sbs{ (\gamma, 1)\\[.1cm] (\beta, 1)}\, ; z\r) %
= \frac{1}{\Gamma(\gamma)}\, H_{1,2}^{1,1}\l[ -z\, \bgl|\, \sbs{ (1 - \gamma, 1)\\[.1cm] %
(0, 1), (1 - \beta, 1)}\r] = \frac{1}{\Gamma(\gamma)}\, G_{1,2}^{1,1}\l( \sbs{1 - \gamma\\[.1cm] %
0, 1 - \beta}\, \bgl|\, z\r)\nonumber\\
&= \frac{1}{\Gamma(\beta)}\, \ps{}{1}F_1^{}(\gamma;\beta; z) = \frac{1}{\Gamma(\beta)}\, %
M(\gamma;\beta; z),
\end{flalign}
where $M(\gamma;\beta;z)$ denotes Kummer's confluent hypergeometric function (sometimes denoted by 
$\Phi(\gamma;\beta;z)$). Similarly, if we set $\beta = 1$ in Eq.~\eqref{eq_mittH1}, we find that 
\begin{equation} \label{eq_beta=1}
E_{\alpha,1}(z) = \frac{1}{\Gamma(\gamma)}\, \ps{}{1}\psi_1^{}\l( \sbs{ (1,1)\\[.1cm] (1,\alpha)}\, %
; z\r) %
= H_{1,2}^{1,1}\l[ -z\, \bgl|\, \sbs{ (0, 1)\\[.1cm] (0, 1), (0, \alpha)}\r]\ \qquad (\alpha\in \C,\ %
\Re(\alpha) > 0)
\end{equation}

Further extensions of the M-L function to four and six parameters were defined by Salim~\cite{Salim09} and, associated with Weyl fractional integral and differential operators, by Salim and Faraj~\cite{SalFar12}, respectively, by the power series
\begin{equation}  \label{def_mitt3}
E_{\alpha,\beta}^{\gamma,\delta}(z) := \lsum_{n\ge 0} \frac{(\gamma)_n\, z^n} %
{\Gamma(\alpha n + \beta)\, (\delta)_n}\ \qquad \tand \qquad %
E_{\alpha,\beta,r}^{\gamma,\delta,s}(z) := \lsum_{n\ge 0} \frac{(\gamma)_{sn}\, z^n} %
{\Gamma(\alpha n +\beta)\, (\delta)_{rn}}\,.
\end{equation}
In the first case (with four parameters), $\alpha, \beta, \gamma, \delta\in \C$, $\min\bl(\Re(\alpha), \Re(\beta),\Re(\gamma),\Re(\delta) > 0\br)$. While in the second case (with six parameters), 
$\alpha, \beta, \gamma, \delta\in \C$, $\min\bl(\Re(\alpha),\Re(\beta),\Re(\gamma),\Re(\delta) > 0\br)$, with $r,\, s\in \R_+$ and $s\le \Re(\alpha) + r$. In the latter case, $(\gamma)_{ns}$ denotes an 
extended variant of the Pochhammer symbol, defined by $(\gamma)_{sn} := \Gamma(\gamma+sn)/\Gamma(\gamma)$, which reduces to $s^{sn}\lprod_{j=1}^s \l(\frac{\gamma+j-1}{s}\r)_n$\ when $s$ is a non-negative integer.\footnote{%
The Pochhammer symbol $(x)_{nk}$ must not be confused with the {\em generalized Pochhammer $k$-symbol} 
$(x)_{n,k}$ itself, which was introduced by Diaz and Pariguan in~\cite{DiaPar07} and is defined in terms of $\dis \Gamma_k(x) := \llim_{n\to \infty} \frac{n! (nk)^{x/k-1}}{(x)_{n,k}}$\ ($k > 0$) for $x\in \C\setminus k\Z_{<0}$ and $k\in \R$ by the relation
\[
(x)_{n,k} := x(x + k)(x + 2k)\cdots (x + (n-1)k) = \Gamma_k(x + nk)/\Gamma_k(x).\]
When $k = 1$, this reduces to the standard Pochhammer symbol and gamma function. 
%: $(x)_{n,1} := (x)_n$ and $\Gamma_1(x) := \Gamma(x)$.
}
Both functions are entire in the complex $z$-plane of order $\rho = 1/\Re(\alpha-\delta+1)$ and type 
$\sigma = \frac{1}{\rho} \l(\Re(\delta)^{\Re(\delta)}/\Re(\alpha)^{\Re(\alpha)}\r)^\rho$. 

\begin{rems} \label{rems1}
Among the M-L functions with three parameters are the {\em generalized (Kilbas--Saigo) M-L type} special functions. These functions were introduced in 1995 by Kilbas and Saigo~\cite{KilSai95} in the form
\begin{equation} \label{kilsaiml}
E_{(\alpha,m,\ell)}(z) = 1 + \lsum_{k=1}^\infty \lprod_{j=0}^{n-1} \frac{\Gamma(\alpha\lfloor jm+\ell\rfloor+1)}{\Gamma(\alpha\lfloor jm+\ell+1\rfloor+1)}\, z^k\ \qquad (z\in \C),
\end{equation}
where an empty product is assumed to be equal to 1 (`empty product convention'). The generalized M-L type function~\eqref{kilsaiml} is defined for real $\alpha,\, m\in \R$ and $\ell\in \C$ meeting the conditions $\alpha > 0$, $m > 0$\ and $\alpha(jm + \ell) + 1\neq -1,\, -2\ldots$ ($j = 1,\, 2,\ldots$). 

When $\alpha$, $m$\ and $\ell$\ are real numbers which fulfill the above conditions, then $E_{(\alpha,m,\ell)}(z)$ is an entire function of $z$ of order $\rho = 1/\alpha$ and type $\sigma = 1/m$. In particular, if 
$m = 1$, the function reduces to $E_{(\alpha,1,\ell)}(z) := \Gamma(\alpha \ell + 1) %
E_{(\alpha,\alpha\ell+1)}(z)$ (see, e.g., \cite[\S5.2]{GoKiMaRo14} for a detailed study of these generalized (Kilbas--Saigo) M-L type special functions).

A wider extended class of special functions of M-L type was further introduced and studied, e.g., by Kiryakova~\cite{Kirya94,Kirya06,Kirya10b}. This class, based on Luchko--Kilbas--Kiryakova's approach, consists of multi-parameters analogues of $E_{\alpha,\beta}(z)$ called {\em multi-indices M-L functions}, such that the parameters $\alpha := 1/\rho$ and $\beta := \mu$ are replaced by two sets of multi-parameters. 

The M-L functions with $2n$ parameters are defined for $\alpha_j\in \R$ ($\alpha_1^2 +\cdots + \alpha_n^2\neq 0$)\ and $\beta_j\in \C$ ($j = 1,\ldots, n\in \N$) by the series
\begin{equation} \label{multiparml}
E_{(\alpha,\beta)_n}(z) := \lsum_{k\ge 0} \frac{z^k}{\lprod_{j=1}^n \Gamma(\alpha_j k + \beta_j)}\ %
\qquad (z\in \C).
\end{equation}
Under the additional condition that $\Sigma_n = \alpha_1 +\cdots +\alpha_n > 0$, the generalized $2n$-parametric function $E_{(\alpha,\beta)_n}(z)$ is an entire function of $z\in \C$ of order $\rho = 1/\Sigma_n$ and type $\sigma = \lprod_{i=1}^n \l(\Sigma_i/|\alpha_i|\r)^{\alpha_i/\Sigma_i}$. When $n = 1$, the definition in~\eqref{multiparml} coincides with the definition of the two-parametric M-L function
\begin{subequations}
\begin{flalign}
E_{(\alpha,\beta)_1}(z) &:= E_{\alpha,\beta}(z) = \lsum_{k\ge 0}\frac{z^k}{\Gamma(\alpha k + \beta)}\ \qquad (z\in \C),\\
\intertext{%
and similarly for $n = 2$, where $E_{(\alpha,\beta)_n}(z)$ coincides with the four-parametric M-L function 
}
E_{(\alpha,\beta)_2}(z) &= E_{(\alpha_1,\beta_1;\alpha_2,\beta_2)}(z) = %
\lsum_{k\ge 0} \frac{z^k}{\Gamma(\alpha_1 k + \beta_1) \Gamma(\alpha_2 k + \beta_2)}\ \qquad (z\in \C).
\end{flalign}
\end{subequations}
The {\em generalized $2n$-parametric M-L function} $E_{(\alpha,\beta)_n}(z)$ can be represented in terms of the generalized Fox--Wright hypergeometric function $\ps{}{p}\psi_q^{}(z)$ by
\begin{subequations}
\begin{flalign} 
E_{(\alpha,\beta)_n}(z) &= \ps{}{1}\psi_n^{}\l( \sbs{ (1,1)\\[.1cm] %
(\beta_1,\alpha_1),\ldots,(\beta_n,\alpha_n) }\; ; z\r)\ \qquad (z\in \C) \label{multparrep1}\\
\intertext{%
and, via a Mellin--Barnes integral (see~\S\ref{intrepml} below), by
}
E_{(\alpha,\beta)_n}(z) &= \frac{1}{2\pi \ri}\, \lint_{\cL} \frac{\Gamma(s) \Gamma(1-s)} %
{\lprod_{j=1}^n \Gamma(\beta_j-\alpha_j s)}\, (-z)^{-s} \rd s\ \qquad (z\neq 0). \label{multparrep2}
\end{flalign}
\end{subequations}
For $\Re(\Sigma_n) > 0$, one can choose the left loop $\cL_{-\infty}$ as a contour of integration
in~\eqref{multparrep2}. Carrying out this integral by use of the theory of residues immediately
yields the series representation given in~\eqref{multiparml}. The extension of $E_{(\alpha,\beta)_n}(z)$ with the right loop $\cL_{+\infty}$ chosen as the contour of integration is entailed through the Mellin--Barnes integral representation in~\eqref{multparrep2}. This yields the convergence of the contour integral for all values of parameters $\alpha_1,\ldots, \alpha_n\in \C$, $\beta_1,\ldots, \beta_n\in \C$ such that $\Re(\Sigma_n) < 0$.
\end{rems}

\subsection{Integral representations of Mittag-Leffler type functions} \label{intrepml}
The integral representations of M-L type functions play a prominent role in the analysis of entire functions. As for any function of the M-L type, the M-L functions with one to three parameters can be
represented via {\em Mellin--Barnes integrals} obtained, for example, by a calculus of residues.
\begin{lem} \label{lem_mbrep}
Let $\alpha\in \R_+$, $\beta,\, \gamma\in \C$ and $\Re(\gamma) > 0$. Then the Mellin--Barnes integral representation of Prabhakar's three-parametric function is
\begin{equation} \label{eq_mbrep}
E_{\alpha,\beta}^{\gamma}(z) = \frac{1}{\Gamma(\gamma)}\, \frac{1}{2\pi \ri}\, \lint_{\cC} %
\frac{\Gamma(s) \Gamma(\gamma-s)}{\Gamma(\beta-\alpha s)}\, (-z)^{-s} \rd s\ %
\qquad (z\in \C,\ |\arg z| < \pi)
\end{equation}
The contour of integration $\cC = (c-\ri \infty, c+\ri \infty)$, $0 < c < \Re(\gamma)$, separates all the poles of $\Gamma(s)$ at the points $s = -\nu$ ($\nu\in \N$) to the left from those of 
$\Gamma(\gamma-s)$ at the points $s = \nu + \gamma$ ($\nu\in \N$) to the right.\footnote{%
The present path of integration $\cC$ is simpler than the ones considered, for example, in Eq.~\eqref{multparrep2} and in the case of a Hankel contour, in Eq.~\eqref{complexmlintrep} below (see also Appendix~\color{cyan}C\color{black}\ and Appendix~\color{cyan}D\color{black}), but however runs exactly along the same lines. 
} 
\end{lem}
\begin{proof}
The contour integral in~\eqref{eq_mbrep} is obtained by the sum of residues technique evaluated at the poles $s = 0, -1, -2,\ldots$. Hence,
\begin{flalign} \label{eq_mbintrep}
\frac{1}{2\pi \ri}\, \lint_{\cC} \frac{\Gamma(s) \Gamma(\gamma-s)}{\Gamma(\beta-\alpha s)}\, %
(-z)^{-s}\rd s &= \lsum_{k=0}^\infty \llim_{s\to -k} \bgl( \frac{(s+k) \Gamma(s) \Gamma(\gamma - s) %
(-z)^{-s}} {\Gamma(\beta-\alpha s)}\bgr)\nonumber\\
= \lsum_{k=0}^\infty \frac{(-1)^k}{k!}\, & \frac{\Gamma(\gamma + k)}{\Gamma(\beta + \alpha k)}\, 
(-z)^{k} = \Gamma(\gamma) \lsum_{k=0}^\infty \frac{(\gamma)_k} {\Gamma(\beta + \alpha k)}\, %
\frac{z^k}{k!} = \Gamma(\gamma) E_{\alpha,\beta}^{\gamma}(z). 
\end{flalign}
Eq.~\eqref{eq_mbintrep} is the expression of $E_{\alpha,\beta}^{\gamma}(z)$ as a Mellin--Barnes integral, and the lemma follows.
\end{proof}
Mellin--Barnes integral representations of the one- and two parametric M-L functions are immediately deduced from the one given in Lemma~\ref{lem_mbrep} by setting respectively $\gamma = \beta = 1$ and $\gamma = 1$. The following related representations of $E_{\alpha,1}^{1}(z) := E_\alpha(z)$ and $E_{\alpha,\beta}^{1}(z) := E_{\alpha,\beta}(z)$ write in terms of Mellin--Barnes integrals
\begin{subequations}
\begin{flalign*}
E_\alpha(z) &= \frac{1}{2\pi \ri}\, \lint_{\cC_1} \frac{\Gamma(s) \Gamma(1-s)} %
{\Gamma(1-\alpha s)}\, (-z)^{-s} \rd s\ \qquad (z\in \C,\ |\arg z| < \pi)\ \ \qquad \tand\\
E_{\alpha,\beta}(z) &= \frac{1}{2\pi \ri}\, \lint_{\cC_2} \frac{\Gamma(s) \Gamma(1-s)} %
{\Gamma(\beta-\alpha s)}\, (-z)^{-s} \rd s\ \qquad (z\in \C,\ |\arg z| < \pi),
\end{flalign*}
\end{subequations}
where each contour of integration $\cC_1$ and $\cC_2$ is again a straight line which starts at 
$c-\ri \infty$ and ends at $c+\ri \infty$ ($0 < c < 1$), leaving all poles of each integrand respectively separated to the left and to the right of each line contour.

\begin{rem}
M-L type functions play a basic role in the solution of fractional differential equations and integral equations of Abel type. Therefore, studying and developing their theory and stable methods is a first important step for their numerical computation. In this respect, integral representations of M-L type functions in terms of Mellin--Barnes integrals are sometimes not the easiest nor the more useful ones to handle. This is indeed the case of the two-parametric M-L function, for example, which enjoys various integral representations considered by D{\v{z}}rba{\v{s}}jan~\cite{Dzrba60,Dzrba66}, Erd{\'e}lyi {\em et al.}~\cite[\color{cyan}Vol~3, \S8.1]{ErMaObTr53} and Wright~\cite{Wright34} that are far more fruitful to the design of performing numerical algorithms (see Appendix~\color{cyan}A\color{black}\ for more detail on such integral representations).
\end{rem}
Also, on evaluating the residues at the poles of $\Gamma(1-s)$, there results the following analytic continuations of the one- and two-parametric M-L functions,
\begin{equation}
E_{\alpha}(z) = - \lsum_{n=1}^\infty \frac{z^{-n}}{\Gamma(1-\alpha n)}\ \quad \tand\ \quad %
E_{\alpha,\beta}(z) = - \lsum_{n=1}^\infty \frac{z^{-n}}{\Gamma(\alpha n + \beta)}.
\end{equation}
Many important propertie of $E_{\alpha}(z)$ and $E_{\alpha,\beta}(z)$ follow from the integral representation of the two-parametric M-L function given by D{\v{z}}rba{\v{s}}jan~\cite[\color{cyan}1952, 1960, 1966]{Dzrba52,Dzrba60,Dzrba66} and Erd{'e}lyi {\em et al.}~\cite[Vol.~I-III, 1953]{ErMaObTr53} (which reduces to $E_{\alpha}(z) := E_{\alpha,1}(z)$ when $\beta = 1$)
\begin{equation} \label{complexmlintrep}
E_{\alpha,\beta}(z) = \frac{1}{2\pi \ri} \lint_{\cH} \frac{t^{\alpha-\beta}\, \re^t} %
{t^{\alpha} - z}\, \rd t\ \qquad (z,\, \alpha,\, \beta\in \C,\ \Re(\alpha) > 0,\, \Re(\beta) > 0),
\end{equation}
where the contour of integration (or {\em Hankel path}) $\cH$ is the loop starting and ending at $-\infty$, and encircling the disk $|t|\le |z|^{1/\alpha}$ in the positive sense (counterclockwise): 
$|\arg(t)| < \pi$ on $\cH$. The integrand in~\eqref{complexmlintrep} has a branch point at $t = 0$, and the complex $t$-plane is cut along the negative real axis. The integrand is single-valued in the cut plane, where the principal branch of $t^{\alpha}$ is taken.\footnote{%
The representation~\eqref{complexmlintrep} can be proved by expanding the integrand in powers of $t$ and integrating term by term by making use of Hankel's integral path $\cH$ for the reciprocal of the gamma function, namely
\[
\frac{1}{\Gamma(\beta)} = \frac{1}{2\pi \ri} \lint_{\cH} \re^t\, t^{-\beta}\rd t\ \qquad (\beta\in \C).\]
Hankel's integral representation of $1/\Gamma(z)$ ($z\in \C$) is shown in Appendix~\color{cyan}C\color{black}\ (see also, e.g., Temme~\cite[\color{cyan}Chap.~3, \S3.2]{Temme96}).
}
(Detailed studies of M-L and M-L type functions may be found, e.g., in D{\v{z}}rba{\v{s}}jan~\cite{Dzrba66}, Erd\'elyi {\em et al.}~\cite{ErMaObTr53}, Gorenflo {\em et al.}~\cite[\color{cyan}Chap.~4, \S4.8]{GoKiMaRo14}, Kilbas {\em et al.}~\cite[\color{cyan}Chap.~1, \S1.8]{KiSrTr06}, etc.)

The integral representation~\eqref{complexmlintrep} of the two-parametric M-L function is used for instance to obtain the asymptotic expansion of $E_{\alpha,\beta}(z)$\ ($|z|\to \infty$) which is different in the case $0 < \alpha < 2$ from the case when $\alpha\ge 2$. According to the values of 
$|\arg z|$, there result the corresponding asymptotic expansions of $E_{\alpha,\beta}(z)$ and 
$E_{\alpha} := E_{\alpha,1}$ when $|z|$ is large (see Appendix~\color{cyan}A\color{black}).

The two-parametric M-L function defined in the form of Eq.~\eqref{def_mitt2} exists only for the values on parameters $\Re(\alpha) > 0$ and $\beta\in \C$. However, an analytic continuation of $E_{\alpha,\beta}(z)$ depending on real parameters $\alpha,\, \beta$ may be performed by extending its domain to negative values of $\alpha$. The integral representation~\eqref{2mlintrep} is similar to~\eqref{complexmlintrep} along the Hankel path $\cH$ (see, e.g., \cite{Dzrba60,GoKiMaRo14,KiSrTr06}),
\begin{equation} \label{2mlintrep}
E_{\alpha,\beta}(z) = \frac{1}{2\pi} \lint_{\cH} \frac{t^{\alpha-\beta}\, \re^t}{t^{\alpha} - z}\, %
\rd t\ \qquad (z\in \C,\ \alpha,\, \beta\in \R,\, \alpha < 0).
\end{equation}
The integral representation~\eqref{2mlintrep} makes it possible to define the function $E_{-\alpha,\beta}(z)$ for real negative values of the first parameter (see Appendix~\color{cyan}B\color{black}\ for a proof).

As another example, consider the Mellin--Barnes contour integral represention of Wright's function 
$\phi(\alpha,\beta;z)$ defined in~\eqref{wrightphi}. It may be given by the use of the sum of residues technique, as in Lemma~\ref{lem_mbrep},
\begin{equation} \label{phimlrep}
\phi(\alpha,\beta;z) %
= \frac{1}{2\pi \ri}\, \lint_\cC \frac{\Gamma(s)}{\Gamma(\beta-\alpha s)}\, (-z)^{-s} \rd s,
\end{equation}
where the path of integration $\cC$ separates all the poles at $s = -\nu$\ ($\nu\in \N$) to the left.
If $\cC =  (c-\ri \infty, c+\ri \infty)$ ($c\in \R$), then the representation~\eqref{phimlrep} is valid provided either of the following two conditions holds: 
\begin{flalign*}
(i)\ 0 < \alpha < 1,\, |\arg(-z)| < (1 - \alpha)\pi/2\ \quad & \tand\ \quad z\neq 0\ \qquad \tor\\
(ii)\ \alpha = 1,\, \Re(\beta) > 1 + 2c,\, \arg(-z) = 0\ \quad & \tand\ \quad z\neq 0.
\end{flalign*}
(See Kilbas {\em et al.} in~\cite[\color{cyan}Cor.~4.1]{KiSaSa04} for more detailed conditions on $\cC$.)

\begin{rem}
The Laplace transforms $\cL[\phi]$ and $\cL[E_{\alpha,\beta}]$ can be expressed in terms of each other:
\begin{subequations}
\begin{multline*}
\cL\bl(\phi(\alpha,\beta;t)\br)(s) = \lint_0^\infty \re^{-st} \phi(\alpha,\beta;t)\rd t %
= \frac{1}{s}\, E_{\alpha,\beta}\l(\frac{1}{s}\r)\ \qquad \tand\\
\cL\bl(E_{\alpha,\beta}(t)\br)(s) = \frac{1}{s}\, \phi\l(\alpha,\beta;\frac{1}{s}\r) %
= \frac{1}{s}\, \ps{}{2}\psi_1^{}\l( \sbs{ (1,1), (1,1)\\[.1cm](\beta, \alpha)}\, ; \frac{1}{s}\r) %
\quad (\alpha > -1,\ \alpha,\, \beta\in \C,\ \Re(s) > 0).
\end{multline*}
\end{subequations}
\end{rem}
Further, under the constraints $\Re(\alpha) > 0$, $\Re(\beta) > 0$, $\Re(\gamma) > 0$, 
$s > r^{1/\Re(\alpha)}$, the Laplace transform of Prabhakar's three-parametric M-L function is equal to
\begin{subequations}
\begin{flalign}
\cL\bl(E_{\alpha,\beta}^\gamma(t)\br)(s) &:= E_{\alpha,\beta}^\gamma\l(t^\alpha\r) %
= \frac{1}{s^\beta}\, \l(1-\frac{1}{s^\alpha}\r)^{-\gamma} = \frac{1}{s}\, \ps{}{2}\psi_1^{}%
\bgl( \sbs{ (\gamma, 1), (1, 1)\\[.1cm] (\beta, \alpha)}\, ; \frac{1}{s}\bgr) \label{laplace}\\
\intertext{%
and, by applying the Mellin inversion formula to the Mellin--Barnes representation of $E_{\alpha,\beta}(z)$ (with $\alpha\in \R_+$,\ $\beta, \gamma\in \C$, $\beta\neq 0$,\ $|\arg z| < \pi$. See Appendix~\color{cyan}A\color{black}), its Mellin transform is equal to
}
\cM\bl(E_{\alpha,\beta}^\gamma(-wt)\br)(s) &:= \lint_0^\infty t^{-s} %
E_{\alpha,\beta}^\gamma\l(t^\alpha\r)\rd t =  \frac{\Gamma(s) \Gamma(\gamma-s)} %
{\Gamma(\gamma) \Gamma(\beta-\alpha s)}\, w^{-s}\ \qquad (0 < \Re(s) < \Re(\gamma)). \label{mellin}
\end{flalign}
\end{subequations}

\subsection{Higher transcendental functions} \label{htf}
For any positive integers $p, q$, the generalized complex function $\ps{}{p}\psi_q^{}(z)$ has been defined in the seminal papers of Fox~\cite{Fox28} and Wright's papers of 1934--1940~\cite{Wright34} (see also, e.g., \cite{AskDaal10}). Later on, the so-called Fox--Wright (or Fox--Wright psi) function has been extensively studied by Gorenflo {\em et al.}~\cite[\color{cyan}Chaps.~3--6]{GoKiMaRo14}, Kilbas {\em et al.}~\cite{Kilbas05,KiSaSa04,KiSrTr06}, Kiryakova~\cite{Kirya94,Kirya06,Kirya10a,Kirya10b,Kirya15}, 

Here and in the sequel, the following synthetic vector notations are used for convenience whenever no ambigity may arise:
\[
\bm{\rap} := (a_1,\ldots,\, a_p),\ \qquad \Gamma(\bm{\rap}) %
:= \lprod_{i=1}^p \Gamma(a_i)\ \qquad \tand \ \qquad (a_i, \alpha_i)_1^p %
:= (a_1, \alpha_1),\ldots,\, (a_p, \alpha_p).\]

\begin{defi} \label{def_foxwright}
The {\em generalized Fox--Wright psi function} of $z\in \C$ is defined formally by the series
\begin{equation} \label{eq_foxwright}
\ps{}{p}\psi_q^{}(z) = %
\ps{}{p}\psi_q^{}\bgl( \sbs{ (a_i, \alpha_i)_1^p\\[.1cm] (b_j, \beta_j)_1^q}\, ; z\bgr) %
:= \lsum_{n\ge 0}\, \frac{\lprod_{i=1}^p \Gamma(a_i + \alpha_i n)} %
{\lprod_{j=1}^q \Gamma(b_j + \beta_j n)}\, \frac{z^n}{n!}\,,
\end{equation}
where $a_i, b_j\in \C$, $\alpha_i, \beta_j\in \R$ ($i = 1,\ldots, p$ and $j = 1,\ldots, q$).
\end{defi}
Wright proved several results on the asymptotic expansion of $\ps{}{p}\psi_q^{}(z)$ for all values of the argument $z$ which fullfill the property $\lsum_{j=1}^q \beta_j - \lsum_{i=1}^p \alpha_i > -1$. Under this  constraint, it was proved in~\cite{KiSaTr02} that the series is an entire function of $z\in \C$. Conditions of convergence are provided by the following lemma.

\begin{lem}[Kilbas--Srivastava--Trujillo~{\cite[Thm.~1]{KiSrTr06}}] \label{lem_foxwrightcvce}
Under the assumptions stated in Definition~\ref{def_foxwright}, let
\begin{equation} \label{eq_foxwrightcvce}
\Delta := \lsum_{j=1}^q \beta_j - \lsum_{i=1}^p \alpha_i,\ \ %
\delta := \lprod_{i=1}^p |\alpha_i|^{-\alpha_i}\, \lprod_{j=1}^q |\beta_j|^{\beta_j}\ \  %
\tand\ \ \mu := \lsum_{j=1}^q b_j - \lsum_{i=1}^p a_i + \frac{p - q}{2}.
\end{equation}
Then, the conditions of convergence are as follows
\bi
\item[(i)] If $\Delta > -1$, then the series in~\eqref{eq_foxwright} is absolutely convergent for all 
$z \in \C$.
\item[(ii)] If $\Delta = -1$, then the series in~\eqref{eq_foxwright} is absolutely convergent for 
$z\le \delta$, and for $z = \delta$ and $\Re(\mu) > 1/2$.
\ei
\end{lem}

When $a_i, b_j\in \R$ ($i = 1,\ldots, p$ and $j = 1,\ldots, q$), the Fox--Wright function 
$\ps{}{p}\psi_q^{}(z)$ has the following integral representation as a Mellin–Barnes contour integral
\begin{equation} \label{foxwrightmbint}
\ps{}{p}\psi_q^{}\bgl( \sbs{(a_i, \alpha_i)_1^p\\[.1cm] (b_j, \beta_j)_1^q}\, ; z\bgr) %
=  \frac{1}{2\pi \ri}\, \lint_{\cH} \frac{\lprod_{i=1}^p \Gamma(a_i + \alpha_i s)} %
{\lprod_{j=1}^q \Gamma(b_j + \beta_j s)}\, \Gamma(s)\ (-z)^{-s}\rd s,
\end{equation}
where the contour integration $\cH$ separates all the poles of $\Gamma(s)$ at $s = -\nu$ ($\nu\in \N$) to the left from all the poles of $\Gamma(a_i + \alpha_i s)$ at $s = (a_i + \nu_i)/\alpha_i$ 
($i = 1,\ldots, p$ and $\nu_i\in \N$ to the right. If $\cH =  (c-\ri \infty, c+\ri \infty)$ ($c\in \R$), then representation~\eqref{foxwrightmbint} is valid provided either of the
following two conditions hold:
\begin{flalign*}
(i)\ & \Delta < 1,\ |\arg(-z)| < (1 - \Delta)\pi/2,\ z\neq 0\ \qquad \tor\\
(ii)\ & \Delta = 1,\ \Re(\mu) > (\Delta + 1)c + 1/2\ \quad \tand \ \quad \arg(-z) = 0,\ z\neq 0.
\end{flalign*}
Conditions for the representation~\eqref{foxwrightmbint} are also given for the cases when 
$\cH = \cL_{-\infty}$ ($\cL_{+\infty}$, resp.) is a loop located in a horizontal strip starting at the point $-\infty+\ri \tau_1$ ($+\infty+\ri \tau_1$, resp.) and terminating at the point $-\infty+\ri \tau_2$ ($+\infty+\ri \tau_2$, resp.) with $-\infty < \tau_1 < \tau_2 < +\infty$.

In turns, the normalized Fox--Wright psi function is defined by
\begin{equation} \label{eq_psi*}
\ps{}{p}\psi_q^{*}\bgl( \sbs{(a_i, \alpha_i)_1^p\\[.1cm] (b_j, \beta_j)_1^q}\, ; z\bgr) := %
\frac{\Gamma(\bm{\rbq})}{\Gamma(\bm{\rap})}\, %
\ps{}{p}\psi_q^{}\l( \sbs{(a_i, \alpha_i)_1^p\\[.1cm] (b_j, \beta_j)_1^q}\, ; z\r),
\end{equation}
and trivially occurs as a generalization of the generalized hypergeometric function $\ps{}{p}F_q(z)$ (see also, e.g., Eq.~\eqref{FW}). Whenever $p = q = 1$, the normalization reduces to
\[
\ps{}{1}\psi_1^{*}\bgl( \sbs{(a, \alpha)\\[.1cm] (b, \beta)}\, ; z\bgr) := %
\frac{\Gamma(b)}{\Gamma(a)}\, \lsum_{n=0}^\infty \frac{\Gamma(a + \alpha n)}{\Gamma(b + \beta n)}\, %
\frac{z^n}{n!}\,.
\]
The importance of $\ps{}{1}\psi_1^{*}$ appears in Fermat's last theorem and in applied problems, such as the solution of trinomial equations.

Now, if we put $\alpha_i = \beta_j = 1$ for all $1\le i\le p$ and $1\le j\le q$ in Definition~\ref{def_foxwright}, then we obtain the following relation between the generalized Fox--Wright function 
$\ps{}{p}\psi_q^{}$ and the generalized hypergeometric function $\ps{}{p}F_q^{}$,\footnote{%
From Eq.~\eqref{eq_psiF}, the conditions of convergence of $\ps{}{p}F_q^{}(z)$ are easily recovered. When $p\le q$ and $\Delta\ge 0$, $\ps{}{p}F_q^{}(z)$ is an entire function of $z\in \C$. But if $p=q+1$, it is absolutely convergent in the unit disk $U = \lbrace z\,|\, |z| < 1\rbrace$ ($\delta = 1$) and diverges for all $z\neq 0$ if $p > q+1$. In the case when the variable $z$ in $\ps{}{q+1}F_q^{}(z)$ is not equal to unity, it is assumed tacitly that $|z| < 1$, while if $z = 1$, the condition 
$\Re\l(\lsum_{j=1}^q \beta_j - \lsum_{i=1}^{q+1} \alpha_i\r) > 0$\ is required (see~\cite[\color{cyan}\S4.1]{ErMaObTr53}).
}
\begin{equation} \label{eq_psiF}
\ps{}{p}\psi_q^{}\bgl( \sbs{(a_i, 1)_1^p\\[.1cm] (b_j, 1)_1^q}\, ; z\bgr) = %
\frac{\Gamma(\bm{\rap})}{\Gamma(\bm{\rbq})}\, \ps{}{p}F_q^{}\l( \sbs{\bm{\rap}\\[.1cm] %
\bm{\rbq}}\, ; z\r).
\end{equation}
The $\ps{}{p}\psi_q^{}$ and $\ps{}{p}F_q^{}$ functions are special cases of the more general special functions Fox $H$- and Meijer $G$-functions.

General conditions of existence of M-L type series defined in~\eqref{def_mitt1} to~\eqref{def_Prab} and their relations with generalized Fox--Wright-, Fox $H$- and Meijer $G$-functions (in Definitions~\ref{def_foxwright}, \ref{def_foxfunction} and \ref{def_meijerG}, resp.) are fully studied, for example, in~\cite{GoKiMaRo14,HaMaSa11,KiSrTr06,MaSaHa10}, along with their representations with suitable Mellin--Barnes contour integrals.

\begin{rem} \label{rem3}
The  generalized {\em Lommel--Wright function} denoted by $J_{\rho,\lambda}^{\mu,\nu}(z)$ is defined by
\begin{flalign}
J_{\rho,\lambda}^{\mu,\nu}(z) &:= \lsum_{n=0}^\infty \, \frac{(-1)^n} %
{\Gamma(\lambda+n+1)^\nu \Gamma(\rho+\lambda+\mu n +1)}\, \l(\frac{z}{2}\r)^{\rho+2\lambda+2n}\nonumber\\
&= \l(\frac{z}{2}\r)^{\rho+2\lambda}\, \ps{}{1}\psi_{\nu+1}^{}\bgl( \sbs{ (1, 1) %
\\[.1cm] (\lambda + 1, 1)^\nu, (\rho + \lambda + 1, \mu)}\, ; -z^2/4\bgr), \label{def_lommelwright}
\end{flalign}
where $\mu > 0$, $\rho,\, \lambda\in \C$ and $\nu\in \N$.
Turning to Wright's generalization of the classical Bessel function $J_\rho(z)$, that is the {\em Bessel--Wright function} $J_{\rho}^\mu(z)$, it is defined by 
\[
J_{\rho}^\mu(z) := \lsum_{n=0}^\infty \frac{(-z)^n}{n! \Gamma(\rho + \mu n + 1)}\, %
= \ps{}{0}\psi_{1}^{}(-; (\rho+1, \mu); -z).\]
Eq.~\eqref{def_lommelwright} readily yields the following relationships with the classical Bessel function, its generalization and a specialized Lommel--Wright function:  
$J_\rho(z)(z) = J_{\rho,0}^{1,1}(z),$\  $J_{\rho,\lambda}^\mu(z) = J_{\rho,\lambda}^{\mu,1}(z)$, as well as $J_{\rho}^{\mu}(z) = z^{-\rho/2} J_{\rho,\lambda}^{\mu,1}(2\sqrt{2})$.
\end{rem}

\begin{defi} \label{def_foxfunction}
The $H${\em -function} was introduced by Fox in~\cite{Fox28} as a generalized hypergeometric function defined by an integral representation in terms of the Mellin--Barnes contour integral
\begin{equation} \label{def_fox}
H_{p,q}^{m,n}\l[z\,\bgl|\, \sbs{(a_i, \alpha_i)_1^p\\[.1cm] (b_j, \beta_j)_1^q} \r] := %
\frac{1}{2\pi \ri} \lint_{\cL} \frac{\lprod_{j=1}^m \Gamma(b_j - \beta_j s) %
\lprod_{i=1}^n \Gamma(1 - a_i + \alpha_i s) } {\lprod_{j=m+1}^q \Gamma(1 - b_j + \beta_j s) %
\lprod_{i=n+1}^p \Gamma(a_i - \alpha_i s) }\ z^{s}\rd s.
\end{equation}
Here $\cL$ is a suitable contour in $\C$ and $z^s = \exp\bl(s\ln |z| + \ri \arg z\br)$. Due to the occurrence of the factor $z^s$ in the integrand, the $H$-function is, in general, multi-valued, but it can be made one-valued on the Riemann surface of $\log z$ by choosing a proper branch. In Eq.~\eqref{def_fox}, an empty product, when it occurs, is taken to be equal to 1. The order ($m, n, p, q$) consists in integers with $0\le m\le q$ and $0\le n\le p$. The parameters fulfills the conditions 
$a_i,\, b_j\in \C$ and $\alpha_i, \beta_j\in \R_+$  ($i = 1,\ldots, p$), ($j = 1,\ldots, q$). As for the conditions on the types of contours ensuring the existence and analyticity of the Fox $H$-function in disks $D\subset \C$, one can see \cite{AskDaal10,Fox28},\cite{KiSrTr06},\cite[\color{cyan}App.]{Kirya94},\cite[\color{cyan}Chap.~1]{SrivastManocha84},\cite{Wright34}, etc.
\end{defi}

Definition~\ref{def_foxfunction} is still valid when the $\alpha_i$'s and the $\beta_j$'s are positive rational numbers. Therefore, the $H$-function contains, as special cases, all of the functions which are expressible in terms of the $G$-function. More importantly, it contains the generalized hypergeometric Fox--Wright function defined in Definition~\ref{def_foxwright}, the generalized Bessel function $J_\nu^\mu$, the generalizations of the M-L function $E_{\alpha,\beta}$, etc. For example, 
$\ps{}{p}\psi_q^{}(z)$ is one of these (additional) special cases of the $H$-function, which is obviously not contained in the class of $G$-function. By Definition~\ref{def_foxwright}, it is also easily extended to the complex plane as follows,
\begin{equation} \label{eq_psiH}
\ps{}{p}\psi_q^{}\bgl( \sbs{(a_i, \alpha_i)_1^p\\[.1cm] (b_j, \beta_j)_1^q}\, ; z\bgr) = %
\lsum_{n=0}^\infty \frac{\lprod_{i=1}^p \Gamma(a_i + \alpha_i n)} %
{\lprod_{j=1}^q \Gamma(b_j + \beta_j n)}\, \frac{z^n}{n!} =  %
H_{p,q+1}^{1,p}\l[-z\,\bgl|\, \sbs{(1 - a_i, \alpha_i)_1^p\\[.1cm] (0, 1), (1 - b_j, \alpha_j)_1^q} \r]\,.
\end{equation}
Proper conditions of convergence are shown from Definition~\ref{def_foxwright} in~\cite[\color{cyan}\S2.1]{AskDaal10}.

The special case for which the Fox $H$-function reduces to the Meijer $G$-function is when 
$\alpha_1 =\cdots = \alpha_p = \beta_1 =\cdots = \beta_q = \kappa$, $\kappa > 0$ (see same references as above). In that case,
\begin{equation} \label{eq_FoxMeijer}
H_{p,q}^{m,n}\l[z\,\bgl|\, \sbs{(a_i, \kappa)_1^p\\[.1cm] (b_j, \kappa)_1^q} \r] := %
\frac{1}{\kappa}\, G_{p,q}^{m,n}\l( \sbs{\bm{\rap}\\[.1cm] \bm{\rbq}}\, \bgl|\, z^{1/\kappa}\r),
\end{equation}
Additionally, when setting $\alpha_i = \beta_j = 1$ in Eq.\eqref{eq_psiH} (i.e. $\kappa = 1$ in Eq.~\eqref{eq_FoxMeijer}), the Fox $H$- and the Fox--Wright Psi functions turn readily into the Meijer $G$-function. A general definition of the complex $G$-function is given by the following line integral of Mellin--Barnes type and may be viewed as an inverse Mellin transform.

Meijer's G-function $G_{p,q}^{m,n}\bgl( \sbs{\bm{\rap}\\[.1cm] \bm{\rbq}}\, \bgl|\, z\bgr)$ was actually first introduced as a generalization of the Gau{\ss} hypergeometric function presented in the form of the series 
\begin{equation} \label{meigerGF}
G_{p,q}^{m,n}\l( \sbs{1 - \bm{\rap}\\[.1cm] 0, 1 - \bm{\rbq}}\, \bgl|\, -z\r) %
= \frac{\Gamma(\bm{\rap})}{\Gamma(\bm{\rbq})}\, \ps{}{p}F_q^{}\l(\bm{\rap}; \bm{\rbq}; z\r).
\end{equation}
Later this definition was replaced by the Mellin–Barnes representation of the $G$-function.

\begin{defi} \label{def_meijerG}
The {\em Meijer $G$-function} of the order $(m,n,p,q)$ is defined by an integral representation of Mellin--Barnes contour integral given by
\begin{equation} \label{eq_meijer}
G_{p,q}^{m,n}\bgl( \sbs{\bm{\rap}\\[.1cm] \bm{\rbq}}\, \bgl|\, z\bgr) := %
\frac{1}{2\pi \ri} \lint_{\cL} \frac{ \lprod_{j=1}^m \Gamma(b_j - s) %
\lprod_{i=1}^n \Gamma(1 - a_i + s)} {\lprod_{j=m+1}^q \Gamma(1 - b_j + s) %
\lprod_{i=n+1}^p \Gamma(a_i - s) }\ z^s\rd s,
\end{equation}
where $\cL$ is a suitably chosen contour such that $z\neq 0$, $z^s = \exp\l(s\ln |z| + \ri \arg z\r)$ with a single valued branch of $\arg z$. Eq.~\eqref{eq_meijer} is valid under the following assumptions:
\bi
\item[(a)] The empty product is assumed to be equal to 1. 
\item[(b)] Parameters $(m,n,p,q)$ satisfy the relation $0\le m\le q$ and $0\le n\le p$, where $m$, $n$, $p$ and $q$ are non-negative integers.
\item[(c)] The complex numbers $a_i - b_j\neq 1,\, 2,\, 3,\ldots$ for $i = 1,\, 2,\, \ldots,\, n$ and 
$j = 1,\, 2,\, \ldots,\, m$, which implies that no pole of any $\Gamma(b_j - s)$, %
$j= 1,\, 2,\, \ldots,\, m$, coincides with any pole of $\Gamma(1 - a_i + s)$, %
$i = 1,\, 2,\, \ldots,\, n$.
\ei
\end{defi}

\subsubsection{Elementary identities of higher transcendental functions} \label{elids}
Among the most useful elementary properties of the Meijer $G$-function are two straightforward consequences of Definition~\ref{def_meijerG}, from which the first identity is readily derived,
\begin{flalign} \label{prop_meijer}
z^\mu G_{p,q}^{m,n}\bgl( \sbs{\bm{\rap}\\[.1cm] \bm{\rbq}}\, \bgl|\, z\bgr) &= %
G_{p,q}^{m,n}\bgl( \sbs{\bm{\rap} + \mu\\[.1cm] \bm{\rbq} + \mu}\, \bgl|\, z\bgr),
\intertext{%
while the second one exhibits the fact that we can, w.l.o.g., suppose that $p\le q$ in the discussion of the $G$-function,
}
G_{p,q}^{m,n}\bgl( \sbs{\bm{\rap}\\[.1cm] \bm{\rbq}}\, \bgl|\, 1/z\bgr) &= %
G_{q,p}^{n,m}\bgl( \sbs{1 - \bm{\rbq}\\[.1cm] 1 - \bm{\rap}}\, \bgl|\, z\bgr).
\end{flalign}
According to the general translation formulas of the Fox $H$- and Meijer $G$-functions (see~\cite{AskDaal10,KiSrTr06}), we also have  
\begin{subequations}
\begin{flalign}
H_{p,q}^{m,n}\l[z\,\bgl|\, \sbs{(a_i, \alpha_i)_1^p\\[.1cm] (b_j, \beta_j)_1^q} \r] &= %
H_{q,p}^{n,m}\l[1/z\,\bgl|\, \sbs{(1 - b_j, \beta_j)_1^q\\[.1cm] (1 - a_i, \alpha_i)_1^p} \r]\ %
\qquad \tand \label{eq_hypFoxH}\\
\ps{}{p}F_q\bgl(\sbs{\bm{\rap}\\[.1cm] \bm{\rbq}};z\bgr) &= \frac{\Gamma(\bm{\rbq})}{\Gamma(\bm{\rap})}\ %
G_{p,q+1}^{1,p}\l( \sbs{1-\bm{\rap}\\[.1cm] 0, 1-\bm{\rbq}}\, \bgl|\, -z\r) = \frac{\Gamma(\bm{\rbq})} %
{\Gamma(\bm{\rap})}\ G_{q+1,p}^{p,1}\bgl( \sbs{1, \bm{\rbq}\\[.1cm] \bm{\rap}}\, \bgl|\, -1/z\bgr).\label{eq_hypMeijer}
\end{flalign}
\end{subequations}
Unless there occurs a non-positive integer value for at least one of its upper parameters $\bm{\rap}$, the latter formula~\eqref{eq_hypMeijer} holds, in which case the hypergeometric function terminates in a finite polynomial. Then, the gamma prefactor of either $G$-function vanishes and the parameter sets of the $G$-functions violate the requirement $a_i - b_j\neq$ 1, 2, 3,\ldots for $i = 1, 2 \ldots, n$ and $j = 1, 2,\ldots m$ from Definition~\ref{def_meijerG} of $G_{p,q}^{m,n}$ (The conditions of convergence are exposed in details in~\cite{AskDaal10}.) Therefore, calculations on L-S and R-S fractional integration operators are amenable to Laplace and Mellin integral transforms for example, as shown in~\cite{Dzrba66},\cite[\color{cyan}App.]{Kirya06},\cite[\color{cyan}Chap.~2, \S7]{SamKilMar93}. 

The $G$- and $H$-functions encompass almost all the elementary and special functions. Among them, the 
M-L function and the Fox--Wright psi function. From Eq.~\eqref{prop_meijer}, $\ps{}{p}\psi_q^{}(z)$ can easily simplify to the generalized hypergeometric function $\ps{}{p}F_q(z)$ with parameters $a_i$'s and $b_j$'s and a gamma prefactor which consists of these parameters. According to the relationship between the $H$- and $G$-functions in Eq.~\eqref{eq_FoxMeijer} and to the expression of $\ps{}{p}\psi_q^{}(z)$ in terms of the $H$-function in Eq.~\eqref{eq_psiH}, the following useful identities hold

\begin{flalign} 
\ps{}{p}\psi_q^{}\bgl( \sbs{(a_i, 1)_1^p\\[.1cm] (b_j, 1)_1^q}\, ; z\bgr) &= %
H_{p,q+1}^{1,p} \l[-z\,\bgl|\, \sbs{(1 - a_i, 1)_1^p\\[.1cm] (0, 1), (1 - b_j, 1)_1^q} \r]\nonumber\\
&= G_{p,q+1}^{1,p}\bgl( \sbs{1 - \bm{\rap}\\ 0, 1 - \bm{\rbq}}\, \bgl|\, -z\bgr) %
= \frac{\Gamma(\bm{\rap})}{\Gamma(\bm{\rbq})}\ \ps{}{p}F_q\bgl( \sbs{\bm{\rap}\\[.1cm] %
\bm{\rbq}}\, ; z\bgr), \label{FW}
\end{flalign}
where $z,\, a_i,\, b_j\in \C$ ($i = 1,\ldots, p$; $j = 1,\ldots, q$) and %
$\Re\l(\lsum_{j=1}^q b_j - \lsum_{i=1}^p a_i\r) > 0$. Conversely, setting $a_i = b_j = 1$ in Eq.~\eqref{eq_psiH} yields the generalized hypergeometric function 
$\ps{}{p}F_q\l( \sbs{\alpha_1,\ldots, \alpha_p\\[.1cm] \beta_1,\ldots,\beta_q}\, ; z\r)$.

Finally, notice that the $H$- and $G$-functions are analytic functions of $z$ with a branch point at the origin. Especially, the kernel functions $H_{m,m}^{m,0}$ and $G_{m,m}^{m,0}$ ($n = 0$, $m = p = q$) of the operators of generalized fractional calculus that we consider, are analytic functions in the unit disc and vanish identically outside it, i.e. for $|z| > 1$.

\section{Fractional calculus of generalized M-L type functions} \label{genfracint}
Here and throughout the paper, the class of functions we deal with is a space of analytic functions in some domain $D\subseteq \C$. More precisely such functions are currently taken in 
\[
\fH_\mu(D) := \bl\{f(z) = z^\mu \tilde{f}(z) \,|\, \tilde{f}\in \fH(D)\br\},\] 
where $\mu\ge 0$ is a real number and $\fH(D)$ is the space of analytic (single-valued) functions in a complex domain $D$, starlike with respect to the origin $z = 0$.
For $\alpha\in \C$ and $\Re(\alpha) > 0$, composition relations between a very large number of generalized M-L type functions and the operators of fractional integration $\cI^{\alpha}$ and differentiation $\cD^{\alpha}$ are well established.

A fairly natural functional space where the {\em generalized fractional differintegration operators} may be considered is the space of {\em weakly differentiable functions} (i.e. Fréchet a.e.) in 
$L_{\rloc}^1(\Omega)$, where $L_{\rloc}^1(\Omega)$ denotes the set of locally integrable functions on a compact subset $\Omega$ of $\R$ (see, e.g., Bergounioux {\em et al.}~\cite{BeLeNaTo16}, Kiryakova~\cite[\color{cyan}Chap.~1]{Kirya94}, \cite{Kirya10a,Kirya10b}, Samko {\em et al.}~\cite[\color{cyan}\S1.1]{SamKilMar93}) and Srivastava~\cite{Srivast16}.

\subsection{Riemann--Liouville fractional calculus} \label{rlfraccal}
For a suitable class of weakly differentiable functions $\varphi\in L_{\rloc}^1\bl((a,b]\br)$ ($a < b\in \R$), it is convenient to set $a = 0$, without loss of generality (see definition below). Then, the related fractional integration operators are referred to by the notation $\cI_{0+}^{\alpha}$ and 
$\cI_{x-}^{\alpha}$, respectively and the {\em left-sided} and {\em right-sided} {\em Riemann–Liouville fractional integrals} of order $\alpha\in \C$ are defined by  
\begin{subequations}
\begin{flalign}
\cI_{0+}^{\alpha} f(x) &:=\frac{1}{\Gamma(\alpha)} \lint_0^x \frac{1}{(x-t)^{1-\alpha}}\, f(t)\, \rd t\ \qquad (\alpha\in \C,\, \Re(\alpha) > 0,\ x > 0), \label{eq_RLop1}\\
\cI_{x-}^{\alpha} f(x) &:= \frac{1}{\Gamma(\alpha)} \lint_x^\infty \frac{1}{(t-x)^{1-\alpha}}\, f(t)\, \rd t\ \qquad (\alpha\in \C,\, \Re(\alpha) > 0,\ x < 0). \label{eq_RLop2}
\end{flalign}
\end{subequations}
This notion is a natural consequence of the well known formula (Cauchy--Dirichlet?), that reduces the calculation of the $n$-fold primitive ($n\in \N$) of a real function $f(x)$ to a single integral of convolution type 
\[
I_{a+} f(x) = \frac{1}{(n-1)!} \lint_a^x (x - t)^{n-1} f(t)\, dt,\]
which equals 0 at $t = a$ with its derivatives of order $1, 2,\ldots, n-1$. The definition also requires $f(x)$ and $I_{a+}f(x)$ to be {\em causal functions}, that is, vanishing for $x < 0$.

As for the corresponding {\em Riemann--Liouville fractional differentiation operators}, $\cD_x^{\alpha}$, they are defined by
\begin{subequations}
\begin{flalign}
\l(\cD_{0+}^{\alpha} f\r)(x) &:= \l(\frac{\rd}{\rd x}\r)^{n}\, \l(\cI_{0+}^{n-\alpha} f\r)(x) %
= \frac{1}{\Gamma(n-\alpha)}\, \l(\frac{\rd}{\rd x}\r)^{n}\, %
\lint_0^x \frac{f(t)}{(x-t)^{1-n+\alpha}}\, \rd t\ \quad (x > 0), \label{eq_diffop1}\\
\l(\cD_{x-}^{\alpha} f\r)(x) &:=\l(\frac{\rd}{\rd x}\r)^{n}\,\l(\cI_{x-}^{n-\alpha} f\r)(x) %
= \frac{1}{\Gamma(n-\alpha)}\, \l(\frac{\rd}{\rd x}\r)^{n}\, %
\lint_x^\infty \frac{f(t)}{(t-x)^{1-n+\alpha}}\, \rd t\ \quad (x < 0), \label{eq_diffop2}
\end{flalign}
\end{subequations}
where $\lfloor \Re(\alpha)\rfloor$ denotes the integral part of $\Re(\alpha)$, 
$\l\{\alpha\r\} :=\alpha - \lfloor \Re(\alpha)\rfloor$ denotes its fractional part and 
$n := \lceil  \Re(\alpha)\rceil = \lfloor \Re(\alpha)\rfloor + 1$ (see, e.g., \cite{KiSaSa04} and \cite[\color{cyan}\S\S2.1--2.4]{SamKilMar93}).
All such well-defined Riemann--Liouville (R--L for short) fractional operators are called respectively the left-and right-sided fractional integrals and fractional derivatives of order $\alpha$. In a synthetic form, the R-L operators are denoted by $\cR_z^\alpha$ ($\alpha\in \C$) and written as
\[
\cR_{z}^{\alpha} f(z) :=
\begin{cases}
\frac{1}{\Gamma(\alpha)} \lint_0^z (z-t)^{\alpha-1} f(t)\rd t\ & \quad \tand \qquad %
\frac{1}{\Gamma(\alpha)} \lint_z^\infty (t-z)^{\alpha-1} f(t)\rd t \qquad  (\Re(\alpha) > 0),\\[.3cm]
\frac{\rd^n}{\rd z^n} \bl(\cR_z^{\alpha + n} f\br)(z) \qquad & (-n < \Re(\alpha)\leq 0, %
\quad n = \lfloor \Re(\alpha)\rfloor + 1\in \N),\\
\end{cases}
\]
provided that the defining integrals exists.

For example, it is easily seen that, for $Re(\mu) > -1$, $z > 0$,\footnote{%
From Eq.~\eqref{diffex1} we can check the intuitive idea that $\cD^{-1}$ is the integration operator. Examine indeed the operator $\cD^{-1}$ applied to a monomial.
\[
\cD_{0+}^{-1}(z^\mu) = \frac{\Gamma(\mu+1)}{\Gamma(\mu+2)}\, z^{\mu+1} = \frac{z^{\mu+1}}{\mu+1},\]
from which we see that effectively this is the integral operator. Now, by linearity of the operator we obtain the standard definition of the integral and in general we write
\[
\cD_{0+}^{-1}\bl(f(z)\br) = \lint_0^z f(t)\rd t.\]
}
\begin{flalign}
\cD_{0+}^{\alpha}(z^\mu) = \frac{\Gamma(\mu+1)}{\Gamma(\mu+1-\alpha)}\, z^{\mu-\alpha},\ %
& \qquad \tand\ \qquad \cD_{0+}^{\alpha}(z^{-\mu}) = \frac{\Gamma(\mu+\alpha)}{\Gamma(\mu)}\, %
z^{-\mu-\alpha} \label{diffex1}\\
\intertext{%
and, more importantly, that
}
E_{\alpha,\beta}^\gamma(z) &= \frac{1}{\Gamma(\gamma)}\, %
\cD_{0+}^{\alpha}\l(z^{\gamma-1} E_{\alpha,\beta}(z)\r)\ \qquad (\Re(\gamma) > 0). \label{diffex2}
\end{flalign}

\subsection{Saigo's fractional differintegration operators} \label{saigo}
Samko~{\em et al.}~\cite[\color{cyan}Sec.~2, \S2]{SamKilMar93} and Erd\'elyi--Kober extended the R--L fractional operators (see, e.g., Kilbas {\em et al.} in~\cite[\color{cyan}\S2.6]{KiSrTr06}) and references therein). A useful generalization of fractional integral operators was later introduced by Saigo~\cite{Saigo78,Saigo96} by adding Gau{\ss}'s hypergeometric function in the kernel. Saigo's operators include Weyl's type and Erd\'elyi--Kober's ones.

\begin{defi} [Saigo~{\cite[\color{cyan}p.~393]{SaiMae96}}] \label{def_saigo}
Let $\alpha,\beta,\gamma\in \C$ and $\Re(\alpha) > 0$. For a suitable class $\cC$ of complex functions specified below, Saigo's left- and right-sided operators of fractional integration are defined for $x\in \R^{+}$ by
\begin{subequations}
\begin{flalign} 
\l(\cI_{0+}^{\alpha,\beta,\gamma} f\r)(x) &:= \frac{x^{-\alpha-\beta}}{\Gamma(\alpha)}\, %
\lint_0^x (x-t)^{\alpha-1}\, \ps{}{2}F_1\l( \sbs{\alpha+\beta, -\gamma\\[.1cm] \alpha} ; 1 - t/x\r)\, %
f(t)\,\rd t\ \qquad \tand \label{eq_RLEKop1}\\
\l(\cI_{x-}^{\alpha,\beta,\gamma} f\r)(x) &:= \frac{x^{-\alpha-\beta}}{\Gamma(\alpha)}\, %
\lint_x^{\infty} (t-x)^{\alpha-1}\, t^{-\alpha-\beta}\, %
\ps{}{2}F_1\l( \sbs{\alpha+\beta, -\gamma\\[.1cm] \alpha} ; 1 - x/t\r)\, f(t)\,\rd t. \label{eq_RLEKop2}
\end{flalign}
\end{subequations}
Similarly, Saigo's left- and right-sided fractional differentiation operators $\cD_{0+}^{\alpha,\beta,\gamma}$ and $\cD_{x-}^{\alpha,\beta,\gamma}$ corresponding respectively to~\eqref{eq_RLEKop1} and~\eqref{eq_RLEKop2}, write as follows with $\Re(\alpha) > 0$ and $n := \lfloor \Re(\alpha)\rfloor + 1$,
\begin{subequations}
\begin{flalign} 
\l(\cD_{0+}^{\alpha,\beta,\gamma} f\r)(x) &:= \l(\cI_{0+}^{-\alpha,-\beta,-\gamma} f\r)(x) %
= \l(\frac{\rd}{\rd x}\r)^{n} \l(\cI_{0+}^{-\alpha,-\beta,-\gamma+n} f\r)(x)\ \qquad \tand \label{eq_Saigodiffop1}\\
\l(\cD_{x-}^{\alpha,\beta,\gamma} f\r)(x) &:= \l(\cI_{x-}^{-\alpha,-\beta,-\gamma} f\r)(x) %
= \l(\frac{\rd}{\rd x}\r)^{n} \l(\cI_{x-}^{-\alpha,-\beta,-\gamma+n} f\r)(x). \label{eq_Saigodiffop2}
\end{flalign}
\end{subequations}
\end{defi}
When $\beta = -\alpha$, the above fractional operators~\eqref{eq_RLEKop1} to \eqref{eq_Saigodiffop2} include the classical Riemann--Liouville's in~\eqref{eq_RLop1}--\eqref{eq_diffop2}:
\[
\cI_{0+}^{\alpha,-\alpha,\gamma} = \cI_{0+}^{\alpha},\quad \cI_{x-}^{\alpha,-\alpha,\gamma} %
= \cI_{x-}^{\alpha}\ \qquad \tand \qquad \ \cD_{0+}^{\alpha,-\alpha,\gamma} %
= \cD_{0+}^{\alpha},\quad \cD_{x-}^{\alpha,-\alpha,\gamma} = \cD_{x-}^{\alpha}.\]

Here, $\cC$ is the class of analytic functions $f(z)$ in a simply-connected region of the $z$-plane containing the origin (for example the unit disk), of the order $f(z) = \cO\l(|z|^\varepsilon\r)$ ($|z|\to 0$), where $\varepsilon > \max\bl\{0,\beta-\eta\br\}$ and the multiplicity of $(x-t)^{\alpha-1}$ ($(t-x)^{\alpha-1}$, resp.) is removed by requiring $\ln(x-t)$ ($\ln(t-x)$, resp.) to be real when 
$x > t$ ($x < t$, resp.). (Exhaustive conditions of existence and convergence of the r.h.s. of Eqs.~\eqref{eq_RLEKop1} and \eqref{eq_RLEKop2} may be found, e.g., in~\cite{AskDaal10,Kirya94,Kirya06}.) Saigo's left- and right-sided fractional integrals are consistently employed to prove the results in Sections~\ref{lsMK} and \ref{rsMK}.\footnote{%
As shown e.g. by Samko {\em et al.} in~\cite[\color{cyan}Chap.~4, \S22]{SamKilMar93}, the fractional integration process can also be carried out by use of the formula for differentiating the Cauchy type integral to the (real or complex) order $\alpha$:
\[
\cD_z^\alpha(z) = \frac{\Gamma(1+\alpha)}{2\pi\ri} \int_{\cL} \frac{f(t)}{(t-z)^{1+\alpha}}\, \rd t.\]
}

\begin{rem} \label{rem4}
For suitable functions $f$ of a real variable and $\alpha\in \R_+$, there holds the fundamental property of operators of fractional integration, namely the additive index law (or semigroup property), according to which
\[
\cJ_{a+}^\alpha\, \cJ_{a+}^\beta = \cJ_{a+}^{\alpha+\beta}\ \qquad \tand \qquad %
\cJ_{b-}^\beta\, \cJ_{b-}^\alpha = \cJ_{b-}^{\alpha+\beta}\ \qquad (\alpha, \beta\ge 0),\]
where, for complementation we define $\cJ_{a+}^{0} = \cJ_{b-}^{0} := \I$ (Identity operator), i.e., 
$\cJ_{a+}^\alpha f(z) = \cJ_{b-}^\alpha f(z) = f(z)$.  
It can be shown however that the fractional differentiation operator $\cD^\alpha$ is neither commutative nor additive in general. Indeed, denoting by $\cD^n$ the operator of ordinary differentiation of order 
$n\in \N$, then $\cD^n \cJ^n = \I$, whereas $\cJ^n \cD^n\neq \I$; i.e., $\cD^n$ is left-inverse (but not right-inverse) to the corresponding integral operator $\cJ^n$. Furthermore, by $\cI_t^{\alpha} f(0+)$ we mean $\llim_{t\to 0+} \cI_t^{\alpha} f(t)$ (if the limit exists) and this limit may be infinite.
\end{rem}

\subsection{Saigo--Maeda's fractional differintegration operators} \label{saigomaeda}
A natural extension of Saigo's generalized fractional differintegral operators stated in Definition~\ref{def_saigo} (Eqs.~\eqref{eq_RLEKop1}--\eqref{eq_diffop2}) was introduced further. This extension involves Appell's two variable hypergeometric series $F_3$ in the kernel (a generalization of Gau{\ss}'s hypergeometric function of one variable). Appell's two variable $F_3$ is defined by the double series
\begin{equation} \label{appellF3}
F_3(\alpha, \alpha', \beta, \beta', \gamma ; x, y) := \lsum_{m,n\ge 0} %
\frac{(\alpha)_m (\alpha')_n (\beta)_m (\beta)_n}{(\gamma)_{m+n}}\, \frac{x^m}{m!} \frac{y^n}{n!}\ %
\qquad (\max\bl(|x|, |y|\br) < 1).
\end{equation}
Several generalized fractional differintegral operators involving Appell's two variable hypergeometric function $F_3$ in the kernel have been addressed in~\cite{GoKiMaRo14,SamKilMar93}. For a suitable class of function such as $\cC$, {\em Saigo--Maeda's (or Marichev--Saigo--Maeda's) fractional differintegral operators} are introduced therefrom.

\begin{defi} [Saigo--Maeda~{\cite[\color{cyan}p.~393, Eqs.~(4.12)--(4.13)]{SaiMae96}}] \label{def_saigomaeda} Let $\alpha,\, \alpha',\, \beta,\, \beta',\, \gamma\in \C$ and $\Re(\gamma) > 0$. Then, for $x\in \R_+$,
\begin{subequations}
\begin{flalign} 
\l(\cI_{0+}^{\alpha,\alpha',\beta,\beta',\gamma} f\r)\!(x) &:= \frac{x^{-\alpha}}{\Gamma(\gamma)}\, %
\lint_0^x (x-t)^{\gamma-1} t^{-\alpha'}\, F_3\l(\alpha,\alpha',\beta,\beta',\gamma; 1 - t/x, 1 - x/t\r) %
f(t)\rd t \label{eq_intsaigomaeda1}\\
\shortintertext{%
and
}
\l(\cI_{x-}^{\alpha,\alpha',\beta,\beta',\gamma} f\r)\!(x) &:= \frac{x^{-\alpha'}} %
{\Gamma(\gamma)}\, \lint_x^\infty (t-x)^{\gamma-1} t^{-\alpha}\, F_3\l(\alpha,\alpha',\beta,\beta',\gamma\, ; 1 - x/t, 1 - t/x\r) f(t)\rd t. \label{eq_intsaigomaeda2}
\end{flalign}
\end{subequations}
Similarly, set $n = \lfloor \Re(\gamma)\rfloor + 1$, then
\begin{subequations}
\begin{flalign} 
\l(\cD_{0+}^{\alpha,\alpha',\beta,\beta',\gamma} f\r)\!(x) &:= %
\l(\cI_{0+}^{-\alpha',-\alpha,-\beta',-\beta,-\gamma} f\r)\!(x) \nonumber\\
& = \l(\frac{\rd}{\rd x}\r)^{n}\l(\cI_{0+}^{-\alpha',-\alpha,-\beta'+n,-\beta,-\gamma+n} f\r)\!(x) \label{eq_derivsaigomaeda1}\ \qquad \tand\\
\l(\cD_{x-}^{\alpha,\alpha',\beta,\beta',\gamma} f\r)\!(x) &:= %
\l(\cI_{x-}^{-\alpha',-\alpha,-\beta',-\beta,-\gamma} f\r)\!(x) \nonumber\\
& = (-1)^n\, \l(\frac{\rd}{\rd x}\r)^{n} \l(\cI_{x-}^{-\alpha',-\alpha,-\beta',-\beta+n,-\gamma+n} f\r)\!(x). \label{eq_derivsaigomaeda2}
\end{flalign}
\end{subequations}
\end{defi}
As a direct consequence of the above definitions, the following lemma holds.

\begin{lem} [Saigo--Maeda~{\cite[\color{cyan}p.~394, Eqs.~(4.18)--(4.19)]{SaiMae96}}] \label{lem_powfct}
Let $\alpha,\, \alpha',\, \beta,\, \beta',\, \gamma,\, \rho\in \C$. For $f(x) := x^{\rho-1}$, we get the formulas\par
(i)\ If $\Re(\gamma) > 0$ and $\Re(\rho) > \max\bl(0,\Re(\alpha+\alpha'+\beta-\gamma), %
\Re(\alpha'-\beta')\br)$, then
\begin{subequations}
\begin{flalign}
\l(\cI_{0+}^{\alpha,\alpha',\beta,\beta',\gamma} x^{\rho-1}\r)\!(x) &= %
\frac{\Gamma(\rho) \Gamma(\rho+\gamma-\alpha-\alpha'-\beta) \Gamma(\rho-\alpha'+\beta')} %
{\Gamma(\rho+\beta') \Gamma(\rho+\gamma-\alpha-\alpha') \Gamma(\rho+\gamma-\alpha'-\beta) }\, %
x^{\rho-\alpha-\alpha'+\gamma-1}. \label{eq_powfct1}\\
\intertext{%
\kern.5cm (ii)\ If $\Re(\gamma) > 0$ and $\Re(\rho) < 1 + \min\bl(0,\Re(-\beta),\Re(\alpha+\alpha'-\gamma), \Re(\alpha+\beta'-\gamma)\br)$, then
}
\l(\cI_{x-}^{\alpha,\alpha',\beta,\beta',\gamma} x^{\rho-1}\r)\!(x) &= %
\frac{\Gamma(1+\alpha+\alpha'-\gamma-\rho) \Gamma(1+\alpha+\beta'-\gamma-\rho) \Gamma(1-\beta-\rho)} %
{\Gamma(1-\rho) \Gamma(1+\alpha+\alpha'+\beta'-\gamma-\rho) \Gamma(1+\alpha-\beta-\rho)}\, %
x^{\rho-\alpha-\alpha'+\gamma-1}. \label{eq_powfct2}
\end{flalign}
\end{subequations}
\end{lem}

\begin{rems} \label{rems2}
Reducing Saigo--Maeda's fractional differintegral operators~\eqref{eq_intsaigomaeda1}--\eqref{eq_derivsaigomaeda2} to Saigo's~\eqref{eq_RLEKop1}--\eqref{eq_RLEKop2} is commonly done by means of the following relations, 
\begin{subequations}
\begin{flalign}
\cI_{0+}^{\alpha,0,\beta,\beta',\gamma} = \cI_{0+}^{\gamma,\alpha-\gamma,-\beta}\ \qquad %
& \tand\ \qquad \cI_{x-}^{\alpha,0,\beta,\beta',\gamma} %
= \cI_{x-}^{\gamma,\alpha-\gamma,-\beta}\ \qquad (\gamma\in \C). \label{eq_redintsaigo}\\
\shortintertext{%
and, similarly,
}
\cD_{0+}^{0,\alpha',\beta,\beta',\gamma} = \cD_{0+}^{\gamma,\alpha'-\gamma,\beta'-\gamma}\ \qquad %
& \tand\ \qquad \cD_{x-}^{0,\alpha',\beta,\beta',\gamma} %
= \cD_{x-}^{\gamma,\alpha'-\gamma,\beta'-\gamma}\ \qquad (\Re(\gamma) > 0).\label{eq_redderivsaigo}
\end{flalign}
\end{subequations}
Lemma~\ref{lem_powfct} and Eqs.~\eqref{eq_redintsaigo}--\eqref{eq_redderivsaigo} are useful and often required in fractional calculus.

The Appell function $F_3$ in two variables which appears in the definitions~\eqref{eq_intsaigomaeda1} and~\eqref{eq_intsaigomaeda2} satisfies a system of two linear partial differential equations of the second order and reduces to the Gauss hypergeometric function $\ps{}{2}F_1$ in the form
\[
F_3(\alpha,\gamma-\alpha,\beta,\gamma-\beta;\gamma;x,y) = \ps{}{2}F_1\l( \sbs{\alpha, \beta\\[.1cm] \gamma}\, ; x+y-xy\r).\]
Moreover, it is easily observed that
\begin{flalign*}
F_3(\alpha,0,\beta,\beta';\gamma;x,y) &= F_3(\alpha,\alpha',\beta,0;\gamma;x,y) %
= \ps{}{2}F_1\l( \sbs{\alpha, \beta\\[.1cm] \gamma}\, ; x\r)\ \qquad \tand\\
F_3(0,\alpha',\beta,\beta';\gamma;x,y) &= F_3(\alpha,\alpha',0,\beta';\gamma;x,y) %
= \ps{}{2}F_1\l( \sbs{\alpha', \beta'\\[.1cm] \gamma}\, ; x\r).
\end{flalign*}
\end{rems}
In view of the above reduction formula, the general operators reduce to the aforementioned Saigo's operators $\cI_{0+}^{\alpha,\beta,\gamma}$ and $\cI_{x-}^{\alpha,\beta,\gamma}$ defined in Definition~\ref{def_saigo}, Eqs.~\eqref{eq_RLEKop1} and~\eqref{eq_RLEKop2}.

\section{The generalized $M$-series and $K$-function} \label{M-K}
The {\em generalized $M$-series and the $K$-function} introduced in~\cite{Sharma08,Sharma12} both extend the Fox--Wright generalized hypergeometric function $\ps{}{p}\psi_q^{}(z)$ stated in Definition~\ref{def_foxwright} and of Prabakhar's three-parametric generalized M-L function defined in Eq.~\eqref{def_Prab}. Namely,
\begin{defi} \label{Mdef}
Let $z, \alpha, \beta\in \C$ and $\Re(\alpha) > 0$, then
\begin{equation} \label{def_M}
\ps{}{p}M_q^{\alpha,\beta}(z) = \ps{}{p}M_q^{\alpha,\beta}\bgl( \sbs{\bm{\rap}\\[.1cm] %
\bm{\rbq}}\, ; z\bgr) := \lsum_{n\ge 0} \frac{(a_1)_n\cdots (a_p)_n}{(b_1)_n\cdots (b_q)_n}\, %
\frac{z^n}{\Gamma(\alpha n + \beta)}\,.
\end{equation}
\end{defi}

\begin{defi} \label{Kdef}
Let $z, \alpha, \beta, \gamma\in \C$ and $\Re(\alpha) > 0$, then
\begin{equation} \label{def_K}
\ps{}{p}K_q^{\alpha,\beta;\gamma}(z) = %
\ps{}{p}K_q^{\alpha,\beta;\gamma}\bgl( \sbs{\bm{\rap}\\[.1cm] \bm{\rbq}}\, ; z\bgr) %
:= \lsum_{n\ge 0} \frac{(a_1)_n\cdots (a_p)_n}{(b_1)_n\cdots (b_q)_n}\, %
\frac{(\gamma)_n}{\Gamma(\alpha n + \beta)}\,\frac{z^n}{n!}\,.
\end{equation}
\end{defi}
As a matter of fact, the $K$-function stands for a three-parametric variant of the (two-parametric) $M$-series. Of course, if $\gamma = 1$, $\ps{}{p}K_q^{\alpha,\beta;1}(z)$ coincides with $\ps{}{p}M_q^{\alpha,\beta}(z)$.

Both series are defined provided that none of the bottom parameters $b_j$'s ($j = 1, 2,\ldots, q$) is a non-positive integer. If any of the top parameters $a_j$ is in $\Z_{\le 0}$, then the series terminate. From the ratio test, the series are convergent for all $z$ if $p\le q$, and they are divergent if 
$p > q + 1$. On the circle $|z| = 1$ and when $p = q + 1$, the conditions of convergence of the series depend on the parameters. More precisely, they are absolutely convergent when $\Re\l(\lsum_{j=1}^q b_j - \lsum_{i=1}^p a_i\r) > 0$, they are conditionally convergent when $-1 < \Re\l(\lsum_{j=1}^q  b_j - \lsum_{i=1}^p a_i\r)\le 0$ at $z = -1$ and they are divergent if $\Re\l(\lsum_{j=1}^q  b_j - \lsum_{i=1}^p a_i\r)\le -1$. The exhaustive proofs of convergence of both series in~\eqref{def_M}--\eqref{def_K} are deduced from the general theory of the Fox--Wright function (see, e.g., \cite[\color{cyan}\S16.2]{AskDaal10} and~\cite[\color{cyan}Thm.~1.5]{KiSrTr06}).

\subsection{Relations to the Fox--Wright and M-L type functions} \label{rel}
According to some well-chosen specific values of the parameters, the $M$-series and the $K$-function can be easily be reduced to classical generalized special functions. As a straighforward consequence of Definitions~\ref{def_foxwright}, \ref{Mdef} and \ref{Kdef}, the next two expressions involve the generalized Fox--Wright psi function, its normalized variant and the $H$-function. 

\begin{prop} \label{MKpsi}
Let $z, \alpha, \beta\in \C$ and $\Re(\alpha) > 0$, then
\begin{subequations}
\begin{flalign} 
\ps{}{p}M_q^{\alpha,\beta}(z) &= \frac{\Gamma(\bm{\rbq})}{\Gamma(\bm{\rap})}\, %
\lsum_{n=0}^\infty \frac{\lprod_{i=1}^p \Gamma(a_i + n)}{\lprod_{j=1}^q \Gamma(b_j + n)}\, %
\frac{z^n}{\Gamma(\alpha n+\beta)} = \frac{\Gamma(\bm{\rbq})}{\Gamma(\bm{\rap})}\ %
\ps{}{p+1}\psi_{q+1} \bgl( \sbs{ (a_i, 1)_1^p, (1,1)\\[.1cm] (b_j, 1)_1^q, (\beta,\alpha)}\, ; z\bgr)\nonumber\\
&= \ps{}{p+1}\psi_{q+1}^{*}\bgl( \sbs{ (a_i, 1)_1^p, (1,1)\\[.1cm] (b_j, 1)_1^q, (\beta,\alpha)}\, ; z\bgr) = \frac{\Gamma(\bm{\rbq})}{\Gamma(\bm{\rap})}\ %
H_{p+1,q+2}^{1,p+1} \l[-z\,\bgl|\, \sbs{(1 - \alpha_i, 1)_1^p, (0,1)\\[.1cm] %
(0,1), (1 - b_j, 1)_1^q, (1-\beta,\alpha)}\r] \label{Mpsi}
\shortintertext{%
and, similarly,
}
\ps{}{p}K_q^{\alpha,\beta;\gamma}(z) &= \frac{1}{\Gamma(\gamma)}\, %
\frac{\Gamma(\bm{\rbq})}{\Gamma(\bm{\rap})}\ \ps{}{p+2}\psi_{q+2} %
\l( \sbs{(a_i, 1)_1^p, (\gamma,1), (1,1)\\[.1cm] (b_j, 1)_1^q, (1,1), (\beta,\alpha)}\, ;z\r) %
= \frac{1}{\Gamma(\gamma)}\, \ps{}{p+2}\psi_{q+2}^{*}\l( \sbs{(a_i, 1)_1^p, (\gamma,1), (1,1)\\[.1cm] %
(b_j, 1)_1^q, (1,1), (\beta,\alpha)}\, ;z\r)\nonumber\\ 
&= \frac{1}{\Gamma(\gamma)}\, \frac{\Gamma(\bm{\rbq})}{\Gamma(\bm{\rap})}\ %
H_{p+2,q+3}^{1,p+2} \l[-z\,\bgl|\, \sbs{(1 - \alpha_i, 1)_1^p, (\gamma,1), (1,1))\\[.1cm] %
(0,1), (1 - \beta_j, 1)_1^q, (1,1),  (1-\beta,\alpha)}\r], \label{Kpsi}
\end{flalign}
\end{subequations}
provided that each member of the equations exists.
\end{prop}
The above Eqs.~\eqref{Mpsi} and~\eqref{Kpsi} exhibit the fact that the generalized $M$-series and $K$-function are in fact two obvious (rather trivial) variants of the generalized Fox-Wright psi function
$\ps{}{p}\psi_{q}$ and its normalized form $\ps{}{p}\psi_{q}^{*}$ defined in Eq.~\eqref{eq_psi*}. 

A natural unification and generalization of the Fox-Wright function $\ps{}{p}\psi_{q}^{*}$ defined by Eq.~\eqref{eq_psi*} anf of Lerch's transcendent zeta function $\Phi(z,s,a)$ defined in Note~\color{magenta}2\color{black}\ is indeed accomplished by introducing essentially arbitrary numbers of numerator and denominator parameters in the definition of $\Phi(z,s,\tau)$.

When there exists neither top nor bottom parameters in the definitions of the $M$-series and the $K$-function in Eqs.~\eqref{def_M} and~\eqref{def_K}, one obtains, respectively,
\begin{flalign} \label{eq_Mpsi} 
\ps{}{0}M_0^{\alpha,\beta}(-;-;z) &= \lsum_{n=0}^\infty \frac{z^n}{\Gamma(\alpha n+\beta)} %
=: E_{\alpha,\beta}(z)\ \qquad \tand\\
\ps{}{0}K_0^{\alpha,\beta;\gamma}(-;-;z) &= \lsum_{n=0}^\infty \frac{(\gamma)_n}{\Gamma(\alpha n+\beta)} %
\frac{z^n}{n!} = E_{\alpha,\beta}^\gamma(z) = \ps{}{1}M_1^{\alpha,\beta}(z). \label{eq_Kpsi}
\end{flalign}
The formulas coincide with the two-parametric M-L function in~\eqref{eq_Mpsi} and with  Prabhakar's generalized form in~\eqref{eq_Kpsi}, respectively. Now, if we let $\gamma = 1$ in formula~\eqref{eq_Kpsi}, the two-parametric M-L function may be rewritten in the form
\begin{equation} \label{M-L0}
\ps{}{0}K_0^{\alpha,\beta;1}(-;-; z) = E_{\alpha,\beta}^1(z) := E_{\alpha,\beta}(z) %
= \ps{}{0}M_0^{\alpha,\beta}(\gamma;1; z),
\end{equation} 
and setting $\beta = 1$ in~\eqref{M-L0}entails $E_{\alpha,1}(z) := E_{\alpha}(z)$. Finally,  putting 
$\alpha = 1$ in the above relation~\eqref{M-L0} yields $E_{1,1}(z) := E_1(z)$, which is also the exponential function $\re^z$~\cite{Rainv60}.

\section{Fractional calculus of the $M$-series and the $K$-function} \label{MKfc}
With the help of the left-sided (L-S for short) Riemann--Liouville fractional integral and derivative of order $\nu\in \C$, it can be shown first in the following Theorem~\ref{thm_fcMK}, that as naturally expected for fractional calculus operators of generalized hypergeometric special functions, the $M$-series $\ps{}{p}M_q^{\alpha,\beta}$ and the $K$-function $\ps{}{p}K_q^{\alpha,\beta;\gamma}$ write respectively in terms of $\ps{}{p+1}M_{q+1}^{\alpha,\beta}$ and $\ps{}{p+1}K_{q+1}^{\alpha,\beta;\gamma}$ up to prefactors.

\begin{thm} \label{thm_fcMK}
Let $z,\, \alpha,\, \beta,\, \gamma,\, \nu\in \C$ and $\Re(\alpha),\, \Re(\gamma),\, \Re(\nu) > 0$ and $a_i,\ b_j\in \C$ ($i = 1,\ldots, p$, $j = 1,\ldots, q$), then
\begin{flalign}
\cI_{0+}^\nu \l( \ps{}{p}M_q^{\alpha,\beta}\r)(z) &= \frac{z^\nu}{\Gamma(\nu+1)}\ %
\ps{}{p+1}M_{q+1}^{\alpha,\beta}\bgl(\sbs{\bm{\rap}, 1\\[.1cm] \bm{\rbq}, \nu+1}\, ; z\bgr)\ %
\qquad \tand\nonumber\\
\cI_{0+}^\nu \l( \ps{}{p}K_q^{\alpha,\beta;\gamma}\r)(z) &= \frac{z^\nu}{\Gamma(\nu+1)}\, %
\ps{}{p+1}K_{q+1}^{\alpha,\beta;\gamma}\bgl(\sbs{\bm{\rap}, 1\\[.1cm] \bm{\rbq}, \nu+1}\, ; z\bgr). \label{fcK1}\\
\shortintertext{%
Similarly,
}
\cD_{0+}^\nu \l( \ps{}{p}M_q^{\alpha,\beta}\r)(z) &= \frac{z^{-\nu}}{\Gamma(1-\nu)}\ %
\ps{}{p+1}M_{q+1}^{\alpha,\beta}\bgl(\sbs{\bm{\rap}, 1\\[.1cm] \bm{\rbq}, 1-\nu}\, ; z\bgr)\ %
\qquad \tand \nonumber\\
\cD_{0+}^\nu \l( \ps{}{p}K_q^{\alpha,\beta;\gamma}\r)(z) &= \frac{z^{-\nu}}{\Gamma(1-\nu)}\ %
\ps{}{p+1}K_{q+1}^{\alpha,\beta;\gamma}\bgl( \sbs{\bm{\rap}, 1\\[.1cm] \bm{\rbq}, 1-\nu}\, ; z\bgr). \label{fcK2}
\end{flalign}
\end{thm}

\begin{proof}
Consider the L-S fractional R--L operator $\cI_{0+}^\nu$ of the generalized $M$-series, defined in~\eqref{eq_RLop1}.
\[
\cI_{0+}^\nu \l(\ps{}{p}M_q^{\alpha,\beta}\r)(z) = \frac{1}{\Gamma(\nu)}\, \lint_0^z (z-t)^{\nu-1} %
\lsum_{k=0}^\infty \frac{(\bm{\rap})_k}{(\bm{\rbq})_k}\, \frac{t^k}{\Gamma(\alpha k + \beta)}\, \rd t.\]
The uniform convergence of the $M$-series follows, for example, from the convergence of the integral
involved. Thus, term by term integration yields
\begin{flalign}
\cI_{0+}^\nu \l(\ps{}{p}M_q^{\alpha,\beta}\r)(z) &= \frac{1}{\Gamma(\nu)}\, %
\lsum_{k=0}^\infty \frac{(\bm{\rap})_k}{(\bm{\rbq})_k}\, \frac{z^{\nu-1}}{\Gamma(\alpha k +\beta)}\, %
\lint_0^z (1-t/z)^{\nu-1} t^k\rd t. \nonumber\\
\intertext{%
Substituting the variable $u$ for $t/z$, by definition of $(x)_k = \Gamma(x+k)/\Gamma(x)$, we have
}
\cI_{0+}^\nu \l(\ps{}{p}M_q^{\alpha,\beta}\r)(z) &= \frac{1}{\Gamma(\nu)}\, %
\lsum_{k=0}^\infty \frac{(\bm{\rap})_k}{(\bm{\rbq})_k}\, \frac{z^{\nu+k}}{\Gamma(\alpha k + \beta)} %
\lint_0^1 (1-u)^{\nu-1} u^k\rd u\nonumber\\
&= \frac{z^\nu}{\Gamma(\nu)}\, \lsum_{k=0}^\infty \frac{(\bm{\rap})_k}{(\bm{\rbq})_k}\, %
\frac{(1)_k}{(\nu+1)_k} \frac{z^k}{\Gamma(\alpha k + \beta)}\ \qquad \tand \nonumber\\
\cI_{0+}^\nu \l(\ps{}{p}M_q^{\alpha,\beta}\r)(z) &= \frac{z^\nu}{\Gamma(\nu+1)}\, %
\ps{}{p+1}M_{q+1}^{\alpha,\beta} \bgl( \sbs{\bm{\rap}, 1\\[.1cm] \bm{\rbq}, \nu+1}\, ; z\bgr).
\end{flalign}
The proof for $\ps{}{p}K_q^{\alpha,\beta,\gamma}(z)$ proceeds exactly along the same lines.
In other words, a L-S fractional integral of an $M$-series (a $K$-function, resp.) can be rewritten as an $M$-series (a $K$-function, resp.), up to a prefactor, with the values of the indices $p$ and $q$ each incremented by one. This proves the first (integration) part of the theorem.

As for the derivation part of the theorem, consider likewise the L-S fractional R--L derivative operator $\cD_{0+}^\nu$ of the generalized $M$-series, defined in~\eqref{eq_diffop1}.
\[
\cD_{0+}^\nu \l(\ps{}{p}M_q^{\alpha,\beta}\r)(z) = \frac{1}{\Gamma(n-\nu)}\, \l(\frac{\rd}{\rd z}\r)^n %
\lint_0^z (z-t)^{n-\nu-1}\, \lsum_{k=0}^\infty \frac{(\bm{\rap})_k}{(\bm{\rbq})_k}\, %
\frac{t^k}{\Gamma(\alpha k + \beta)}\, \rd t,\]
where $n = \lfloor \Re(\nu)\rfloor + 1$. As above, term by term integration leads to
\begin{flalign*}
\cD_{0+}^\nu \l(\ps{}{p}M_q^{\alpha,\beta}\r)(z) &= \frac{1}{\Gamma(n-\nu)}\, %
\l(\frac{\rd}{\rd z}r\r)^n \lsum_{k=0}^\infty \frac{(\bm{\rap})_k}{(\bm{\rbq})_k}\, %
\frac{1}{\Gamma(\alpha k +\beta)} \lint_0^z (z-t)^{n-\nu-1} t^k\rd t\\
&= \frac{1}{\Gamma(n-\nu)}\, \l(\frac{\rd}{\rd z}\r)^n \lsum_{k=0}^\infty \frac{(\bm{\rap})_k} %
{(\bm{\rbq})_k}\, \frac{1}{\Gamma(\alpha k +\beta)}\, B(n-k,k-1)\\
&= \frac{1}{\Gamma(n-\nu)}\, \l(\frac{\rd}{\rd z}\r)^n \lsum_{k=0}^\infty \frac{(\bm{\rap})_k} %
{(\bm{\rbq})_k}\, \frac{z^{n-\nu+k}}{\Gamma(\alpha k +\beta)}\, %
\frac{\Gamma(n-\nu) \Gamma(k+1)}{\Gamma(n-\nu+k+1)}\,,
\end{flalign*}
where $k + 1 > 0$ and $n - \nu > 0$. Now, by the modified beta function formula
\[
\lint_a^b (t-a)^{\alpha-1} (b-t)^{\beta-1} = (b-a)^{\alpha+\beta-1}\, B(\alpha,\beta)\ %
\qquad (\Re(\alpha),\ \Re(\beta) > 0),\]
differentiating $n$ times the term $z^{n-\nu+k}$ yields, provided that $k + 1 > 0$\ and 
$(k - \nu + 1)_n\neq 0$,
\begin{flalign}
\cD_{0+}^\nu \l(\ps{}{p}M_q^{\alpha,\beta}\r)(z) &= z^{-\nu}\, \lsum_{k=0}^\infty %
\frac{(\bm{\rap})_k}{(\bm{\rbq})_k}\, \frac{z^k}{\Gamma(\alpha k +\beta)}\, %
\frac{\Gamma(k+1)}{\Gamma(k-\nu+1)}\ \qquad \tand\nonumber\\
\cD_{0+}^\nu \l(\ps{}{p}M_q^{\alpha,\beta}\r)(z) &= \frac{z^{-\nu}}{\Gamma(1-\nu)}\, %
\lsum_{k=0}^\infty \frac{(\bm{\rap})_k}{(\bm{\rbq})_k}\, \frac{z^k}{\Gamma(\alpha k +\beta)} %
= \frac{z^{-\nu}}{\Gamma(1-\nu)}\, \ps{}{p+1}M_{q+1}^{\alpha,\beta} %
\bgl( \sbs{\bm{\rap}, 1\\[.1cm] \bm{\rbq}, 1-\nu}\, ; z\bgr). \label{DM}
\end{flalign}
The L-S fractional integral and derivative of the $K$-function are achieved exactly along the same lines, since the pochhammer symbol $(\gamma)_k$ in the numerator of the summand $\frac{(\bm{\rap})_k}{(\bm{\rbq})_k}\, \frac{(\gamma)_k}{\Gamma(\alpha k + \beta)}\,\frac{z^k}{k!}$ is transported all along the proof steps.

Therefore, the L-S fractional derivative of an $M$-series (a $K$-function, resp.) is still an $M$-series (a $K$-function, resp.), up to a prefactor, with the value of the indices $p$ and $q$ each incremented by one. This proves the second (derivation) part of the theorem.
\end{proof}
More generally, the classes of Wright, Fox--Wright, Fox $H$- functions, etc. are all closed under Riemann--Liouville fractional integrals and derivatives, but involve a greater order of functions in each class (see Section~\ref{genfracint}).

All proofs in the following Subsections~\ref{lsMK} and~\ref{rsMK} are also very similar since, being merely parametric variants, the generalized $M$-series and the $K$-function share themselves quite similar features. Saigo's generalized fractional integral operators recalled in Section~\ref{genfracint} are basic means and tools for the purpose of integrating M-L type functions such as the $M$-series and the $K$-function in the complex plane. It must be pointed out also that the prefactors $t^{\sigma-1}$ and $t^{-\alpha-\sigma}$ which take place respectively in their left- and right-sided fractional integrands are required for proving the results. 

\subsection{Left-sided generalized fractional integrations} \label{lsMK}
In the present cases, a theorem and a corollary are deduced for the $M$-series in \S\ref{lsM} and the $K$-function in \S\ref{lsK}. By contrast with Theorem~\ref{thm_fcMK}, both proofs use Saigo's left-sided fractional integration formula~\eqref{eq_RLEKop1} applied to the $M$-series and the $K$-function as functions in $\cC$ (see~\cite{Kirya06,KiSaSa04,Saigo96,SaxSai05}). The results are due in part to Sharma~\cite{Sharma08}\footnote{% 
Kishan Sharma actually considers in~\cite{Sharma08} the L-S and the R-S fractional integrals of a simpler $M$-series: more precisely, 
$I_{0+}^{\alpha,\beta,\gamma}\bgl(t^{\eta-1}\ps{}{p}M_q^{\xi,\eta}\l(ct^\xi\r)\bgr)(x)$ and 
$I_{-}^{\alpha,\beta,\gamma}\bgl(t^{-\alpha-\eta}\ps{}{p}M_q^{\xi,\eta}\l(ct^\xi\r)\bgr)(x)$. This leads to formulas which write in terms of a Fox--Wright function $\ps{}{p+2}\psi_{q+2}$ in the L-S case and of a Fox--Wright function $\ps{}{p+3}\psi_{q+3}$ in the R-S case, respectively . Same remarks about the $K$-function $\ps{}{p}K_q^{\xi,\eta}\l(ct^\xi\r)$ considered in~\cite{Sharma12}.
}
but was recently revisited for the $M$-series by Kumar and Saxena~\cite{KumSax15} by using Samko--Maeda's fractional operators.

\subsubsection{Left-sided fractional integration of the $M$-series} \label{lsM}
The generalized Fox--Wright psi functions which are obtained in this case naturally expresses as 
$\ps{}{p+3}\psi_{q+3}$ functions, up to prefactors.
\begin{thm} \label{thm_lsM}
Let\ $z > 0$ and suppose also that the parameters $\alpha, \beta, \gamma,\ \eta,\, \xi,\, \sigma\in \C$ are constrained by $\Re(\alpha) > 0$, $\Re(\xi) > 0$, $\mu > 0$, $c\in \R$ and $a_i,\, b_j\in \C$ 
($i = 1,\ldots, p$, $j = 1,\ldots, q$). 
Let\ $\cI_{0+}^{\alpha,\beta,\gamma}$ be Saigo's L-S generalized fractional integral, then the following formula holds true.
\begin{flalign} \label{eq_lsM} %\delimiterfactor=1200 
\cI_{0+}^{\alpha,\beta,\gamma}\bgl( t^{\sigma-1} \ps{}{p}M_q^{\xi,\eta}\bl(ct^\mu\br)\bgr)(z) &= %
\frac{\Gamma(\bm{\rbq})}{\Gamma(\bm{\rap})}\, z^{\sigma-\beta-1}\nonumber\\
& \times \ps{}{p+3}\psi_{q+3}\bgl( \sbs{(a_i, 1)_1^p, (\sigma,\mu), (-\beta+\gamma+\sigma, \mu), %
(1,1)\\[.1cm] (b_j, 1)_1^q, (-\beta + \sigma, \mu), (\alpha + \gamma + \sigma, \mu), (\eta, \xi)}\, ; %
cz^\mu\bgr),
\end{flalign} 
provided that each member of the equation exists.
\end{thm}

\begin{proof}
By Definition~\ref{def_saigo} of Saigo's L-S fractional integration operator $\cI_{0+}^{\alpha,\beta,\gamma}$ applied to the $M$-function in the class $\cC$, Eq.~\eqref{eq_RLEKop1} rewrites as follows,
\begin{flalign} \label{eq_lsM1}
\cI_{0+}^{\alpha,\beta,\gamma}\bgl(t^{\sigma-1}\ps{}{p}M_q^{\xi,\eta}\bl(ct^\mu\br)\bgr)(z) &=\nonumber\\
\frac{z^{-\alpha-\beta}}{\Gamma(\alpha)}\, & \lint_0^z (z-t)^{\alpha-1}\, %
\ps{}{2}F_1\l( \sbs{\alpha+\beta, -\gamma\\[.1cm] \alpha} ; 1 - t/z\r)\, t^{\sigma-1} %
\ps{}{p}M_q^{\xi,\eta}\bl(c t^\mu\br)\,\rd t.\\
\shortintertext{%
Now, since
}
\ps{}{2}F_1\l( \sbs{\alpha+\beta,-\gamma\\[.1cm] \alpha} ; 1 - t/z\r) &:= \lsum_{n=0}^\infty %
\frac{(\alpha+\beta)_n (-\gamma)_n}{(\alpha)_n}\, \frac{(1 - t/z)^n}{n!}\nonumber\\ 
&= \lsum_{n=0}^\infty \frac{\Gamma(\alpha+\beta+n) \Gamma(-\gamma+n) \Gamma(\alpha)} %
{\Gamma(\alpha+\beta) \Gamma(-\gamma) \Gamma(\alpha+n)}\,\frac{(1 - t/z)^n}{n!}\,,\label{eq_2F1}
\end{flalign}
we may plug Eq.~\eqref{eq_2F1} into Eq.~\eqref{eq_lsM1}. Under the constaints of the theorem, interchanging the order of integration and summation within the integrand can be justified by the absolute convergence of the integral and the uniform convergence of the series involved. After substituting the variable $u$ for $t/z$ in the integral, a few simplifications lead to
\begin{multline} \label{eq_lsM3} %\delimiterfactor=1200
\cI_{0+}^{\alpha,\beta,\gamma} \bgl(t^{\sigma-1} \ps{}{p}M_q^{\xi,\eta}\bl(ct^\mu\br)\bgr)(z) = %
z^{\sigma-\beta-1}\,\\
\times\, \lsum_{n=0}^\infty \frac{\Gamma(\alpha+\beta+n) \Gamma(-\gamma+n)}{\Gamma(\alpha+n) %
\Gamma(n+1)}\, \lint_0^1 (1-u)^{\alpha-1+n}\, u^{\sigma-1}\, %
\ps{}{p}M_q^{\xi,\eta}\bl(c (zu)^{\mu}\br)\, \rd t,
\end{multline}
where, by Definition~\ref{Mdef},
\[
\ps{}{p}M_q^{\xi,\eta}\bl(c (zu)^\mu\br) := \frac{\Gamma(\bm{\rbq})}{\Gamma(\bm{\rap})}\ %
\lsum_{n=0}^\infty \frac{\Gamma(\bm{\rap} + n)}{\Gamma(\bm{\rbq} + n)}\ %
\frac{u^{\mu n} \l(c z^{\mu}\r)^n}{\Gamma(\xi n + \eta)}\,.\]
Next, using the beta function and Gau{\ss}'s summation theorem, Eq.~\eqref{eq_lsM3} simplifies and may be reexpanded into the final expression
\begin{flalign} \label{eq_lsM2}
\cI_{0+}^{\alpha,\beta,\gamma} \bgl(t^{\sigma-1} \ps{}{p}M_q^{\xi,\eta}\bl(ct^\mu\br)\bgr)(z) &= %
\frac{\Gamma(\bm{\rbq})}{\Gamma(\bm{\rap})}\, z^{-\sigma-\beta-1}\nonumber\\
\times \lsum_{n=0}^\infty \frac{\Gamma(\bm{\rap} + n)}{\Gamma(\bm{\rbq} + n)}\, & %
\frac{\Gamma(\sigma + \mu n) \Gamma(-\beta + \gamma + \sigma + \mu n) \Gamma(n+1)} %
{\Gamma(\alpha + \gamma + \sigma + \mu n) \Gamma(-\beta + \sigma + \mu n) \Gamma(\eta + \xi n)}\, %
\frac{\bl(cz^\mu\br)^n}{n!}\,,
\end{flalign}
Finally, by interpreting the r.h.s. in~\eqref{eq_lsM2} by means of Eq.~\eqref{eq_foxwright} in Definition~\ref{def_foxwright}, we can write this last summation in terms of the Fox--Wright function $\ps{}{p+3}\psi_{q+3}{}$ in Eq.~\eqref{eq_lsM}, and the theorem follows.
\end{proof} 

\subsubsection{Left-sided fractional integration of the $K$-function} \label{lsK}
The following result makes also use of Saigo's L-S fractional integration operator.
\begin{cor} \label{cor_lsK}
Let\ $z > 0$ and suppose also that the parameters $\alpha, \beta, \gamma,\ \eta,\, \nu, \xi,\, \sigma\in \C$ are constrained by $\Re(\alpha) > 0$, $\Re(\nu) > 0$, $\Re(\xi) > 0$, $\mu > 0$, $c\in \R$ 
and $a_i,\, b_j\in \C$ ($i = 1,\ldots, p$, $j = 1,\ldots, q$). By Definition~\ref{def_saigo}, Saigo's L-S fractional integration operator $\cI_{0+}^{\alpha,\beta,\gamma}$ yields
\begin{flalign} \label{eq_lsK}
\cI_{0+}^{\alpha,\beta,\gamma} \bgl(t^{\sigma-1} \ps{}{p}K_q^{\xi,\eta;\nu}\bl(ct^\mu\br)\bgr)(z) &= %
\frac{\Gamma(\bm{\rbq})}{\Gamma(\bm{\rap})}\, \frac{z^{\sigma-\beta-1}}{\Gamma(\nu)}\nonumber\\
\times &\, \ps{}{p+3}\psi_{q+3}\bgl(\sbs{ (a_i, 1)_1^p, (\sigma, \mu), (-\beta+\gamma + \sigma, \mu), % 
(\nu, 1)\\[.1cm] (b_j, 1)_1^q, (-\beta+\sigma, \mu), (\alpha+\gamma + \sigma, \mu), (\eta, \xi)}\, ; % 
cz^\mu\bgr), 
\end{flalign} 
provided each member of the equation exists.
\end{cor}

\begin{proof}
Along the same lines as in Theorem~\ref{thm_lsM} and in view of Definition~\ref{def_saigo}, Eq.~\eqref{eq_RLEKop1} applied to the $K$-function in the class $\cC$ writes now
\begin{flalign} \label{eq_lsK1}
\cI_{0+}^{\alpha,\beta,\gamma} \bgl(t^{\sigma-1} \ps{}{p}K_q^{\xi,\eta;\nu}\bl(ct^\mu\br)\bgr)(z) &= %
\frac{z^{\sigma-\beta-1}}{\Gamma(\alpha)}\nonumber\\
\lint_0^z & (z-t)^{\alpha-1}\, \ps{}{2}F_1\l( \sbs{\alpha+\beta, -\gamma\\[.1cm] \alpha}\, ; 1 - t/z\r)\, t^{\sigma-1} \ps{}{p}K_q^{\xi,\eta;\nu}\bl(c t^\mu\br)\,\rd t.
\end{flalign}
Upon interchanging the order of integration and summation (which is guaranteed by the absolute convergence of the integral and the uniform convergence of the series), the integrand can be evaluated after the substitution of variables $u = t/z$. By means of the beta function and Gau{\ss}'s summation theorem, formula~\eqref{eq_lsK1} then coincides with an infinite summation similar to Eq.~\eqref{eq_lsM2} (up to the parameter $\nu$): 
\begin{flalign} \label{eq_lsK2}
\cI_{0+}^{\alpha,\beta,\gamma} \bgl(t^{\sigma-1} \ps{}{p}K_q^{\xi,\eta;\nu}\bl(ct^\mu\br)\bgr)(z) &= %
\frac{\Gamma(\bm{\rap})}{\Gamma(\bm{\rbq})}\, \frac{z^{\eta-\beta-1}}{\Gamma(\nu)}\nonumber\\
\times \lsum_{n=0}^\infty & \frac{\Gamma(\bm{\rap} + n)}{\Gamma(\bm{\rbq} + n)}\, %
\frac{\Gamma(\sigma + \mu n) \Gamma(-\beta + \gamma + \sigma + \mu n) \Gamma(\nu + n)} %
{\Gamma(-\beta + \sigma + \mu n) \Gamma(\alpha + \gamma + \sigma + \mu n) \Gamma(\eta + \xi n)}\, %
\frac{\bl(cz^\mu\br)^n}{n!}\,.
\end{flalign} 
By Eq.~\eqref{eq_foxwright} in Definition~\ref{def_foxwright}, Corollary~\ref{cor_lsK} runs parallel to 
Theorem~\ref{thm_lsM} and Eq.~\eqref{eq_lsK2} may be expressed in terms of the Fox--Wright function $\ps{}{p+3}\psi_{q+3}$ in Eq.~\eqref{eq_lsM}. Therefore, the corollary is completed. 
\end{proof}
Note that if we set $\nu = \eta = 1$, in Eq.~\eqref{eq_lsK}, then we get the L-S fractional integral of the $M$-series $\ps{}{p}M_q^{\xi,1}\l(ct^{\mu}\r)$.

\subsection{Right-sided generalized fractional integrations} \label{rsMK}
The proofs are along the same lines as in the L-S case in Section~\ref{lsMK}. As basic tool, the Definition~\ref{def_saigo} of Saigo's right-sided (R-S for short) fractional integration formula~\eqref{eq_RLEKop2} is applied to the $M$-series and the $K$-function as functions in $\cC$ (see also~\cite{Kirya06,KiSaSa04,Saigo96,SaxSai05}). This provides the following two results in \S\S\ref{rsM}--\ref{rsK}.

\subsubsection{Right-sided fractional integration of the $M$ series} \label{rsM}
This theorem is the R-S counterpart of the L-S operator of fractional integration carried out in~Theorem~\ref{thm_lsM}.
\begin{thm} \label{thm_rsM}
Let\ $z > 0$ and suppose also that the parameters $\alpha, \beta, \gamma,\ \eta,\, \xi,\, \sigma\in \C$ are constrained by $\Re(\alpha) > 0$, $\Re(\xi) > 0$, $\mu > 0$, $c\in \R$ and $a_i,\, b_j\in \C$ 
($i = 1,\ldots, p$, $j = 1,\ldots, q$). Let $\cI_{z-}^{\alpha,\beta,\gamma}$ be the R-S operator of generalized fractional integration. Then, the following formula holds true.
\begin{flalign} \label{eq_rsM}
\cI_{z-}^{\alpha,\beta,\gamma} \bgl(t^{-\alpha-\sigma}\ps{}{p}M_q^{\xi,\eta}\bl(ct^{-\mu}\br)\bgr)(z) &= \frac{\Gamma(\bm{\rbq})}{\Gamma(\bm{\rap})}\, z^{-2\alpha-2\beta-\sigma}\nonumber\\
& \times\, \ps{}{p+3}\psi_{q+3}\bgl( \sbs{ (a_i, 1)_1^p, (\alpha + \beta + \sigma, \mu), %
(\alpha + \gamma + \sigma , \mu), (1,1)\\[.1cm] (b_j, 1)_1^q, (\eta, \xi), %
(2\alpha + \beta + \gamma + \sigma, \mu), (\alpha + \sigma, \mu)}\, ; cz^{-\mu}\bgr), 
\end{flalign}
provided that each member of the equation exists.
\end{thm}

\begin{proof}
By Definition~\ref{def_saigo} of Saigo's R-S fractional integration operator $\cI_{z-}^{\alpha,\beta,\gamma}$, Eq.~\eqref{eq_RLEKop1} applied to the $M$-function in the class $\cC$ satisfies
\begin{multline} \label{eq_rsM1} %\delimiterfactor=1200
\cI_{z-}^{\alpha,\beta,\gamma} \bgl(t^{-\alpha-\sigma} \ps{}{p}M_q^{\xi,\eta}\bl(ct^{-\mu}\br)\bgr)(z) =\\
\frac{1}{\Gamma(\alpha)}\, \lint_z^{\infty} (t-z)^{\alpha-1}\, t^{-\alpha-\beta} %
\ps{}{2}F_1\bl( \sbs{\alpha+\beta,-\gamma\\ \alpha}\, ; 1 - z/t\br)\, %
t^{-\alpha-\sigma}\, \ps{}{p}M_q^{\xi,\eta}\l(ct^{-\mu}\r)\, \rd t.
\end{multline}
The proof now follows the same lines as the proofs in~Theorem~\ref{thm_lsM} and Corollary~\ref{cor_lsK}. Interchanging the order of integration and summation within the integrand in~\eqref{eq_rsM1} is justified by absolute and uniform convergence of the integral and the series, respectively. After substituting the variable $u$ for $z/t$ and simplification, the R-S generalized fractional integral may be rewritten
\begin{multline} \label{eq_rsM2} 
\cI_{z-}^{\alpha,\beta,\gamma} \bgl(t^{-\alpha-\sigma} %
\ps{}{p}M_q^{\xi,\eta}\bl(ct^{-\mu}\br)\bgr)(z) = \frac{z^{-\alpha-\beta-\sigma}} %
{\Gamma(\alpha+\beta) \Gamma(-\gamma)}\\
\times\, \lsum_{n=0}^\infty \frac{\Gamma(\alpha+\beta + n) \Gamma(-\gamma + n)}{\Gamma(\alpha + n) %
\Gamma(n+1)}\, \lint_0^1 u^{\alpha - \beta - \sigma - 1} (1-u)^{\alpha + n - 1}\, %
\ps{}{p}M_q^{\xi,\eta}\bl(c z^{-\mu} u^{\mu}\br)\, \rd u.
\end{multline} 
Next, the inner integral in the above Eq.~\eqref{eq_rsM2} may be evaluated by means the beta function and Gau{\ss}'s summation theorem. The formula develops into an infinite summation, in the same vein as the sum in Eq.~\eqref{eq_lsM}, which yields 
\begin{flalign} \label{eq_rsM3} 
\cI_{z-}^{\alpha,\beta,\gamma} \bgl(t^{-\alpha-\sigma}\ps{}{p}M_q^{\xi,\eta}\bl(ct^{-\mu}\br)\bgr)(z) &= \frac{\Gamma(\bm{\rbq})}{\Gamma(\bm{\rap})}\, z^{-2\alpha-2\beta-\sigma}\nonumber\\
\times\, \lsum_{n=0}^\infty \frac{\Gamma(\bm{\rap} + n)}{\Gamma(\bm{\rbq} + n)}\, & %
\frac{\Gamma(\alpha + \beta + \sigma + \mu n) \Gamma(\alpha+\gamma+\sigma + \mu n) \Gamma(n+1)} %
{\Gamma(2\alpha + \gamma + \beta + \sigma + \mu n) \Gamma(\alpha + \sigma + \mu n) \Gamma(\eta+ \xi n)}\,
\frac{\bl(c z^{-\mu}\br)^n}{n!}\,.
\end{flalign}
Last, by interpreting the r.h.s. in Eq.~\eqref{eq_rsM3} by means of Eq.~\eqref{eq_foxwright} in Definition~\ref{def_foxwright}, the Fox--Wright function $\ps{}{p+3}\psi_{q+3}{}$ in~\eqref{eq_rsM} is readily derived and the theorem follows.
\end{proof}

\subsubsection{Right-sided fractional integration of the $K$ function} \label{rsK}
This last result is again the R-S counterpart of Theorem~\ref{thm_rsM}.
\begin{cor} \label{cor_rsK}
Let\ $z > 0$ and suppose also that the parameters $\alpha, \beta, \gamma,\ \eta,\, \nu, \xi,\, \sigma\in \C$ are constrained by $\Re(\alpha) > 0$, $\Re(\nu) > 0$, $\Re(\xi) > 0$, $\mu > 0$, $c\in \R$ 
and $a_i,\, b_j\in \C$ ($i = 1,\ldots, p$, $j = 1,\ldots, q$). Let $\cI_{z-}^{\alpha,\beta,\gamma}$ be the R-S operator of the generalized fractional integration, then there holds
\begin{flalign} \label{eq_rsK}
\cI_{z-}^{\alpha,\beta,\gamma}\bgl(t^{-\alpha-\sigma}\ps{}{p}K_q^{\xi,\eta;\nu}\bl(ct^{-\mu}\br)\bgr)(z) &= \frac{z^{-2\alpha-2\beta-\sigma}}{\Gamma(\nu)}\,\frac{\Gamma(\bm{\rbq})} %
{\Gamma(\bm{\rap})}\nonumber\\
& \times \ps{}{p+3}\psi_{q+3}\bgl( \sbs{ (a_i, 1)_1^p, (\alpha + \beta + \sigma, \mu), %
(\alpha + \gamma + \sigma, \mu), (\nu, 1)\\[.1cm] (b_j, 1)_1^q, (\alpha + \sigma, \mu), %
(2\alpha + \beta + \gamma + \sigma, \mu), (\eta, \xi)} ; cz^{-\mu}\bgr),
\end{flalign}  
provided that each member of the equation exists.
\end{cor}

\begin{proof}
By Definition~\ref{def_saigo} applied to the $K$-function in $\cC$, Saigo's R-S fractional integration operator is deduced from Eq.~\eqref{eq_RLEKop2} by the equation 
\begin{multline} \label{eq_rsK1}
\cI_{z-}^{\alpha,\beta,\gamma}\bgl(t^{-\alpha-\sigma}\ps{}{p}K_q^{\xi,\eta;\nu}\bl(ct^{-\mu}\br)\bgr)(z) = \frac{1}{\Gamma(\alpha)}\\
\lint_z^{\infty} (t-z)^{\alpha-1}\, t^{-\alpha-\beta}\, %
\ps{}{2}F_1\l( \sbs{\alpha+\beta, -\gamma\\[.1cm] \alpha} ; 1 - z/t\r)\, t^{-\alpha-\sigma}\, %
\ps{}{p}K_q^{\mu,\eta;\nu}\bl(ct^{-\mu}\br)\, \rd t.
\end{multline}
As in the former proofs, the generalized hypergeometric series in the kernel is defined as an infinite sum in the same form as Eq.~\eqref{eq_2F1}. Upon interchanging the order of integration and summation is again guaranteed under the constarints of the theorem (by absolute convergence and uniform convergence of the integral and the series, respectively). The evaluation of the inner integral is thus carried out by making the substitution $u = z/t$, with the help of the beta function and by Gau{\ss}'s summation theorem. Eq.~\eqref{eq_rsK1} rewrites again in the form of an infinite summation similar to the one in Eq.~\eqref{eq_lsM2}. As in the previous corollary, we obtain
\begin{multline} \label{eq_rsK2} 
\cI_{z-}^{\alpha,\beta,\gamma}\bgl(t^{-\alpha-\sigma}\ps{}{p}K_q^{\xi,\eta;\nu}\bl(ct^{-\mu}\br)\bgr)(z) = \frac{z^{-2\alpha-2\beta-\sigma}}{\Gamma(\nu)}\ \frac{\Gamma(\bm{\rbq})}{\Gamma(\bm{\rap})}\\
\times\, \lsum_{n=0}^\infty \frac{\Gamma(\bm{\rap} + n)}{\Gamma(\bm{\rbq} + n)}\, %
\frac{\Gamma(\alpha+\beta + \sigma + \mu n) \Gamma(\alpha + \gamma + \sigma  + \mu n) \Gamma(\nu + n)} %
{\Gamma(\alpha + \sigma + \mu n) \Gamma(2\alpha+\beta +\gamma+\sigma + \mu n) \Gamma(\eta + \xi n)}\, %
\frac{\bl(cz^{-\mu}\br)^n}{n!}\,.
\end{multline}
There remains to derive the Fox--Wright function $\ps{}{p+3}\psi_{q+3}$ in Eq.~\eqref{eq_rsK} from the infinite summation in Eq.~\eqref{eq_rsK2}. This simply proceeds from Eq.~\eqref{eq_foxwright} in Definition~\ref{def_foxwright}, and the result follows. 
\end{proof} 
Note that if we set $\nu = \eta = 1$, in Eq.~\eqref{eq_rsK}, then we get the R-S fractional integral of the $M$-series $\ps{}{p}M_q^{\mu,1}\l(ct^{-\mu}\r))$.

\begin{rems} \label{rems3}
In conclusion to this section it is interesting to point out that Saigo's (L-S and R-S) fractional integration of the $M$-series and the $K$-function bring along Fox--Wright functions of the 
$\ps{}{p+3}\psi_{q+3}$ type in every proof. This is partly due to the fact that all integrands contain L-S prefactors $t^{\sigma-1}$ and R-S prefactors $t^{\alpha-\sigma}$.

Following Kiryakova e.g in~\cite{Kirya94,Kirya06,Kirya10a}, Saigo's fractional integrals can be considered as examples of operators for generalized fractional integration (of R--L type) in the complex plane, with a suitable analytic kernel-function (in the class $\cC$, for example). As such, all proofs involving the fractional calculus of the Fox--Wright function can be shown by taking a Mellin--Barnes type contour integral $\cL$ in $\C$, for which the conditions ensuring the existence and analyticity of all function involved can be seen in~\cite{AskDaal10,KiSrTr06},\cite[\color{cyan}App.]{Kirya94},\cite{Kirya10a,SamKilMar93}, etc. (see also~\S\ref{intrepml}).
\end{rems}

\section{Fractional calculus involving $F_3$} \label{fcMKF3}
The present section, makes full use of the extended L-S and R-S fractional calculus involving the Appell's two variable hypergeometric function $F_3$ defined in~Eq.~\eqref{appellF3}, Definition~\ref{def_saigomaeda} (see, e.g., \cite{KumSax15}). It must be pointed out that, for suited results, the prefactor $t^{\sigma-1}$ is required in Saigo--Maeda's L-S and R-S fractional integrands and derivatives, as in Subsections~\ref{lsMK} and~\ref{lsMK} (see Remarks~\ref{rems3}). 

\subsection{Left- and right-sided fractional integration of the $M$-series and the $K$-function}
In this part, the L-S and R-S generalized fractional integration formulas of the $M$-series and the $K$-function are derived.

\subsubsection{L-S fractional integral formulas of the $M$-series and the $K$-function}
The theorem makes use of Saigo--Maeda's L-S fractional integral defined by~\ref{def_saigomaeda}.

\begin{thm} \label{thm_intlsMK}
Let $\alpha,\, \alpha',\, \beta,\, \beta',\, \gamma,\ \eta,\, \nu,\, \xi,\, \sigma\in \C$, $\Re(\nu) > 0$, $\Re(\xi) > 0$, $\Re(\gamma) > 0$, $\mu > 0$, $c\in \R$ and $a_i,\, b_j\in \C$ 
($i = 1,\ldots, p$, $j = 1,\ldots, q$), then, for $z > 0$, the following relations hold true
\begin{subequations}
\begin{flalign} 
\cI_{0+}^{\alpha,\alpha',\beta,\beta',\gamma} \bgl(t^{\sigma-1} \ps{}{p}M_q^{\xi,\eta} \l(ct^{\mu}\r)\bgr)(z) &= \frac{\Gamma(\bm{\rbq})}{\Gamma(\bm{\rap})}\, z^{\sigma-\alpha-\alpha'+\gamma-1}\nonumber\\
\times & \ps{}{p+4}\psi_{q+4}\bgl( \sbs{ (a_i, 1)_1^p, (\sigma, \mu), (\sigma+\gamma-\alpha, \mu), %
(\sigma+\beta'-\alpha', \mu), (1,1)\\[.1cm] (b_j, 1)_1^q, (1-\sigma, \mu), %
(\sigma+\gamma-\alpha'-\beta, \mu), (\sigma+\beta', \mu), (\eta, \xi ) }\, ; cz^{\mu}\bgr), \label{eq_intlsM}
\shortintertext{%
provided that each member of the equation exists.
}
\cI_{0+}^{\alpha,\alpha',\beta,\beta',\gamma} \bgl(t^{\sigma-1} \ps{}{p}K_q^{\xi,\eta;\nu} \l(ct^{\mu}\r)\bgr)(z) &= \frac{\Gamma(\bm{\rbq})}{\Gamma(\bm{\rap})}\, \frac{z^{\sigma-\alpha-\alpha'+\gamma-1}}{\Gamma(\nu)}\nonumber\\
\times & \ps{}{p+4}\psi_{q+4}\bgl( \sbs{ (a_i, 1)_1^p, (\sigma, \mu), (\sigma+\gamma-\alpha, \mu), %
(\sigma+\beta'-\alpha', \mu), (\nu,1)\\[.1cm] (b_j, 1)_1^q, (1-\sigma, \mu), %
(\sigma+\gamma-\alpha'-\beta, \mu) (\sigma+\beta', \mu), (\eta, \xi ) }\, ; cz^{\mu}\bgr), \label{eq_intlsK}
\end{flalign}
provided that each member of the equation exists.
\end{subequations}
\end{thm}

\subsubsection{Right-sided fractional integral formulas of the $M$-series and the $K$-function}
Following the definition of Saigo--Maeda's R-S fractional integral defined by~\ref{def_saigomaeda}, we get the following theorem.

\begin{thm} \label{thm_intrsMK}
Let $\alpha,\, \alpha',\, \beta,\, \beta',\, \gamma,\ \eta,\, \nu,\, \xi,\, \sigma\in \C$, $\Re(\nu) > 0$, $\Re(\xi) > 0$, $\Re(\gamma) > 0$, $\mu > 0$, $c\in \R$ and $a_i,\, b_j\in C$ ($i = 1,\ldots, p$, 
$j = 1,\ldots, q$), then, for $z > 0$, the following two relations holds true
\begin{subequations}
\begin{flalign} 
\cI_{z-}^{\alpha,\alpha',\beta,\beta',\gamma} \bgl(t^{\sigma-1} & \ps{}{p}M_q^{\xi,\eta} \l(ct^{-\mu}\r)\bgr)(z) = \frac{\Gamma(\bm{\rbq})}{\Gamma(\bm{\rap})}\, z^{\sigma-\alpha-\alpha'+\gamma-1}\nonumber\\
\times & \ps{}{p+4}\psi_{q+4}\bgl( \sbs{ (a_i, 1)_1^p, (1+\alpha+\alpha'-\gamma-\sigma, \mu), %
(1+\alpha+\beta'-\gamma-\sigma, \mu), (1-\beta-\sigma, \mu), (1,1)\\[.1cm] (b_j, 1)_1^q, %
(1-\sigma, \mu), (1+\alpha-\alpha'-\gamma-\sigma, \mu), (1-\alpha-\beta-\sigma, \mu), %
(\eta, \xi)}\, ; cz^{-\mu}\bgr), \label{eq_intrsM}\\
\shortintertext{%
provided that each member of the equation exists.
}
\cI_{z-}^{\alpha,\alpha',\beta,\beta',\gamma} \bgl(t^{\sigma-1} & \ps{}{p}K_q^{\xi,\eta;\nu} %
\l(ct^{-\mu}\r) \bgr)(z) = \frac{\Gamma(\bm{\rbq})}{\Gamma(\bm{\rap})}\, %
\frac{z^{\sigma-\alpha-\alpha'+\gamma-1}}{\Gamma(\nu)}\nonumber\\
\times & \ps{}{p+4}\psi_{q+4}\bgl( \sbs{ (a_i, 1)_1^p, (1+\alpha+\alpha'-\gamma-\sigma, \mu), %
(1+\alpha+\beta'-\gamma-\sigma, \mu), (1-\beta-\sigma, \mu), (\nu,1)\\[.1cm] (b_j, 1)_1^q, %
(1-\sigma, \mu), (1+\alpha-\alpha'-\gamma-\sigma, \mu), (1-\alpha-\beta-\sigma, \mu), %
(\eta, \xi)}\, ; cz^{-\mu}\bgr),  \label{eq_intrsK}
\end{flalign}
provided that each member of the equation exists.
\end{subequations}
\end{thm}

\begin{proof}[Short proofs]
Both proofs of Theorems~\ref{thm_intlsMK} and~\ref{thm_intrsMK} run along the very same lines. Under the assumptions of the theorems, let the function $\varphi(z)$ denote $\ps{}{p}M_q^{\xi,\eta}(z)$ and 
$\ps{}{p}K_q^{\xi,\eta;\nu}$ (resp.). By Eqs.~\eqref{eq_intsaigomaeda1} and~\eqref{eq_intsaigomaeda2} in Definition~\ref{def_saigomaeda}, Saigo--Maeda's L-S and R-S fractional integrals of $\varphi$ write respectively,
\begin{subequations}
\begin{flalign} 
\cI_{0+}^{\alpha,\alpha',\beta,\beta',\gamma} \bgl( t^{\sigma-1} \varphi\l(ct^{\mu}\r) \bgr)(z) %
&:= \frac{z^{-\alpha}}{\Gamma(\gamma)}\nonumber\\
\lint_0^\infty (z-t)^{\gamma-1} & t^{\alpha'+\sigma-1}\, %
F_3\l(\alpha,\alpha',\beta,\beta';\gamma; 1 - t/z; 1 - z/t\r)\, \varphi\l(ct^{\mu}\r)\, \rd t\\
\shortintertext{%
and
}
\cI_{z-}^{\alpha,\alpha',\beta,\beta',\gamma} \bgl( t^{\sigma-1} \varphi\l(ct^{-\mu}\r) \bgr)(z)  %
&:= \frac{z^{-\alpha'}}{\Gamma(\gamma)}\nonumber\\
\lint_z^\infty (t-z)^{\gamma-1} & t^{-\alpha+\sigma-1}\, %
F_3\l(\alpha,\alpha',\beta,\beta';\gamma; 1 - z/t; 1 - t/z\r)\, \varphi\l(ct^{-\mu}\r)\, \rd t.
\end{flalign}
\end{subequations}
By definition of Appell's two variable series $F_3$ and in view of Eq.~\eqref{eq_powfct1} in Lemma~\ref{lem_powfct}, expanding $F_3$ and interchanging the summations and the integrals is justified, under the constraints of the theorems, by the absolute convergence of integrals and the uniform convergence of series. Next, the evaluation of the inner integrals by using the beta function and Gau{\ss}'s summation theorem, yields Eqs.~\eqref{eq_intlsM}--\eqref{eq_intlsK} and Eqs.~\eqref{eq_intrsM}--\eqref{eq_intrsK}. This provides the sketchproofs of Theorems~\ref{thm_intlsMK} and~\ref{thm_intrsMK}.
\end{proof}

\subsection{Left- and right-sided fractional differentiation of the $M$-series and the $K$-function}
In the subsection, the L-S and R-S generalized fractional derivative formulas of the $M$-series and the $K$-function are deduced from Saigo--Maeda's L-S fractional derivative introduced in Definition~\ref{def_saigomaeda}.

\subsubsection{L-S fractional derivative formulas of the $M$-series and the $K$-function}

\begin{thm} \label{thm_derivlsMK}
Let $\alpha,\, \alpha',\, \beta,\, \beta',\, \gamma,\ \eta,\, \nu,\, \xi,\, \sigma\in \C$, $\Re(\nu) > 0$, $\Re(\xi) > 0$, $\Re(\gamma) > 0$, $\mu > 0$, $c\in \R$ and and $a_i,\, b_j\in \C$ ($i = 1,\ldots, p$, $j = 1,\ldots, q$), then, for $z > 0$, the following formula hold true
\begin{subequations}
\begin{flalign} 
\cD_{0+}^{\alpha,\alpha',\beta,\beta',\gamma}\bgl( t^{\sigma-1} & \ps{}{p}M_q^{\xi,\eta} \l(ct^{\mu}\r)\bgr)(z) = \frac{\Gamma(\bm{\rbq})}{\Gamma(\bm{\rap})}\, z^{\sigma-\alpha-\alpha'+\gamma-1}\nonumber\\
& \times \ps{}{p+4}\psi_{q+4}\bgl( \sbs{ (a_i, 1)_1^p, (\sigma, \mu), (\sigma+\alpha+\alpha'+\beta' %
-\gamma, \mu), (\sigma+\alpha-\beta, \mu), (1,1)\\[.1cm] (b_j, 1)_1^q, (\sigma+\alpha+\alpha'-\gamma, %
\mu), (\sigma+\alpha+\beta'-\gamma, \mu), (\sigma-\gamma, \mu), (\eta,\xi) }\, ; cz^{\mu}\bgr), \label{eq_derivlsM}
\shortintertext{%
provided that each member of the equation exists.
}
\cD_{0+}^{\alpha,\alpha',\beta,\beta',\gamma}\bgl(t^{\sigma-1} & \ps{}{p}K_q^{\xi,\eta;\nu} \l(ct^{\mu}\r)\bgr)(z) = \frac{\Gamma(\bm{\rbq})}{\Gamma(\bm{\rap})}\, %
\frac{z^{\sigma-\alpha-\alpha'+\gamma-1}}{\Gamma(\nu)}\nonumber\\
& \times \ps{}{p+4}\psi_{q+4}\bgl( \sbs{ (a_i, 1)_1^p, (\sigma, \mu), (\sigma+\gamma-\alpha, \mu), %
(\sigma+\beta'-\alpha', \mu), (\nu,1)\\[.1cm] (b_j, 1)_1^q, (1-\sigma, \mu), %
(\sigma+\gamma-\alpha'-\beta, \mu) (\sigma+\beta', \mu), (\eta, \xi ) }\, ; cz^{\mu}\bgr), 
\label{eq_derivlsK}
\end{flalign}
provided that each member of the equation exists.
\end{subequations}
\end{thm}

\subsubsection{R-S generalized fractional derivative formulas of the $M$-series and the $K$-function}

\begin{thm} \label{thm_derivrsMK}
Let $\alpha,\, \alpha',\, \beta,\, \beta',\, \gamma,\ \eta,\, \nu,\, \xi,\, \sigma\in \C$, $\Re(\nu) > 0$, $\Re(\xi) > 0$, $\Re(\gamma) > 0$, $\mu > 0$, $c\in \R$ and $a_i,\, b_j\in \C$ 
($i = 1,\ldots, p$, $j = 1,\ldots, q$), then, for $z > 0$, the following two relations holds true
\begin{subequations}
\begin{flalign}
\cD_{z-}^{\alpha,\alpha',\beta,\beta',\gamma} & \bgl( t^{\sigma-1}\ps{}{p}M_q^{\xi,\eta} \l(ct^{-\mu}\r)\bgr)(z) = \frac{\Gamma(\bm{\rbq})}{\Gamma(\bm{\rap})}\, z^{\sigma-\alpha-\alpha'+\gamma-1}\nonumber\\
& \times \ps{}{p+4}\psi_{q+4}\bgl( \sbs{ (a_i, 1)_1^p, (1-\alpha-\alpha'+\gamma-\sigma, \mu), %
(1-\alpha'-\beta+\gamma-\sigma, \mu), (1+\beta'-\sigma, \mu), (1,1)\\[.1cm] (b_j, 1)_1^q, %
(1-\sigma, \mu), (1-\alpha-\alpha'-\beta+\gamma-\sigma, \mu), (1-\alpha'+\beta'-\sigma, \mu), %
(\eta, \xi)}\, ; cz^{-\mu}\bgr), \label{eq_derivrsM}
\shortintertext{%
provided that each member of the equation exists.
}
\cD_{z-}^{\alpha,\alpha',\beta,\beta',\gamma} & \bgl(t^{\sigma-1} \ps{}{p}K_q^{\xi,\eta;\nu} %
\l(ct^{-\mu}\r) \bgr)(z) = \frac{\Gamma(\bm{\rbq})}{\Gamma(\bm{\rap})}\, %
\frac{z^{\sigma-\alpha-\alpha'+\gamma-1}}{\Gamma(\nu)}\nonumber\\
& \times \ps{}{p+4}\psi_{q+4}\bgl( \sbs{ (a_i, 1)_1^p, (1+\alpha+\alpha'-\gamma-\sigma, \mu), %
(1+\alpha+\beta'-\gamma-\sigma, \mu), (1-\beta-\sigma, \mu), (\nu,1)\\[.1cm] (b_j, 1)_1^q, %
(1-\sigma, \mu), (1+\alpha-\alpha'-\gamma-\sigma, \mu), (1-\alpha-\beta-\sigma, \mu), %
(\eta, \xi)}\, ; cz^{-\mu}\bgr), \label{eq_derivrsK}
\end{flalign}
provided that each member of the equation exists.
\end{subequations}
\end{thm}

\begin{proof}[Short proofs]
Both proofs of Theorems~\ref{thm_derivlsMK} and~\ref{thm_derivrsMK} are quite similar to those of Theorems~\ref{thm_intlsMK} and~\ref{thm_intrsMK}. Under the above assumptions, let the function 
$\varphi(z)$ denote $\ps{}{p}M_q^{\xi,\eta}(z)$ and $\ps{}{p}K_q^{\xi,\eta;\nu}$, according to the case. From Saigo--Maeda's L-S and R-S Eqs.~\eqref{eq_derivsaigomaeda1} and \eqref{eq_derivsaigomaeda2} in Definition~\ref{def_saigomaeda}, fractional derivatives are written respectively, with $n = \lfloor \Re(\gamma)\rfloor + 1$,
\begin{subequations}
\begin{flalign*}
\cD_{0+}^{\alpha,\alpha',\beta,\beta',\gamma} \bgl( t^{\sigma-1} \varphi\l(ct^{\mu}\r) \bgr)(z) &:=
\l(\frac{\rd}{\rd z}\r)^n \cI_{0+}^{-\alpha',-\alpha,-\beta'+n,-\beta,-\gamma+n} \bgl( t^{\sigma-1} % 
\varphi\l(ct^{\mu}\r) \bgr)(z)\ \quad \tand\\
\cD_{z-}^{\alpha,\alpha',\beta,\beta',\gamma} \bgl( t^{\sigma-1} \varphi\l(ct^{-\mu}\r) \bgr)(z) &:= (-1)^n \l(\frac{\rd}{\rd z}\r)^n \cI_{z-}^{-\alpha',-\alpha,-\beta',-\beta+n,-\gamma+n} %
\bgl( t^{\sigma-1} \varphi\l(ct^{-\mu}\r) \bgr)(z).
\end{flalign*}
\end{subequations}
By definition of $F_3$ and in view of Eq.~\eqref{eq_powfct1} in Lemma~\ref{lem_powfct}, expanding the series $F_3$ and interchanging the order of integration and summation is again guaranteed by the conditions of the theorems: absolute of integrals and uniform convergence of series (resp.). Next, from Eqs.~\eqref{eq_derivlsM}--\eqref{eq_derivlsK} and Eqs.~\eqref{eq_derivrsM}--\eqref{eq_derivrsK}, the identity $\frac{\rd^n}{\rd z^n} z^m = \frac{\Gamma(m+1)}{\Gamma(m-n+1)} z^{m-n}$, where $m\ge n$, involves a few simplifications, which yield the generalized Fox--Wright functions $\ps{}{p+4}\psi_{q+4}$ in Theorems~\ref{thm_intlsMK} and~\ref{thm_intrsMK}. This  completes the sketchproofs of the theorems. \end{proof}

\begin{rems} \label{rems4}
By setting  parameters ($p,\, q$, $\alpha,\, \alpha',\, \beta,\, \beta',\, \gamma,\ \eta,\, \mu,\, \xi,\, \sigma$, etc.) to various specific values, a number of special cases of the above theorems may be obtained by virtue of relations~\eqref{eq_redintsaigo} and~~\eqref{eq_redderivsaigo} in Remarks~\ref{rems2}.\par
\no For example, if $\alpha' = 0$, $\sigma = \eta$, $\mu = \xi$ and in view of Eq.~\eqref{eq_redintsaigo}, Theorems~\ref{thm_intlsMK} and~\ref{thm_intrsMK} coincide with Saigo's L-S and R-S fractional integrals of the $M$-series and the $K$-function, as carried out in~\S\S\ref{lsMK}--\ref{rsMK} (Section~\ref{MKfc}). Along the same lines, Riemann--Liouville and Erd{\'e}lyi--Kober fractional differintegral operators can be also obtained by putting $\beta = -\alpha$ and $\beta = 0$, respectively in Theorems~\ref{thm_derivlsMK}--\ref{thm_derivrsMK} and in Theorems~\ref{thm_intlsMK}--\ref{thm_intrsMK}.

Notice that Riemann--Liouville's, Saigo's and Saigo--Maeda's operators of fractional calculus operate as well for a very great number of generalized hypergeometric special functions, such as the generalized hypergeometric function $\ps{}{p}F_{q}$, the generalized Mittag-Leffler and Wright type functions, Fox 
$H$- and Meijer $G$-functions, Fox--Wright $\ps{}{p}\psi_{q}$ and Lommel--Wright,functions, Appell's one- and multi-variable functions, generalized incomplete gamma type functions $\ps{}{p}\gamma_{q}$ and 
$\ps{}{p}\Gamma_{q}$, incomplete Pochhammer symbols $(\lambda;x)_\nu$ and $[\lambda;x]_\nu$ ($\lambda, \nu\in \C$, $x\ge 0$), etc. (see a few recent papers such as, e.g., \cite{Kirya15,Kumar16,KumSax15,Parmar15,SrivastAgarwal13}).
\end{rems}

\section{Conclusion \& perspectives}
The results obtained so far are extensions of the works carried out by many authors, actually in increasing number and ability within the past years. On account of the general nature of the differintegral operators of Mittag-Leffler type functions, highly transcendental functions, etc. a number of known results can be easily found as special cases of the present results and conversely. Fractional calculus also offers wide perspectives in various domains of mathematics spreading out from pure and applied analysis, such as $q$-analogues for instance, to algebra of operators.

%It is strongly believed that all these new research results should eventually lead to various applications, especially in physics, such as the discovery of truly non-equilibrium statistical mechanics beyond Boltzmann's and Gibbs'; e.g. in entropy production, reaction, diffusion, reaction-diffusion and so forth.

\newpage
\appendix
\section*{Appendices}
\addcontentsline{toc}{section}{Appendices}
\renewcommand{\thesubsection}{\Alph{subsection}}
\numberwithin{thm}{subsection}
\numberwithin{equation}{subsection}

\subsection{Asymptotic expansion of M-L type functions ($|z|\to \infty$)} \label{appA}

\subsubsection{Asymptotic expansion of $E_{\alpha,\beta}(z)$~\cite{GoKiMaRo14,HaMaSa11}} \label{as2ml}
The asymptotic expansion of the two-parametric M-L function is based on the integral representation of $E_{\alpha,\beta}(z)$ in the form~\eqref{complexmlintrep} given in \S\ref{intrepml} with the same Hankel path $\cH$. Following D{\v{z}}rba{\v{s}}jan~\cite[\color{cyan}1952, 1960, 1966]{Dzrba52,Dzrba60,Dzrba66}, Erd{\'e}lyi {\em et al.}~\cite[\color{cyan}Vol.~III, 1953]{ErMaObTr53}, the case when $\beta = 1$ and the general case with arbitrary $\beta\in \C$ are treated in the survey of Gorenflo {\em et al.}~\cite[\color{cyan}Sect.~4]{GoKiMaRo14}. The following two representations are considered.
\begin{subequations}
\begin{flalign} 
E_{\alpha,\beta}(z) &= \frac{1}{2\pi\ri \alpha} \lint_{\gamma(\epsilon;\delta)} \frac{ \re^{s^{1/\alpha}} s^{(1-\beta)/\alpha} }{s - z}\, \rd s\ \qquad \tfor\ z\in \Omega^{(-)}(\epsilon;\delta)\\ \label{rep1}
\shortintertext{%
and, for $z\in \Omega^{(+)}(\epsilon;\delta)$,
}
E_{\alpha,\beta}(z) &= \frac{1}{\alpha}\,\re^{s^{1/\alpha}} z^{(1-\beta)/\alpha} + %
\frac{1}{2\pi\ri \alpha} \lint_{\gamma(\epsilon;\delta)} \frac{ \re^{s^{1/\alpha}} %
s^{(1-\beta)/\alpha} }{s - z}\, \rd s, \label{rep2}
\end{flalign}
\end{subequations}
under the constraints
\begin{equation} \label{cond}
0 < \alpha < 2\ \qquad \tand\ \qquad \pi\alpha/2 < \delta < \min\bl(\pi,\pi\alpha\br).
\end{equation}

The above contour denoted by $\gamma(\epsilon;\delta) = \bl\{\epsilon > 0$,\, $0 < \delta\le  \pi\br\}$   is oriented by non-decreasing $\arg s$. It consists in the two rays $S_{-\delta} = \bl\{\arg s = -\delta,\, |s|\ge \epsilon\br\}$ and $S_{\delta} = \bl\{\arg s = \delta,\, |s|\ge \epsilon\br\}$, and the circular arc $C_\delta(0;\epsilon) = \bl\{|s| = \epsilon,\, -\delta\le \arg s\le \delta\br\}$.

\no If $0 < \delta < \pi$, then the Hankel contour $\gamma(\epsilon;\delta)$ divides the complex $s$-plane into two unbounded regions, namely $\Omega^{(-)}(\epsilon;\delta)$ to the left of $\gamma(\epsilon;\delta)$ by orientation and $\Omega^{(+)}(\epsilon;\delta)$ to the right of it. If $\delta = \pi$, then the contour consists of the circle $|s| = \epsilon$ and the twice passable ray $-\infty < s\le -\epsilon$.

Using the integral representations in~\eqref{rep1} and~\eqref{rep2} entails asymptotic expansions for the M-L function in the complex plane (see, e.g., \cite[\color{cyan}Thm.~4.3]{GoKiMaRo14}). Let 
$0 < \alpha < 2$, $\beta\in \C$\ be an arbitrary number, and $\delta\in \R$ be chosen to fullfill the condition \eqref{cond}. Thus, for any $m\in \N$ (and for $m = 0$ if the 'empty sum convention' is adopted), for all $z$ such that $|\arg z|\le \delta$,
\begin{subequations}
\begin{flalign}
E_{\alpha,\beta}(z) = \frac{1}{\alpha}\, z^{(1-\beta)/\alpha}\, \re^{z^{1/\alpha}} %
- \lsum_{n=1}^m \frac{z^{-n}}{\Gamma(\beta - \alpha n)}\; &+\; \cO\l(\frac{1}{z^{m+1}}\r)\ %
\qquad (|z|\to \infty)\\
\intertext{%
and analogously, for all $z$ such that $\delta < |\arg z| < \pi$,
}
E_{\alpha,\beta}(z) = - \lsum_{n=1}^m \frac{z^{-n}}{\Gamma(\beta - \alpha n)} %
\; &+\; \cO\l(\frac{1}{z^{m+1}}\r)\ \qquad (|z|\to \infty). \label{asympt2}
\end{flalign}
\end{subequations}
Similarly, in the case when $\alpha\ge 2$, the following asymptotic formula holds for $\beta\in \C$ and 
$m\in \N$ (see, e.g., \cite[\color{cyan}Thm.~4.4]{GoKiMaRo14}).
\begin{flalign} \label{asympt3}
E_{\alpha,\beta}(z) = \frac{1}{\alpha}\ \lsum_{|\arg z+2\pi n|<\frac{3\pi \alpha}{4}} & %
\l(z^{1/\alpha} \re^{2\pi \ri n/\alpha}\r)^{1-\beta}\, %
\exp\l(z^{1/\alpha} \re^{2\pi \ri n/\alpha}\r) \nonumber\\
& - \lsum_{n=1}^m \frac{z^{-n}}{\Gamma(\beta - \alpha n)}\; +\; \cO\l(\frac{1}{z^{m+1}}\r)\ %
\qquad (|z|\to \infty).
\end{flalign}
In the case of $\beta = 1$, $E_{\alpha}(z) := E_{\alpha,1}(z)$ reduces to related asymptotic expansion stated, e.g., in~\cite[\color{cyan}\S.~3.4]{GoKiMaRo14}.

\subsubsection{Asymptotic expansion of Prabhakar's M-L function~\cite{GoKiMaRo14,Paneva12}} \label{as3ml}
By Lemma~\ref{lem_mbrep}, the function $E_{\alpha,\beta}^{\gamma}(z)$ can be represented via the Mellin–Barnes integral given in Eq.~\eqref{eq_mbrep} (cf.~\cite[\color{cyan}\S~5.1.2]{GoKiMaRo14}), under the constraints $z\in \C$ and $|\arg z| < \pi$\ for $\alpha\in \R_+$, $\beta,\, \gamma\in \C$ and $\Re(\gamma) > 0$. Now, when $\beta$ is a sufficiently large real number, one can use Stirling’s formula, valid for any fixed $a$,
\[ 
\Gamma(z + a)\approx \sqrt{2\pi}\, z^{z+a-1/2}\, \re^{-z} \quad \text{as}\ \ |z|\to \infty,\]
in order to get the following asymptotic formula as $x\to \infty$ ($a > 0$, $\alpha > 0$, $\beta > 0$, 
$\gamma > 0$).
\begin{flalign} \label{asformprabml}
\Gamma(\alpha) E_{\alpha,\beta}^{\gamma}\bl(a(\alpha x)^\gamma\br) &\approx \lsum_{n=0}^\infty %
\frac{(\beta)_n\, a^n\, x^{\gamma n}}{n!}\, \frac{\sqrt{2\pi}\, \alpha^{\alpha-1/2}\, \re^{-\alpha}} %
{\sqrt{2\pi}\, \alpha^{\alpha-1/2+\gamma n}\, \re^{-\alpha}}\nonumber\\
&= \lsum_{n=0}^\infty \frac{(\beta)_n}{n!}\,\bl(a (x/\alpha)^\gamma\br)^n = %
\frac{1}{\bl(1 +  a(x/\alpha)^\gamma\br)^\beta}\,.
\end{flalign}

As in the case of the M-L function with two parameters, the asymptotic behaviour of the three parametric function $E_{\alpha,\beta}^\gamma(z)$ critically depends on the values of the parameters $\alpha$, 
$\beta$, $\gamma$ and cannot easily be described. In principle, an asymptotic expansion of Prabhakar's function can be found from its representation via a generalized Wright function or $H$-function (cf.~\cite[\color{cyan}\S~5.1.5]{GoKiMaRo14} or~\cite{KiSrTr06}). To the best of our knowledge, the asymptotic behaviour of $E_{\alpha,\beta}^\gamma(z)$ in different domains of the complex plane (similar, for example, to that of Prop.~3.6, and of Theorems~4.3 and 4.4 in~\cite{GoKiMaRo14}) has not yet been described in an explicit form.

Another step towards an asymptotic expansion was obtained however in~\cite{Paneva12} for non-negative integer values of $\beta = n$, when $n$ gets large. Prabhakar's M-L function is then naturally defined by
\begin{equation} \label{def_prabn}
E_{\alpha,n}^\gamma(z) := \lsum_{k=0}^\infty \frac{(\gamma)_k}{\Gamma(\alpha k + n)} \frac{z^k}{k!}\ %
\qquad (\alpha, \gamma\in \C,\, \Re(\alpha) >0,\ n\in \N).
\end{equation}
Next, given a number $\gamma$, suppose that some coefficients in Def.~\eqref{def_prabn} equal zero; that is, there exists an integer $p\in \N$, such that~\eqref{def_prabn} can be written as follows,
\begin{equation} \label{rep_prabn1}
E_{\alpha,n}^\gamma(z) = z^p \lsum_{k=p}^\infty \frac{(\gamma)_k}{\Gamma(\alpha k + n)} \frac{z^{k-p}}{k!}.
\end{equation}
Now, three main cases may be considered.
\bi
\item[{\em (i)}]\ If $\gamma\in \C\setminus \Z_{\le 0}$, then $p = 0$ for $n\in \N$ and $p = 1$ for $n = 0$. The functions $E_{\alpha,n}^\gamma(z)$ are entire of order $\rho = 1/\Re(\alpha)$ and type $\sigma = 1$.

\item[{\em (ii)}]\ Otherwise, the functions $E_{\alpha,n}^\gamma(z)$ actually reduce to polynomials of power $-\gamma$ and, denoting $m = -\gamma$, the representation~\eqref{rep_prabn1} can be rewriten in the form
\begin{flalign} \label{rep_prabn2}
E_{\alpha,n}^\gamma(z) &= z^p \lsum_{k=p}^m \frac{(\gamma)_k}{\Gamma(\alpha k + n)} \frac{z^{k-p}}{k!} %
\nonumber\\
\shortintertext{%
and, since $(-m)_k = \binom{m}{k}$ when $k\le m$,
}
E_{\alpha,n}^{-m}(z) &= z^p \lsum_{k=p}^m (-1)^k \binom{m}{k}\, \frac{z^{k-p}}{\Gamma(\alpha k + n)}.
\end{flalign}

\item[{\em (iii)}]\ If $\gamma = 0$, then $E_{\alpha,n}^0(z) = 1/\Gamma(n)$ for $n\in \N$ and $E_{\alpha,n}^0(z) = 0$ for $n = 0$. 
\ei

Summarizing, if $\gamma$ is a non-positive integer, as seen above, the three-parametric M-L functions reduce to polynomials, while when $\gamma\notin \Z_{\le 0}$, they are entire functions of order $\rho = 1/\Re(\alpha)$ and type $\sigma = 1$. Finally, the fonctions $E_{\alpha,n}^\gamma(z)$ admit asymptotic behaviours written in the following form as $n\to \infty$.
\begin{equation}
E_{\alpha,n}^\gamma(z) = \frac{(\gamma)_p}{\Gamma(\alpha p + n)}\, z^p %
\l(1 + \theta_{\alpha,n}^\gamma(z)\r)\ \qquad \text{with} \label{asymp_prabn0}
\end{equation}
\begin{subequations}
\begin{flalign}
\theta_{\alpha,n}^\gamma(z) &= \lsum_{k=p+1}^\infty \frac{\Gamma(\alpha p + n)}{\Gamma(\alpha k + n)}\, %
\frac{z^{k-p}}{k!}\ \qquad \tfor\ \gamma\in \C\setminus \Z_{\le 0} \label{asymp_prabn1}\\
\shortintertext{%
and respectively, for $\gamma = -m$, $m\in \N$,
}
\theta_{\alpha,n}^{-m}(z) &= \lsum_{k=p+1}^\infty \frac{(-m)_k}{(-m)_p}\, %
\frac{\Gamma(\alpha p + n)}{\Gamma(\alpha k + n)}\, \frac{z^{k-p}}{k!} = \lsum_{k=p+1}^\infty %
\frac{(-1)^{k-p} \binom{m}{k}}{\binom{m}{p}}\, \frac{\Gamma(\alpha p + n)}{\Gamma(\alpha k + n)}\, %
z^{k-p}. \label{asymp_prabn2}
\end{flalign}
\end{subequations}
In representations~\eqref{asymp_prabn1}--\eqref{asymp_prabn2}, $\gamma\ne 0$ and the parameter $p$ is determined as follows: $p = 0$ for all $n\in \Z_{>0}$ and $p = 1$ for $n = 0$.

\subsubsection{Asymptotics for the multiple M-L function $F_{\alpha,\beta}^{(\mu)}(z)$~\cite{Gerhold12}} \label{asFml}
The two-parametric M-L function $E_{\alpha,\beta}(z)$ can be generalized by attaching an exponent to its Taylor coefficients. For real values $\alpha$, $\beta$ and $\mu > 0$, the series
\begin{equation} \label{def_Fml}
F_{\alpha,\beta}^{(\mu)}(z) := \lsum_{n\ge 0} \frac{z^n}{\Gamma(\alpha n + \beta)^\mu}
\end{equation}
defines an entire function of $z\in \C$ of order $\rho = 1/(\alpha \mu)$. Special cases are the two-parametric M-L function and the Bessel function $I_0(2\sqrt{z}) = \lsum_{n=0}^\infty z^n/n!$.

If $\mu$ is a natural number, then the function in~\eqref{def_Fml} is an instance of the multiple M-L function investigated, e.g., by Kiryakova~\cite{Kirya10b}, in connection with numerous applications to fractional calculus; but it seems that the asymptotic behaviour of the multiple M-L function had not been
studied earlier than Gerhold's paper~\cite{Gerhold12}. The asymptotic behaviour of $F_{\alpha,\beta}^{(\mu)}(z)$ as $z\to \infty$ is considered in a sector of the complex plane, containing the positive real line. For the two-parametric M-L function ($\mu = 1$), this is usually analysed by an integral representation (see previous \S\ref{as2ml}), which appears to have no immediate extension to $\mu = 1$.
Whatsoever, the asymptotic result can be established by approximating the sum by an integral and then use
the Laplace method.

\begin{thm} \label{asym_mml}
Let $\mu$, $\alpha, \beta > 0$\ and $\epsilon > 0$ be arbitrary. Then, as $z\to \infty$ in the sector
\[
|\arg z|\le %
\begin{cases}
\alpha \mu \pi/2 - \epsilon, &\ \qquad 0 < \alpha \mu < 2\\ 
(2 - \alpha \mu/2)\pi &\ \qquad 2\le  \alpha \mu < 4\\
0 &\ \qquad \alpha \mu\ge 4,
\end{cases}
\]
we have the asymptotics
\begin{equation} \label{eqasym_mml}
F_{\alpha,\beta}^{(\mu)}(z)\sim \frac{1}{\alpha\sqrt{\mu}}\, (2\pi)^{(1-\mu)/2}\, %
z^{(\mu-2\beta\mu+1)/2\alpha\mu}\, \re^{\mu z^{1/\alpha\mu}}.
\end{equation}
\end{thm} 
Applying the Laplace method directly does not work for non-real $z$; the absolute values of the
summands in~\eqref{def_Fml} are peaked near $n\approx \alpha^{-1} |z|^{1/\alpha\mu}$, but it seems that one cannot balance the local expansion and the tails. This is caused by oscillations in the summands, which can be dealt with by shifting the problem to the asymptotic evaluation of an integral. The Laplace method then succeeds, after moving the integration contour through a saddle point located approximately
at $\alpha^{-1} z^{1/\alpha\mu}$. For that purpose, an integral representation of~\eqref{def_Fml} is needed as $z\to \infty$ in the sector $|\arg z|\le \max\bl(0, (2-\alpha \mu/2)\pi\bl)$, which leads to
\[
F_{\alpha,\beta}^{(\mu)}(z) = \lint_0^\infty \frac{z^t}{\Gamma(\alpha t + \beta)^\mu}\, \rd t + \cO(z).\]

The saddle point method may then be applied to the above integral, first by locating the saddle point, next by changing the integration contour to a line $\cL$ which begins at $0$ and passes through $t_0$ (valid for large $|z|$) and finally by picking up the dominant contribution of the integral from the range $|t - t_0|\le |t_0|^\tau$ ($\tau\in ]1/2, 2/3[$) around the saddle point. Substituting 
$t = t_0(y+1)$ with $y\ge -1$ in the line $\cL$ divides the r.h.s. in three integrals in the form
\begin{flalign}
\lint_\cL \frac{z^t}{\Gamma(\alpha t + \beta)^\mu}\, \rd t &= t_0 \lint_{-1}^\infty %
\frac{z^{t_0(y+1)}}{\Gamma\bl(\alpha t_0(y+1) + \beta\br)^\mu}\, \rd y\nonumber\\
&= t_0 \l( \lint_{-1}^{-|t_0|^{\tau-1}} + \lint_{-|t_0|^{\tau-1}}^{|t_0|^{\tau-1}} + %
\lint_{|t_0|^{\tau-1}}^\infty \r)\, \frac{z^{t_0(y+1)}}{\Gamma\bl(\alpha t_0(y+1) +\beta\br)^\mu}\,\rd y.
\end{flalign}
For large $|z|$, the central integral provides the r.h.s. of Eq.~\eqref{eqasym_mml} in Theorem~\ref{asym_mml} and the two side integrals are negligible. A full asymptotic expansion can be obtained easily by pushing the local expansion around the saddle point further. 

A possible question for future research remains the finding of an analytic continuation for the multiple M-L function. For $\mu\in \C$ and fixed $z$, one might ask whether the function defined by~\eqref{def_Fml} has an analytic continuation for $\Re(\mu)\le 0$.

\begin{rem}
Whenever $\mu > 0$ and $\alpha = \beta = 1$, the asymptotics as $n\to \infty$ along the real line of 
\[
F_{1,1}^{(\mu)}(z) := \lsum_{n=0}^\infty \frac{z^n}{n!^\mu}\]
is a special case of the above asymptotic behaviour of~\ref{def_Fml}. Namely, in the sector defined in Theorem~\ref{asym_mml} (with $\alpha = \beta = 1$), we get the asymptotics 
\[
F_{1,1}^{(\mu)}(z)\sim \frac{1}{\sqrt{\mu}}\, (2\pi)^{(1-\mu)/2}\, %
z^{(1-\mu)/2\mu}\, \re^{\mu z^{1/\mu}}\ \qquad (z\to \infty).\]

Notice that the (possibly formal) series $F_{1,1}^{(\mu)}(z)$ is not {\em D-finite} (or {\em non-holonomic}) for any $\mu\in \R\setminus \Q$.

Besides, this function satisfies the integral relation
\[
\lint_0^\infty \re^{-t/z} F_{1,1}^{(\mu+1)}(t)\, \rd t = z F_{1,1}^{(\mu)}(z)\ \qquad (z\ne 0).\]
The proof is immediate by simple expansion and substitution $s = t/z$ in the above integral.
\begin{align*}
\lint_0^\infty \re^{-t/z} F_{1,1}^{(\mu+1)}(t)\, \rd t &= %
\lint_0^\infty \re^{-t/z} \lsum_{n=0}^\infty \frac{t^n}{n!^{(\mu+1)}}\, \rd t %
= \lsum_{n=0}^\infty \frac{1}{n!^{(\mu+1)}}\, \lint_0^\infty t^n \re^{-t/z}\, \rd t\\
&= \lsum_{n=0}^\infty \frac{z^{n+1}}{n!^{(\mu+1)}}\, \lint_0^\infty s^n \re^{-s}\, \rd s %
= \lsum_{n=0}^\infty \frac{z^{n+1}}{n!^{(\mu)}}.
\end{align*}
\end{rem}

\subsection{Determination of $E_{-\alpha,\beta}(z)$ with negative value of the first parameter} \label{appB}
To find a defining equation of the function $E_{-\alpha,\beta}(z)$ rewrite the integral representation~\eqref{2mlintrep} of the M-L function with parameters $\alpha,\, \beta\in \R$, $\alpha < 0$ as 
\[
E_{\alpha,\beta}(z) = \frac{1}{2\pi} \lint_{\cH} \frac{\re^\zeta}{\zeta^{\beta} - z\zeta^{-\alpha+\beta}}\, \rd \zeta\ \qquad (z\in \C).\]
After expanding part of the above integrand and then substituting it into Eq.~\eqref{2mlintrep} we obtain the following definition of $E_{-\alpha,\beta}(z)$ with negative value of the first parameter,
\begin{subequations}
\begin{flalign} 
E_{-\alpha,\beta}(z) &= \frac{1}{\Gamma(\beta)} - E_{\alpha,\beta}(\tfrac{1}{z})\ %
\qquad (\alpha > 0,\ z\in \C\setminus \{0\}). \label{def_negparml1}\\
\shortintertext{%
In particular, for $\beta = 1$,
}
E_{-\alpha}(z) &:= E_{-\alpha,1}(z) = 1 - E_{\alpha}(\tfrac{1}{z})\ %
\qquad (\alpha > 0\ z\in \C\setminus \{0\}). \nonumber
\end{flalign}
\end{subequations}
Now, from the known recurrence formula
\begin{subequations}
\begin{flalign} 
E_{\alpha,\beta}(z) &= \frac{1}{\Gamma(\beta)} + z E_{\alpha,\alpha+\beta}(z), \nonumber\\
\intertext{%
another variant of the definition in Eq.~\eqref{def_negparml1} is derived,
}
E_{-\alpha,\beta}(z) &= - \frac{1}{z}\, E_{\alpha,\alpha+\beta}(\tfrac{1}{z})\ \ %
\qquad (\alpha > 0,\, \beta\in \R,\ z\in \C\setminus \{0\}). \label{def_negparml2} 
\end{flalign}
\end{subequations}
The definitions~\eqref{def_negparml1}--\eqref{def_negparml2} actually determine the same function, analytic in $\C\setminus \{0\}$.

From the definition of $E_{\alpha,\beta}(z)$ in~\eqref{def_mitt2}, also results the following series representation of the extended M-L function (i.e. the function corresponding to real negative values of the first parameter),
\[
E_{-\alpha,\beta}(z) = - \lsum_{n=1}^\infty \frac{z^{-1}}{\Gamma(\alpha n + \beta)}\, %
\qquad (z\in \C\setminus \{0\}).\]
By using this series representation and the definitions (\eqref{def_negparml1} or~\eqref{def_negparml2}), the extended two-parametric M-L function with real negative first parameter provides several classical functional, differential and recurrence relations, which are analogous to corresponding relations for $E_{\alpha,\beta}(z)$ with positive valued first parameter.

\subsection{Complex contour for the reciprocal gamma and the beta functions} \label{appC}
Recall the classical definitions of Euler's gamma and beta functions (1729 and 1772, resp.) and {\em  Euler's reflection formula for the gamma function} (1771); all three are required in this appendix. The latter reflection formula~\eqref{reflexform} connects the gamma function with the sine function.\footnote{%
To prove the reflection formula~\eqref{reflexform}, first set $t = s/(s-1)$ in Definition~\eqref{def_beta}, which gives rise to the {\em second beta integral} with integration along the real half line:
\begin{equation} \label{secbeta}
B(x,y) = \lint_0^\infty \frac{s^{x-1}}{(1 + s)^{x+y}}\rd s = \frac{\Gamma(x) \Gamma(y)}{\Gamma(x+y)}\,.
\end{equation}
Next, set $y = 1 - x$, $0 < x < 1$ in Eq.~\eqref{secbeta} to obtain $\Gamma(x) \Gamma(1-x) %
= \lint_0^\infty \frac{t^{x-1}}{(1+t)}\rd t$.

This integral is computed by considering the contour integral $I_x := \lint_\cC \frac{z^{x-1}}{(1 - z)}\rd z$, where $\cC$ consists of two circles about the origin of radii $r$ and $\varepsilon$, respectively, which are joined along the negative real axis from $-r$ to $-\varepsilon$. Move along the outer circle in the counterclockwise direction, and along the inner circle in the clockwise direction. By the residue theorem, $I_x = -2\pi\ri$ when $z^{x-1}$ has its principal value. Thus,
\[
-2\pi\ri = \lint_{-\pi}^{\pi} \frac{\ri r^x \re^{\ri x\theta}}{1 - r\re^{\ri \theta}}\rd \theta %
+ \lint_{r}^{\varepsilon} \frac{t^{x-1} \re^{\ri x\pi}}{1 + t}\rd t + %
\lint_{\pi}^{-\pi} \frac{\ri \varepsilon^x \re^{\ri x\theta}}{1- \varepsilon\re^{\ri \theta}}\rd \theta %
+ \lint_{\varepsilon}^{r} \frac{t^{x-1} \re^{-\ri x\pi}}{1 + t}\rd t.\]
Let $r\to \infty$ and $\varepsilon\to 0$, so that the first and third integrals tend to zero and the second and fourth combine to provide Euler's reflection formula~\eqref{reflexform} for $0 < x < 1$. The full result follows next by analytic continuation. This, in a sense, shows that $1/\Gamma(z)$ is `half of the sine function’. (See, e.g., Temme~\cite[\color{cyan}Chap.~3, \S3.2.5]{Temme96}.)
}
\begin{subequations}
\begin{flalign} 
\Gamma(z) &:= \lint_0^\infty t^{z-1} \re^{-t}\rd t\ \qquad (\Re(z) > 0) \label{def_gamma}\\
B(x,y) &:= \lint_0^1 t^{x-1} (1 - t)^{y-1}\rd t\ \qquad (\Re(x) > 0,\, \Re(y) > 0) \label{def_beta}\\
\Gamma(z) \Gamma(1-z) &= \frac{\pi}{\sin(\pi z)}\ \qquad (z\notin \Z). \label{reflexform}
\end{flalign} 
\end{subequations}

Hankel's contour integral provides one of the most beautiful and useful integral representations of the reciprocal gamma function (see, e.g., Temme~\cite[\color{cyan}Chap.~3, \S3.2.6]{Temme96}). It has the form
\begin{equation} \label{eq_recipgamma}
\frac{1}{\Gamma(z)} = \frac{1}{2\pi \ri} \lint_{\cH} \re^s\, s^{-z}\rd s\ \qquad (z\in \C).
\end{equation}
The contour of integration $\cH$ is the Hankel contour that runs from $-\infty$, $\arg s = -\pi$, encircles the origin in the positive direction (that is counterclockwise) ends at $-\infty$, now with 
$\arg s = +\pi$\ (this is the reason why the notation $\lint_{-\infty}^{(0+)}$ is sometimes used instead of notation $\lint_\cH$). The multi-valued function $s^{-z}$\ is assumed to be real for real values of 
$z$\ and $s$, $s > 0$.

A proof of representation~\eqref{eq_recipgamma} follows immediately from the theory of Laplace transforms: from the well-known integral
\[
\frac{\Gamma(z)}{s^z} = \lint_0^\infty t^{z-1} \re^{-st} \rd t,\]
Eq.~\eqref{eq_recipgamma} is obtained as a special case of the inversion formula. A direct proof follows from a special choice of $\cH$, that is the negative real axis. This is only possible when $\Re(z) < 1$. Under this condition, the contribution from a small circle around the origin, with radius tending to zero, can be neglected. Thus, the r.h.s. of~\eqref{eq_recipgamma} yields,
\[
\frac{1}{2\pi \ri} \bgl(-\lint_\infty^0 \l(s\re^{-\ri \pi}\r)^{-z}\, \re^{-s}\rd s - \lint_0^\infty %
\l(s\re^{\ri \pi}\r)^{-z}\, \re^{-s}\rd s\bgr) = \frac{\sin \pi z}{\pi} \Gamma(1-z).\]
With the help of Euler's reflection formula~\eqref{reflexform}, it is easily shown that the above relation equals indeed the l.h.s. of~\eqref{eq_recipgamma}. In a final step, we can deduce from the principle of analytic continuation that Eq.~\eqref{eq_recipgamma} holds true for all finite complex values of $z$. Namely, both the l.h.s. and the r.h.s. of Eq.~\eqref{eq_recipgamma} are entire functions of $z$. As pointed out in \S\ref{intrepml}, the integral representation of the one- and two-parametric M-L functions is given by Eq.~\eqref{complexmlintrep} by means of the Hankel contour $\cH$ for the 
$1/\Gamma(z)$ integral in~\eqref{eq_recipgamma} (see Note~\color{cyan}5\color{black}).

Another form of~\eqref{eq_recipgamma} gives rise to the following integral representation of $\Gamma(z)$
\begin{equation} \label{repgamma}
\Gamma(z) = \frac{1}{2\ri \sin(\pi z)} \lint_{\cH} s^{z-1}\, \re^s\rd s.
\end{equation}
The substitution $s = -t$ yields an integrand as in the starting point~\eqref{def_gamma} (that is, the definition of $\Gamma$). The main idea which stands behind~\eqref{def_gamma} is that the many-valued function $t^{z-1}$ arising in definition~\eqref{def_gamma} can be used to open up the original contour along $[0,\infty)$, and obtain a representation that is valid in a larger domain of the parameter $z$. This approach can be useful with other special functions, e.g. the complex contour for the beta integral, which works as follows~\cite[\color{cyan}Chap.~3, \S3.2]{Temme96}.

\subsubsection*{A complex contour for the beta function}
Consider the integral
\[
I_{x,y} = \frac{1}{2\pi \ri} \lint_0^{(1+)} \omega^{x-1} (1 - \omega)^{y-1}\rd \omega,\]
where $\Re(x) > 0$ and $y\in \C$. The contour starts and ends at the origin, and encircles the point
1 in the positive (counterclockwise) direction. The argument (or {\em phase}) of $\omega - 1$ is zero at positive points larger than 1. When $\Re(y) > 0$ we can deform the contour along $(0, 1)$. Then we get $I_{x,y} = B(x,y) \sin(\pi y)/y$ and it follows that
\begin{equation}
B(x,y) = \frac{y}{\sin(\pi y)} \frac{1}{2\ri \pi} %
\lint_0^{(1+)} \omega^{x-1} (1 - \omega)^{y-1}\rd \omega.
\end{equation}
The integral is defined for any complex value of $y$. For $y = 1, 2,\ldots$, the integral vanishes;
this absorbs the infinite values of the term in front of the integral.

\subsubsection*{Double contour integral}
It is possible to replace the integral for $\Gamma(z)$ along a half line by a contour integral which converges for all values of $z$. A similar process can be carried out for the beta integral.

Let $P$ be any point between 0 and 1. The following Pochhammer’s extension of the beta integral holds:
\[
\lint_P^{\qquad (1+,0+,1-,0-)} t^{x-1} (1 - t)^{y-1}\rd t = %
\frac{4\pi \re^{\pi \ri (x+y)}}{\Gamma(1-x) \Gamma(1-y) \Gamma(x+y)}\,.\]
Here the contour starts at $P$, encircles the point 1 in the positive (counterclockwise) direction,
returns to $P$, then encircles the origin in the positive direction, and returns to $P$. 
The $1-$, $0-$ indicates that now the path of integration is in the clockwise direction, first around 1 and then 0. The formula is proved by the same method as Hankel’s formula. Notice that it is true for any complex $x$ and $y$: both sides are entire functions of $x$ and $y$.

\subsection{Integral representation of the Gau{\ss} hypergeometric function} \label{appD}
The Mellin--Barnes integral is a contour integral representation involving a product of gamma functions. The \emph{confluent} hypergeometric function $\ps{}{1}F_{1}^{}(z)$, for example, can be represented by such a Mellin--Barnes type contour integral, in which case the integral is valid in the sector 
$|\arg(-z)| < \pi/2$. More generally, all confluent hypergeometric functions, such as $\ps{}{1}F_1^{}(z)$, Kummer's hypergeometric function $M(a;c;z) = \llim_{b\to \infty} \ps{}{2}F_{1}^{}(a,b;c;z/b)$, etc., have integral representations in the form of a Mellin--Barnes type contour integral, each one giving rise to an analytic continuation in a complex domain included in (all or part of) the complex plane (see, e.g., \cite[\color{cyan}\S12.5(iii)]{AskDaal10}),\cite[\color{cyan}App.~D and F]{GoKiMaRo14},\cite[\color{cyan}\S1.6]{KiSrTr06}, etc.). 

Now, the following theorem deals with the hypergeometric function, in  which case the line integral is usually taken along a contour $\cC$ which is a deformation of the imaginary axis such that the integration contour separates all the poles of $\Gamma(a + s) \Gamma(b + s)$ from those of 
$\Gamma(-s)$, and $(-z)^s$ has its principal value.

\begin{thm}
Provided that $a$ is a positive integer, the Mellin--Barnes contour integral representation of Gau{\ss}'s hypergeometric function $\ps{}{2}F_1^{}(a,b;c;z)$ is given by
\begin{equation} \label{eq_mellinbarnes}
\frac{\Gamma(a) \Gamma(b)}{\Gamma(c)}\, \ps{}{2}F_{1}^{}(a,b ; c ; z) =  \frac{1}{2\pi \ri}\, %
\lint_{c - \ri \infty}^{c + \ri \infty} \frac{\Gamma(a + s) \Gamma(b + s)} {\Gamma(c + s)}\, %
\Gamma(-s)\ (-z)^s\rd s.
\end{equation}
Here, $|\arg(-z)| < \pi$ and the path of integration can pass to infinity parallel to the imaginary axis with any finite value of $c = \Re(s)$, on condition that the contour $\cC$ can be indented, if necessary, to separate all the poles of $\Gamma(-s)$ at the points $s = \nu$ ($\nu\in \N$) to the left from all the poles of $\Gamma(a + s)$ and $\Gamma(b + s)$ to the right, at the points $s = - a - \nu$ and
$s = - b - \nu$ ($\nu\in \N$), respectively. (Such a contour may always be drawn if a and b are non-negative integers or zero. Moreover, Eq.~\eqref{eq_mellinbarnes} remains valid also for negative or zero values of $a$.)
\end{thm}

\begin{proof}
This result can be obtained by moving the contour to the right while picking up the residues at $s =$ 0, 1, 2,\ldots. In establishing the result, it was necessary to suppose that $|z| < 1$. From the principle of analytic continuation, the integral~\eqref{eq_mellinbarnes} converges and defines an analytic function of $z$, which is holomorphic in the sector $|\arg(-z)| < \pi$. Hence, 
$\ps{}{2}F_{1}^{}(a,b ; c ; z)$ is defined not only inside the unit circle but also in the complex $s$-plane cut along the real $z$-axis from 0 to $\infty$, provided that the parameters are such that the integration path can be drawn to separate the poles of $\Gamma(-s)$ from the poles of $\Gamma(a + s)$ and $\Gamma(b + s)$.
\end{proof}

If the path of integration is displaced to the left we can proceed to evaluate the integral in~\eqref{eq_mellinbarnes} in a similar manner. In this case, such poles are encountered which correspond to two sequences of simple poles: the poles of $\Gamma(a + s)$ with residues
\begin{equation} \label{eq_res}
(-1)^n \frac{\Gamma(b - a - n) \Gamma(a + n)}{n! \Gamma(c - a - n)}\, (-z)^{-a-n} %
= \frac{\Gamma(a) \Gamma(b - a)}{\Gamma(c - a)}\, (-z)^{-a}\ \frac{(1 + a - c)_n (a)_n} %
{(1 + a - b)_n n!}\, (-z)^{-n}
\end{equation}
at the points $s = - a - \nu$, and the poles of $\Gamma(b + s)$ at $s = - b - \nu$, respectively; the related residues of which are obtained by interchanging $a$ and $b$ in Eq.~\eqref{eq_res}. (This is straightforward by symmetry in the series, $F(a,b;c;z) = F(b,a;c;z)$.) It is first assumed also that $a - b$ is not an integer, so that these poles are indeed simple poles~\cite[\color{cyan}Vol.~1, \S2.1.4]{ErMaObTr53}.

The use of analytic continuation enables to remove the restriction $|z| > 1$ and establish that the following Eq.~\eqref{eq_asympt} holds for $|\arg(-z)| < \pi$ and $a - b\notin \Z$.
\begin{flalign} \label{eq_asympt}
\ps{}{2}F_{1}^{}(a,b ; c ; z) =&{} \frac{\Gamma(c) \Gamma(b - a)}{\Gamma(b) \Gamma(c - a)}\, %
(-z)^{-a}\, \ps{}{2}F_{1}^{}\l(\sbs{a, 1 + a - c\\[.1cm] 1 + a - b} ; \tfrac{1}{z}\r)\nonumber\\
& \qquad +\, \frac{\Gamma(c) \Gamma(a - b)}{\Gamma(a) \Gamma(c - b)}\, (-z)^{-b}\, %
\ps{}{2}F_{1}^{}\l(\sbs{b, 1 + b - c\\[.1cm] 1 + b - a} ; \tfrac{1}{z}\r).
\end{flalign}
Now, when $|z| > 1$, the integral round the translated path vanishes as the path moves to infinity.
In that case, the asymptotic expansion of Gau{\ss}'s hypergeometric function for large $|z|$ results easily from formula~\eqref{eq_asympt} under the above conditions. Whence the asymptotic behaviour of 
$\ps{}{2}F_{1}^{}(a,b ; c ; z)$ as $|z|\to \infty$ in the sector $|\arg(-z)| < \pi$ given by
\begin{flalign} \label{asympthyp}
\ps{}{2}F_{1}^{}(a,b ; c ; z) &= A (-z)^{-a} \bgl(1 + \cO\l(z^{-1}\r) \bgr) %
+ B (-z)^{-b} \bgl(1 + \cO\l(z^{-1}\r) \bgr)\ \qquad \tif\ \  a - b\notin \Z\nonumber\\
\shortintertext{%
and by
}
\ps{}{2}F_{1}^{}\l(a,b;c;z\r) &= C (-z)^{-a}\log(-z) \bgl(1 + \cO\l(z^{-1}\r) \bgr)\ %
\qquad \tif\ \ a - b\in \Z,
\end{flalign}
where $A, B, C\in\C$ are constants. More precisely, $A = \frac{\Gamma(c) \Gamma(b - a)} %
{\Gamma(b) \Gamma(c - a)}$ and $B = \frac{\Gamma(c) \Gamma(a - b)}{\Gamma(a) \Gamma(c - b)}$ by Eq.~\eqref{eq_asympt}, with commuting parameters $a$ and $b$~\cite[\color{cyan}\S\S2.1.4 and 2.3.1]{ErMaObTr53}.

When $a - b$ is an integer, some of the poles become double poles and then, the residues involve terms in $\ln(-z)$ as may be checked in Eq.~\eqref{asympthyp}. For instance, if $b - a$ is a non-negative integer,
\begin{multline*}
\ps{}{2}F_{1}^{}\l(\sbs{a, a + m\\[.1cm] c} ; z\r) = (-z)^{-a}\, \frac{\Gamma(c)}{\Gamma(a+m)} %
\lsum_{n=0}^{m-1} \frac{(a)_{n}(m-n-1)!}{k!\Gamma(c-a-n)}\, z^{-n}\\
+ (-z)^{-a}\, \frac{\Gamma(c)}{\Gamma(a)} \lsum_{n=0}^\infty \frac{(a+m)_{n}}{n!(n+m)! %
\Gamma(c-a-n-m)}\, (-1)^{n} z^{-n-m}\\
\times\, \bgl(\ln(-z) + \psi(n+1) + \psi(n+m+1)  - \psi(a+n+m) - \psi(c-a-n-m)\bgr),
\end{multline*}
where $\psi(z) = \Gamma'(z)/\Gamma(z)$ is the usual psi (or digamma) function, $|z| > 1$ and $|\arg(-z)| < \pi$. (see Askey {\em et al.}~\cite[\color{cyan}Chap.~15, \S15.8(ii)]{AskDaal10} and Erd\'elyi {\em et al.}~\cite[\color{cyan}Vol.~1, \S\S2.1.4 and 2.3.1]{ErMaObTr53} for exhaustive proofs).

Given proper conditions of convergence, one can relate more general Mellin--Barnes type contour integrals to the generalized hypergeometric function $\ps{}{p}F_{q}^{}(z)$ in a similar way. Some special function related to $\ps{}{p}F_{q}^{}(z)$ include the dilogarithm Li$_2(z) = \Sum_{n\ge 0} n^{-2} z^n = z\ \ps{}{3}F_{2}^{}(1,1,1 ; 2,2 ; z)$, Hahn and Wilson polynomials, etc.

\addcontentsline{toc}{section}{References}
\section*{References}

\vskip -1cm
\def\refname{\empty}

\bibliographystyle{article}
\def\bibfmta#1#2#3#4{ {\sc #1}, {#2}, \emph{#3}, #4.}
\bibliographystyle{book}

\begin{thebibliography}{99}\setlength{\itemsep}{-.3mm}
\bibliographystyle{plain}
%
\bibitem{Agar53}\bibfmta
{Agarwal Rana P.}
{\`A propos d'une note de M. Pierre Humbert}
{C. R. Acad. Sci. Paris}{{\bf 236} (1953), 2031--2032}
%
\bibitem{AnnMan12}\bibfmtb
{Annaby Mahmoud H., Mansour Zeinab S.}
{$q$-Fractional Calculus and Equations}{Lect. Note in Maths. {\bf 2056}}{Springer, 2012}
%
\bibitem{AskDaal10}\bibfmtb
{Askey Richard, Daalhuis Olde A.B.} 
{NIST Handbook of Mathematical Functions}{Chapters~13, 15 \& 16}
{Eds. Olver, Lozier, Boisvert, Clark, Cambridge Un. Press, 2010}
%
\bibitem{BanPra16}\bibfmtb
{Bansal Deepak, Prajapat Jugal Kishore}
{Certain geometric properties of the Mittag-Leffler functions}
{Complex Var. Elliptic Equ.}{{\bf 61}:3 (2016), 338--350.}
%
\bibitem{BeLeNaTo16}\bibfmta
{Bergounioux Maïtine, Leaci Antonio, Nardi Giuseppe, Tomarelli Franco}
{Fractional Sobolev Spaces and Bounded Variation Functions}
{Preprint}{Elsevier (2016), 24pp}
%
\bibitem{DiaPar07}\bibfmta
{D{\'i}az Rafael, Pariguan Eddy}
{On hypergeometric functions and $k$-Pochhammer symbol}
{Divulg. Mat.}{{\bf 15}:2 (2007) 179--192}
%
\bibitem{Dzrba52}\bibfmta
{D{\v{z}}rba{\v{s}}jan Mkhitar M.}
{On the integral representation and uniqueness of some classes of entire functions}
{Doklady Akad. Nauk SSSR (N.S.)}{{\bf 85} (1952) 29--32 (in Russian)}
%
\bibitem{Dzrba60}\bibfmta
{D{\v{z}}rba{\v{s}}jan Mkhitar M.}
{On the integral transformations generated by generalized functions 
of the Mittag-Leffler type}{Izv. Akad. Nauk Armjan. SSR Ser. Fiz.-Mat. Nauk}
{{\bf 13}:3 (1960), 21--63 (in Russian)}
%
\bibitem{Dzrba66}\bibfmtb
{D{\v{z}}rba{\v{s}}jan Mkhitar M.}
{Integral transforms and representations of functions in the complex domain}
{``Nauka'', Moscow}{671 pp., 1966 (in Russian)}
%
\bibitem{ErMaObTr53}\bibfmtb
{Erd\'elyi Arthur, Magnus Wilhelm, Oberhettinger Fritz, Tricomi Francesco G.}
{Higher Transcendental Functions}{vol.~I--III}{McGraw-Hill, 1953--1955}
%
\bibitem{Fox28}\bibfmta
{Fox Charles}
{The asymptotic expansion of generalized hypergeometric functions}
{Proc. Lond. Math. Soc.}{{\bf S2–27}:1 (1928), 389--400}
%
\bibitem{Gerhold12}\bibfmta
{Gerhold Stephan}
{Asymptotics for a variant of the Mittag–Leffler function}
{Integr. Transf. Spec. Funct.}{{\bf 23}:6 (2012), 397--403}
%
\bibitem{GoKiMaRo14}\bibfmtb
{Gorenflo Rudolph, Kilbas Anatoly A., Mainardi Francesco, Rogosin Sergei V.}
{Mittag-Leffler Functions, Related Topics and Applications}
{Springer Monogr. in Math.}{Springer-Verlag, 2014}
%
\bibitem{GoKiRo98}\bibfmta
{Gorenflo Rudolph,  Kilbas Anatoly A., Rosogin Sergei V.}
{On the generalized Mittag-Leffler type functions}
{Integr. Transf. Spec. Funct.}{{\bf 7}:3--4 (1998), 215--224}
%
\bibitem{HaMaSa11}\bibfmta
{Haubold Hans J., Mathai, Arakaparampil M., Saxena Ram Kishore}
{Mittag-Leffler Functions and Their Applications}{J. Appl. Math.}
{{\bf 2011} (2011), Art. ID 298628}
%
\bibitem{HumAgr53}\bibfmta
{Humbert Pierre, Agarwal Rana P.}
{Sur la fonction de Mittag-Leffler et quelques-unes de ses 
g\'en\'eralisations}{Bull Sci. Math.}{{\bf 77}:2 (1953), 180--185}
%
\bibitem{Kilbas05}\bibfmta
{Kilbas Anatoly A.}
{Fractional calculus of the generalized Wright function}
{Fract. Calc. Appl. Anal}{{\bf 8}:2 (2005), 113--126}
%
\bibitem{KilSai95}\bibfmta
{Kilbas Anatoly A., Saigo Megumi}
{Solution of Abel integral equations of the second kind 
and of differential equations of fractional order}
{Integr. Transforms Spec. Funct.}{{\bf 95}:5 (1995), 29--34}
%
\bibitem{KiSaSa04}\bibfmta
{Kilbas Anatoly A., Saigo Megumi, Saxena Ram Kishore}
{Generalized Mittag-Leffler functions and generalized fractional 
calculus operator}{Dokl. Akad. Nauk Belarusi}{{\bf 15}:1 (2004), 31--49}
%
\bibitem{KiSaTr02}\bibfmta
{Kilbas Anatoly A., Saigo Megumi, Trujillo Juan J.}
{On the generalized Wright function}
{Fract. Calc. Appl. Anal.}{{\bf 5}:4 (2002), 37--60}
%
\bibitem{KiSrTr06}\bibfmtb
{Kilbas Anatoly A., Srivastava Hari Mohan, Trujillo Juan J.}
{Theory and Applications of Fractional Differential Equations}
{North Holland Math. Studies {\bf 204}}{Elsevier, 2006}
%
\bibitem{Kirya94}\bibfmtb
{Kiryakova Virginia S.}
{Generalized Fractional Calculus and Applications}
{Pitman Res. Notes Math. Ser. {\bf 301}}{Longman Scientific \& Technical, 1994}
%
\bibitem{Kirya06}\bibfmta
{Kiryakova Virginia S.}
{On two Saigo's fractional integral operators in the class of univalent functions}
{Fract. Calc. Appl. Anal.}{{\bf 9}:2 (2006), 159--176}
%
\bibitem{Kirya10a}\bibfmta
{Kiryakova Virginia S.}
{The special functions of fractional calculus as generalized fractional calculus 
operators of some basic functions}{Comput. Math. Appl.}{{\bf 59}:5 (2010), 1128--1141}
%
\bibitem{Kirya10b}\bibfmta
{Kiryakova Virginia S.}
{The multi-index Mittag-Leffler functions as important class of special function 
of fractional calculus}{Comput. Math. Appl.}{{\bf 59}:5 (2010), 1885--1895}
%
\bibitem{Kirya15}\bibfmta
{Kiryakova Virginia S.}
{The role of special functions and generalized fractional calculus 
in studying classes of univalent functions}
{Adv. Math. Sci. Journal}{{\bf 4}:2 (2015), 161--174} 
%
\bibitem{Kumar16}\bibfmta
{Kumar Dinesh}
{On certain fractional calculus operators involving generalized Mittag-Leffler 
functions}{Sahand Comm. Math. Anal.}{{\bf 3}:2 (2016), 33--45}
%
\bibitem{KumSax15}\bibfmta
{Kumar Dinesh, Saxena Ram Kishore}
{Generalized Fractional Calculus of the $M$-Series Involving $F_3$ Hypergeometric function}
{Sohag J. Math.}{{\bf 2}:1 (2015), 17--22} 
%
\bibitem{Levin96}\bibfmtb
{Levin B. Yakovlevitch}
{Lectures on Entire Functions}
{Transl. Math. Monogr.}{Vol.~150, AMS, 1996}
%
\bibitem{MaSaHa10}\bibfmtb
{Mathai Arakaparampil M., Saxena Ram Kishore, Haubold Hans J.}
{On the $H$-function with Applications}
{in ``The $H$-Function'', Chap.~1, 1--43}{Springer, 2010}
%
\bibitem{Mittag03}\bibfmta
{Mittag-Leffler G\"osta Magnus}
{Sur la nouvelle fonction $E_\alpha(z)$}{C. R. Acad. Sci. Paris (S\'erie~2)}
{{\bf 137} (1903), 554--558}
%
\bibitem{Mittag05}\bibfmta
{Mittag-Leffler G\"osta Magnus}
{Sur la repr\'esentation analytique d'une branche uniforme de fonction monog{\`e}ne}
{Acta. Math.}{{\bf 29} (1905), 1069--182}
%
\bibitem{Paneva12}\bibfmta
{Paneva-Konovska Jordanka}
{Inequalities and Asymptotic Formulae for the Three Parametric Mittag-Leffler Functions}
{Math. Balk.}{{\bf 26}:1--2 (2012), 203--210}

\bibitem{Parmar15}\bibfmta
{Parmar Rakesh K}
{A Class of Extenced Mittag-Leffler Functions and their Properties 
Related to Integral Transforms and Fractional Calculus}
{Mathematics}{{\bf 29}:3 (2015), 1069--1082}
%
\bibitem{Prabha71}\bibfmta
{Prabhakar Tilak Raj}
{A singular integral equation with a generalized Mittag-Leffler function in the Kernel}
{Yokohama Math. J.}{{\bf 19} (1971), 7--15}
%
\bibitem{Rainv60}\bibfmtb
{Rainville Earl D}
{Special Functions}{Macmillan Company}{1960}
%
\bibitem{Saigo78}\bibfmta
{Saigo Megumi}
{A remark on integral operators involving the Gau{\ss}\ hypergeometric functions}
{Math. Rep. College General Ed. Kyushu Univ.}{{\bf 11}:2 (1978), 135--143}
%
\bibitem{Saigo85}\bibfmta
{Saigo Megumi}
{A generalization of fractional calculus}
{in Proc. Internat. Workshop on Fractional Calculus}{Ross Priory, Univ. of
Strathclyde, Pitman (1985), 188--198}
%
\bibitem{Saigo96}\bibfmta
{Saigo Megumi}
{On generalized fractional calculus operators}
{In: Recent Advances in Applied Mathematics (Proc. Internat. Workshop}
{Kuwait Univ. (1996), 441--450}
%
\bibitem{SaiMae96}\bibfmta
{Saigo Megumi, Maeda N.}
{More generalization of fractional calculus}{Transform Methods and Special Functions}
{Varna, Bulgaria (1996), 386--400}
%
\bibitem{Salim09}\bibfmta
{Salim Tariq}
{Some properties relating to the generalized Mitta-Leffler function}
{Adv. Appl. Math. Anal.}{{\bf 4} (2009), 21--30}
%
\bibitem{SalFar12}\bibfmta
{Salim Tariq, Faraj Ahmad}
{A generalization of Mittag-Leffler function and Integral operator associated 
with fractional calculus}{J. of Frac. Calc. and Appl.}{{\bf 3}:(5) (2012), 1--13}
%
\bibitem{SamKilMar93}\bibfmtb
{Samko Stephan G., Kilbas Anatoly A., Marichev Oleg I.}
{Fractional Integrals and Derivatives. Theory and Applications}
{Gordon and Breach Science Publisher}{1993}
%
\bibitem{SaxSai05}\bibfmta
{Saxena Ram Kishore, Saigo Megumi}
{Certain properties of the fractional calculus operators associated 
with generalized Mittag-Leffler function}{Fract. Calc. Appl. Anal.}
{{\bf 8}:2 (2005), 141--154}
%
\bibitem{Shakeel14}\bibfmta
{Shakeel Ahmed}
{On the generalized fractional integrals of the generalized 
Mittag-Leffler function}{SpringerPlus}{{\bf 3} (2014), 198}
%
\bibitem{Sharma08}\bibfmta
{Sharma Manoj}
{Fractional  integration and fractional differentiation of the $M$-series}
{Frac. Calc. Appl. Anal.}{{\bf 11}:2 (2008), 187--191}
%
\bibitem{Sharma12}\bibfmta
{Sharma Kishan}
{An Introduction to the Generalized Fractional Integration}
{Bol. Soc. Paran. Math.}{{\bf 30}:2 (2012), 85--90}
%
\bibitem{Srivast14}\bibfmtb
{Srivastava Hari M.}
{Special Functions in Fractional Calculus and Related Fractional Differintegral Equations}
{World Scientific}{2014}
%
\bibitem{Srivast16}\bibfmta
{Srivastava Hari M.}
{Some Families of Mittag-Leffler Type Functions and Associated Operators 
of Fractional Calculus (Survey)}{J. Pure Appl. Math.}{{\bf 7}:2 (2016), 123--145}
%
\bibitem{SrivastAgarwal13}\bibfmta
{Srivastava Hari M., Agarwal Praveen}
{Certain Fractional Integral Operators and the Generalized Incomplete 
Hypergeometric Functions}{Appl. Appl. Math.}{{\bf 8}:2 (2013), 333--345}
%
\bibitem{SrivastManocha84}\bibfmtb
{Srivastava Hari M., Manocha H.L.}
{A treatise on generating functions}
{Horwood Series: Mathematics and its Applications}{John Wiley \& Sons, 1984}
%
\bibitem{Temme96}\bibfmtb
{Temme Nico M.}
{Special Functions. An Intoduction to the classical Functions of Mathematical Physics}
{John Wiley \& Sons}{1996}
%
\bibitem{Wiman05}\bibfmta
{Wiman Anders}
{\"Uber den Fundamentalsatz in der Theorie der Funktionen $E^a(z)$}
{Acta Math.}{{\bf 29}:1 (1905), 191--201}
%
\bibitem{Wright34}\bibfmta
{Wright E. Maitland}
{The asymptotic expansion of the generalized Bessel function}
{J. London Math. Soc.}{{\bf 38}:2 (1934), 257--270}

\end{thebibliography}
\def\bibfmtb#1#2#3#4{ {\sc #1}, \emph{#2}, {#3}, #4.}

%\newpage
\vskip 1cm
\no\rule{\textwidth}{.5pt}
\vskip .3cm
\no {\small Christian {\sc Lavault}
\newline \emph{E-mail:} \href{mailto:lavault@lipn.univ-paris13.fr}{lavault@lipn.univ-paris13.fr}
\newline LIPN, UMR CNRS 7030 -- Laboratoire d'Informatique de Paris-Nord
\newline Universit\'e Paris 13, Sorbonne Paris Cit\'e, F-93430 Villetaneuse.
\newline \emph{URL:} \url{http://lipn.univ-paris13.fr/\textasciitilde lavault}
}

\end{document}